\documentclass[a4paper,11pt,leqno]{article}
\usepackage{amssymb, amsmath, amsthm, graphicx}

\newtheorem{thm}{Theorem}[section]
\newtheorem{lem}[thm]{Lemma}
\newtheorem{cor}[thm]{Corollary}
\newtheorem{prop}[thm]{Proposition}

\newtheorem{rem}{Remark}[section]

\numberwithin{equation}{section}

\renewcommand{\a}{\alpha}
\renewcommand{\b}{\beta}
\newcommand{\e}{\varepsilon}
\newcommand{\de}{\delta}
\newcommand{\fa}{\varphi}
\newcommand{\ga}{\gamma}

\newcommand{\la}{\lambda}
\renewcommand{\th}{\theta}
\newcommand{\si}{\sigma}

\newcommand{\om}{\omega}
\newcommand{\De}{\Delta}
\newcommand{\Ga}{\Gamma}

\newcommand{\Om}{\Omega}
\newcommand{\lan}{\langle}
\newcommand{\ran}{\rangle}

\newcommand{\dist}{\operatorname{dist}\,}

\def\R{{\mathbb{R}}}
\def\N{{\mathbb{N}}}
\def\Z{{\mathbb{Z}}}
\def\T{{\mathbb{T}}}

\def\SS{{\mathbb{S}}}
\allowdisplaybreaks

\def\uu{\underline{u}}
\def\ux{\underline{x}}

\title{KPZ equation, its renormalization and invariant measures}
\author{Tadahisa Funaki$^{1)}$ and Jeremy Quastel$^{2)}$}
\date{\today}
\begin{document}
\maketitle

\begin{abstract}
\noindent
The Kardar-Parisi-Zhang (KPZ) equation is a 
stochastic partial differential equation which is  ill-posed
because  the nonlinearity is marginally defined with respect to the 
roughness of the forcing noise.  However, its Cole-Hopf solution, defined as
the logarithm of the solution of the linear stochastic heat equation 
(SHE) with a multiplicative noise, is a mathematically well-defined
object.  In fact, Hairer \cite{Hai} has recently proved that the solution
of SHE can actually be derived through the Cole-Hopf transform of the 
solution of the KPZ equation with a suitable renormalization  under 
periodic boundary conditions.  This transformation is unfortunately not well adapted to studying the invariant measures of these Markov processes.

The present paper introduces a different type of regularization 
for the KPZ equation on the whole line $\R$ or under periodic
boundary conditions, which is appropriate 
from the viewpoint of studying the invariant measures.
The Cole-Hopf transform applied to this equation leads to 
an SHE with a smeared noise having an extra complicated nonlinear term.  Under time
average and in the stationary regime, it is shown that this term can be replaced by a simple linear term, so that the limit equation
is the linear SHE with an extra linear term with coefficient $\tfrac1{24}$.
The methods are essentially stochastic analytic: The
Wiener-It\^o expansion and a similar method for establishing the  Boltzmann-Gibbs principle are used.
As a result, it is shown that the distribution of a two-sided geometric 
Brownian motion with a height shift given by 
Lebesgue measure is invariant under the evolution determined by
the SHE on $\R$.
\footnote{
\hskip -6mm
$^{1)}$ Graduate School of Mathematical Sciences,
The University of Tokyo, Komaba, Tokyo 153-8914, Japan.
e-mail: funaki@ms.u-tokyo.ac.jp \\
$^{2)}$ Department of Mathematics, University of Toronto, 40 St.\
George Street, Toronto, Ontario, Canada M5S 2E4.
e-mail: quastel@math.toronto.edu}
\footnote{
\hskip -6mm
\textit{Abbreviated title $($running head$)$: KPZ equation, renormalization and
invariant measures.}}
\footnote{
\hskip -6mm
\textit{
Author$^{1)}$ was supported in part by the JSPS Grants $($A$)$ 22244007,
$($B$)$ 26287014 and 26610019.  Author$^{2)}$ was supported by Natural 
Sciences and Engineering Research Council of Canada, a Killam Fellowship,
and the Institute for Advanced Study.
}
}
\end{abstract}

\section{Introduction and main results}  \label{section:1}

The Kardar-Parisi-Zhang (KPZ) equation 
is the stochastic partial differential equation (SPDE) \begin{equation} \label{eq:KPZ}
\partial_t h = \frac12 \partial_x^2 h + \frac12 \big(\partial_x h \big)^2
+ \dot{W}(t,x), \quad x\in\R,
\end{equation}
where $\dot{W}(t,x)$ is 
the space-time Gaussian white noise, in particular, it has 
covariance
$$
E[\dot{W}(t,x)\dot{W}(s,y)] = \de(x-y)\de(t-s).
$$
We consider \eqref{eq:KPZ} in one dimension on the whole line,
or on finite intervals with periodic boundary conditions.
This SPDE is used in the physics literature as a general model for the 
fluctuations of a growing interface and $h=h(t,x)\in \R$ describes the 
height of the interface at time $t$ and position $x\in \R$.
The coefficients $\tfrac12$ are not essential, since one can change
them under scalings in time, position and values of $h$.  The importance of the 
KPZ equation comes from the fact that it reflects the behavior of a wide class of microscopic systems.
Unfortunately, the equation \eqref{eq:KPZ} is ill-posed and does not 
make sense as written.  Indeed, the linear SPDE obtained from
\eqref{eq:KPZ} by dropping the nonlinear term $\tfrac12 (\partial_x h)^2$
has a solution $h(t,x)$ which is $(\tfrac12-\e)$-H\"older continuous 
for every $\e>0$ in the spatial variable $x\in \R$, and this suggests
that the term $\tfrac12 (\partial_x h)^2$ would diverge in the equation \eqref{eq:KPZ}.
To compensate for this diverging factor, one needs to 
introduce a renormalization and the correct form of the KPZ
equation would be
\begin{equation} \label{eq:KPZ-ren}
\partial_t h = \frac12 \partial_x^2 h + \frac12 \left(\big(\partial_x h \big)^2
-\de_x(x)\right) + \dot{W}(t,x), \quad x\in\R.
\end{equation}
The delta function $\de_x$ evaluated at $x$ is certainly $+\infty$.
At least heuristically, one can derive \eqref{eq:KPZ-ren} from
the well-defined linear stochastic heat equation \eqref{eq:1.1} explained below
by applying Cole-Hopf transform and It\^o's formula and 
$-\tfrac12 \de_x(x)$ appears as an It\^o correction term; see 
\eqref{eq:1-aa} and \eqref{eq:2-a} below, in which 
$\xi^\e \to \de_0(0)=\de_x(x)$, since $\eta_2^\e(x)\to\de_0(x)$,
and $\dot{W}^\e\to\dot{W}$ as $\e\downarrow 0$ at least
formally.  The result is called the \emph{Cole-Hopf solution}.  Recently, Hairer \cite{Hai} has shown  how to 
make sense of \eqref{eq:KPZ-ren} (with periodic boundary conditions) by introducing a renormalization structure
without passing through the Cole-Hopf transform.  The resulting (unique) solutions 
are indeed the Cole-Hopf solutions \eqref{eq:1.6}.  
In this sense, and also because they arise in the weakly asymmetric limit of exclusion models \cite{BG}, the Cole-Hopf solution is 
the physically correct solution of the KPZ equation.

To explain the Cole-Hopf solution of the KPZ equation, let us
consider the following one dimensional linear stochastic heat 
equation (SHE) for $Z=Z(t,x)$ with a multiplicative noise:
\begin{equation} \label{eq:1.1}
\partial_t Z = \frac12 \partial_x^2 Z + Z \dot{W}(t,x), \quad x\in\R,
\end{equation}
having an initial value $Z(0,x)\ge 0$.  Unlike the KPZ equation \eqref{eq:KPZ}, this is a well-posed SPDE.  In fact, the solution of the SPDE
\eqref{eq:1.1} is defined in a weak sense, that is, $Z(t) = \{Z(t,x);x\in\R\}$
is called a {\it solution in generalized functions' sense} if it is adapted
to the increasing family
of $\si$-fields generated by $\{\dot{W}(s,x); 0\le s \le t, x\in \R\}$ and satisfies 
\begin{equation} \label{eq:1.2-a}
\lan Z(t),\fa\ran = \lan Z(0),\fa\ran
+ \frac12 \int_0^t \lan Z(s),\partial_x^2\fa\ran ds
+ \int_0^t \int_\R Z(s,x)\fa(x) W(dsdx),
\end{equation}
for every $\fa\in C_0^\infty(\R)$, where $\lan Z,\fa\ran = \int_\R Z(x)
\fa(x)dx$ and the last term is defined as the It\^o integral with
respect to the space-time Gaussian white noise.  The process
$Z(t)$ is called a {\it mild solution} if it is adapted and satisfies 
the stochastic integral equation
\begin{equation} \label{eq:1.3-a}
Z(t,x) = \int_\R p(t,x,y) Z(0,y)dy + 
\int_0^t\int_\R p(t-s,x,y) Z(s,y)W(dsdy),
\end{equation}
where $p(t,x,y) = \frac1{\sqrt{2\pi t}} e^{-(y-x)^2/2t}$ is the
heat kernel on $\R$.  These two
notions of solutions are equivalent in the class 
$\mathcal{C}_{\rm tem}$, which consists of all $Z\in 
\mathcal{C}= C(\R,\R)$ satisfying
$$
\|Z\|_{C_r} = \sup_{x\in\R} e^{-r |x|} |Z(x)| < \infty,
$$
for all $r >0$, equipped with the topology induced by norms
$\{\| \cdot \|_{C_r}; r>0\}$.  It is known that, if 
$Z(0)\in \mathcal{C}_{\rm tem}$, the SPDE \eqref{eq:1.1} has
a unique solution, in both generalized functions' and mild
senses, such that $Z(\cdot) \in C([0,\infty),
\mathcal{C}_{\rm tem})$ a.s. Moreover, the strong comparison theorem 
holds, that is, if $Z(0)$ satisfies in addition that
$Z(0) \in \bar{\mathcal{C}}_+ = C(\R,[0,\infty))$ and $Z(0,x)>0$ for
some $x\in \R$, then $Z(t) \in C((0,\infty),\mathcal{C}_+)$ a.s., 
where $\mathcal{C}_+ =C(\R,\R_+), \R_+ = (0,\infty)$, 
equipped with the usual topology of uniform convergence on 
each bounded interval; see \cite{Mue} and Corollary 1.4
of \cite{Shiga}.

The {\it Cole-Hopf solution} of the KPZ equation is defined from
the solution of \eqref{eq:1.1} as
\begin{equation}  \label{eq:1.6}
h(t,x) := \log Z(t,x),
\end{equation}
which is well-defined since $Z(t,x)>0$.
As we mentioned above, in order to link the Cole-Hopf solution
to the KPZ equation, we need to deal with an infinite It\^o correction
term.  In other words, a certain renormalization factor which balances 
with this diverging term should be introduced in the KPZ equation.

The simplest approximation scheme is to consider
\begin{equation}  \label{eq:1-aa}
\partial_t h = \frac12 \partial_x^2 h + \frac12 \big((\partial_x h)^2
 -\xi^\e \big)+ \dot{W}^\e(t,x), \quad x \in \R,
\end{equation}
where $\dot{W}^\e(t,x)= \big(\dot{W}(t)*\eta^\e\big)(x)$ is a smeared noise
defined from a usual symmetric convolution kernel $\eta^\e$ which tends to 
the $\de$-function $\de_0$ as $\e\downarrow 0$ and 
$\xi^\e=\eta_2^\e(0)$ with  $\eta_2^\e=\eta^\e*\eta^\e$;
see Section \ref{sec:2.1} for more details.  
Under the Cole-Hopf transform \eqref{eq:1.6} or $Z(t,x) :=
e^{h(t,x)}$, by applying It\^o's formula, 
this is equivalent to the SPDE
\begin{equation}  \label{eq:2-a}
\partial_t Z = \frac12 \partial_x^2 Z + Z \dot{W}^\e(t,x),
\quad x \in \R,
\end{equation}
see \cite{BG}, (3.6).  It is easy to see that the solution of \eqref{eq:2-a} 
converges to that of \eqref{eq:1.1} as $\e\downarrow 0$, and therefore 
the solution of \eqref{eq:1-aa} converges to the Cole-Hopf solution of 
the KPZ equation.  However, from the point of view of invariant measures,
\eqref{eq:1-aa} is not a good approximation; in fact it is an open problem to describe
the invariant measures of \eqref{eq:1-aa}.

The present paper introduces a renormalization different from 
\eqref{eq:1-aa} or the one considered in \cite{Hai}, better adapted to finding the invariant measures.  
We consider the following KPZ approximating equation:
\begin{equation}  \label{eq:1}
\partial_t h = \frac12 \partial_x^2 h + \frac12 \big((\partial_x h)^2
 -\xi^\e \big)* \eta_2^\e + \dot{W}^\e(t,x), \quad x \in \R.
\end{equation}
It is a common fact (or folklore) that the invariant measures are essentially
unchanged if we apply an operator $A$ (in our case the convolution with $\eta^\e$)
to the noise term and apply $A$ twice to the drift term at the same time;
see e.g., \cite{HH}.  Here, the convolution commutes with the
second derivatives so that we don't put it in the first term.
Then, the Cole-Hopf transform $Z^\e(t,x)
= e^{h^\e(t,x)}$ applied to the solution $h=h^\e(t,x)$ of \eqref{eq:1} leads to 
an SHE with a smeared noise and an extra complex nonlinear term involving 
a certain renormalization structure:
\begin{equation}  \label{eq:8}
\partial_t Z = \frac12 \partial_x^2 Z + \frac12 Z\left\{\left(
\frac{\partial_x Z}Z \right)^2*\eta_2^\e - \left(\frac{\partial_x Z}Z \right)^2
\right\} + Z \dot{W}^\e(t,x), \quad x\in \R.
\end{equation}

One of the main contributions of this paper is to show that this nonlinear term,
that is the middle term in the right hand side of \eqref{eq:8},  can be replaced by a simple 
linear term divided by a specific constant $24$ in the limit when the corresponding
tilt process is in equilibrium; see Theorems \ref{thm:3.1} and \ref{thm:3.1-M}
below.  Thus, we  derive the SPDE
\begin{equation}  \label{eq:8-b}
\partial_t Z = \frac12 \partial_x^2 Z + \frac1{24} Z+ Z \dot{W}(t,x), \quad x\in \R,
\end{equation}
in the limit as $\e\downarrow 0$.  Or, we can rephrase it, that the solution
$h^\e(t,x)$ of \eqref{eq:1} converges to $h(t,x)+\tfrac1{24}t$, where
$h(t,x)$ is the Cole-Hopf solution defined by \eqref{eq:1.6}.
The constant $\tfrac1{24}$
frequently appears in KPZ related papers and describes the speed of
the vertical drift of the interface.  The same constant
also appears in our formulation and this provides it with a probabilistic
meaning.  If the convolution kernel $\eta^\e$ is asymmetric and satisfies a certain
condition, a constant different from $\frac1{24}$
appears in the limit; see Remark \ref{rem:3.3}-(2).

For technical reasons (see Lemma \ref{cor:3.10} and Proposition
\ref{prop:3.12}), in order  to study the limit $\e\downarrow 0$ , we need to work with the KPZ 
approximating equation \eqref{eq:1}
on finite intervals with periodic boundary conditions:
\begin{equation}  \label{eq:1-M}
\partial_t h = \frac12 \partial_x^2 h + \frac12 \big((\partial_x h)^2
 -\xi^\e \big)* \eta_2^\e + \dot{W}^\e(t,x), \quad x \in \SS_M,
\end{equation}
where $\SS_M=\R/M\Z (=[0,M))$, $M\ge 1$, is a continuous one-dimensional
torus of size $M$.  The convolution $*\eta_2^\e$ is defined in a
periodic sense.  A similar SPDE was studied in Da Prato et.\ al.\ \cite{DDT}.
The Cole-Hopf transformed process $Z^{\e,M}(t,x)
= e^{h^{\e,M}(t,x)}, x\in\SS_M$ applied to the solution $h=h^{\e,M}(t,x)$ of \eqref{eq:1-M} 
satisfies the SPDE:
\begin{equation}  \label{eq:8-M}
\partial_t Z = \frac12 \partial_x^2 Z + \frac12 Z\left\{\left(
\frac{\partial_x Z}Z \right)^2*\eta_2^\e - \left(\frac{\partial_x Z}Z \right)^2
\right\} + Z \dot{W}^\e(t,x), \quad x\in \SS_M.
\end{equation}
We also consider the SPDE
\begin{equation}  \label{eq:8-b-M}
\partial_t Z = \frac12 \partial_x^2 Z + \frac1{24} Z+ Z \dot{W}(t,x), \quad x\in \SS_M,
\end{equation}
which will appear in the limit as $\e\downarrow 0$. 

To state our first result, we introduce some notation.
Let $B^M=\{B^M(x);x\in \SS_M\}$ be the pinned Brownian motion
satisfying $B^M(0)=B^M(M)=0$.  Let $\nu^M$ and $\nu^{\e,M}$ be the distributions
of $B^M$ and $\{B^M*\eta^\e(x) - B^M*\eta^\e(0);x\in\SS_M\}$ with the convolution
defined in a periodic sense on the space $\mathcal{C}_M=C(\SS_M)
(=C(\SS_M,\R))$, respectively.
For a random continuous function $h=\{h(x);x\in\SS_M\}$ on $\SS_M$, we call
$\nabla h=\{h(x)-h(y); x,y\in\R\}$ for $h$ periodically extended on $\R$
 the tilt variables of $h$ and denote
$\nabla h\stackrel{\rm{law}}{=} \nu^{\e,M}$ (or $\nu^M$) if the law of 
$\{h(x)-h(0);x\in\SS_M\}$ is given by $\nu^{\e,M}$ (or $\nu^M$).  
In fact, $\nu^{\e,M}$ is invariant for the tilt process determined from
the SPDE \eqref{eq:1-M}; see Theorem \ref{thm:2}-(1).

Our convergence result on $\SS_M$ is now formulated as follows.

\begin{thm} \label{thm:0}
We fix $M\ge 1$ and assume that the law of $Z^{\e,M}(0,\cdot) = e^{h^{\e,M}(0,\cdot)}$ 
is determined by $(h^{\e,M}(0,0), \nabla h^{\e,M}(0)) \stackrel{\rm{law}}{=} 
\de_{h_0}\otimes\nu^{\e,M}$ with some $h_0\in \R$.  Then, for every $t\ge 0$, 
the law of the solution $Z^{\e,M}(t,\cdot)$ of the SPDE \eqref{eq:8-M} on 
the space $\mathcal{C}_M$ weakly converges as $\e\downarrow 0$ to that
of the solution $Z^M(t,\cdot)$ of the SPDE \eqref{eq:8-b-M} with the initial
distribution determined by $Z^M(0,\cdot) = e^{h^M(0,\cdot)}$ such that
$(h^M(0,0), \nabla h^M(0)) \stackrel{\rm{law}}{=} \de_{h_0}\otimes\nu^M$.
\end{thm}

This theorem immediately implies that the distribution of the tilt variables
of the logarithm of 
the solution $Z^M(t,\cdot)$ of the SPDE \eqref{eq:8-b-M} is given by $\nu^M$
for every $t\ge 0$.

We next pass to the limit $M\to\infty$ by extending  $Z^M(t,x), x\in \SS_M$ 
and $\nu^M$ periodically on $\R$.
We need  some more notation.  Let 
$B=\{B(x);x\in \R\}$ be the two-sided Brownian motion satisfying $B(0)=0$,
that is,  $\{B(x);x\ge 0\}$ and $\{B(-x);x\ge 0\}$ are both Brownian motions
with time parameter $x\ge 0$ and mutually independent.  Let
$\nu$ and $\nu^\e$ be the distributions of $B$ and $\{B*\eta^\e(x)
- B*\eta^\e(0);x\in\R\}$ on the space $\mathcal{C}$, respectively.
The tilt variables $\nabla h$ of $h$ on $\R$ are similarly defined
as above and denote
$\nabla h\stackrel{\rm{law}}{=} \nu$ (or $\nu^\e$) if the law of 
$\{h(x)-h(0);x\in\R\}$ is given by $\nu$ (or $\nu^\e$).  We introduce
the weighted $L^2$-space $L_r^2(\R)$, $r>0$, which is a family of all 
measurable functions $u$ on $\R$ such that
\begin{equation}  \label{eq:L_r-norm}
\| u\|_{L_r^2} := \left(\int_\R u(x)^2 e ^{-2r\chi(x)} dx\right)^{1/2} < \infty,
\end{equation}
where $\chi\in C^\infty(\R)$ is a fixed function satisfying
$\chi(x)=|x|$ for $|x|\ge 1$.

Then, it is standard to show that $Z^M(t,x), x\in \R$ converges to
the solution of the SPDE \eqref{eq:8-b} weakly on $C([0,\infty), L_r^2(\R))$, $r>0$
and $\nu^M$ (periodically extended on $\R$)
converges to $\nu$ weakly on $L_r^2(\R)$, $r>1$ as $M\to\infty$;
see Proposition \ref{prop:SPDE-conv} and its consequences.

As a byproduct, though the factor $\tfrac1{24}Z$ is different in \eqref{eq:8-b},
we can investigate the invariant
measures of the SPDE \eqref{eq:1.1} on $\R$.  Let $\mu^c, c\in \R$ be
the distribution of $e^{B(x)+cx}, x\in\R$ on $\mathcal{C}_+$, where
$B(x)$ is the two-sided Brownian motion such that $\mu^c(B(0)\in dx) =dx$.
In particular, $\mu^c$ are not probability measures.  
Our second result is that $\mu^c$ are invariant for the process $Z(t)$,
which is a solution of the SPDE \eqref{eq:1.1}. 

\begin{thm} \label{thm:1.1}
For every bounded, integrable and continuous function $G$ on $\mathcal{C}_+$,
we have that
$$
\int_{\mathcal{C}_+} G(Z(t))d\mu^c = 
\int_{\mathcal{C}_+} G(Z(0))d\mu^c,
$$
for all $t\ge 0$ and $c\in \R$, where the integrals in both sides 
are defined under the condition that $Z(0)$ is distributed under
$\mu^c$.  More precisely, for example, the left hand side is given by
$$
\int_{\mathcal{C}_+} E_Z[G(Z(t))] \mu^c(dZ),
$$
where $E_Z[\cdot]$ stands for the expectation with respect to the
solution of the SPDE \eqref{eq:1.1} with an initial value $Z(0)=Z \in \mathcal{C}_+$.
\end{thm}

At  the level of the process $\partial_x \log Z(t,x)$, where the corresponding invariant 
measure is a white noise, the result was derived earlier by Bertini and Giacomin \cite{BG} 
via the weakly asymmetric limit of  simple exclusion processes.

Note that the measure $\mu^c$  is invariant, but not reversible
for the process $Z(t)$. Only a kind of Yaglom reversibility holds; see Remark
\ref{rem:Yaglom} below.

If $Z(t,x)$ is a solution of (\ref{eq:1.1}) and $c\in\mathbb{R}$
then $Z^c(t,x):=e^{cx + \tfrac12 c^2 t}Z(t,x+ct)$ is also a solution 
(with a new space-time Gaussian white noise $\dot{W}$).  Therefore, 
if one can show the invariance of $\mu^0$ for $Z(t)$, since $\mu^0$ is 
invariant under both shifts $Z(x)\mapsto Z(x+a)$ and $Z(x)\mapsto e^a Z(x)$, 
we have that the $\mu^c$ is also invariant for the process $Z(t)$.  For this
reason, in the proof of Theorem \ref{thm:1.1}, we may
assume $c=0$ without loss of generality and write $\mu$ for $\mu^0$.

We will mostly work with tilt variables $\partial_x h$ or $\partial_x \log Z$ 
rather than heights to avoid the difficulty caused by the non-normalizability
of the measure $\mu$.  This is carried
out by introducing an equivalence relation to the state space
$\mathcal{C}_+$ of $Z(t,x)$; see Section \ref{section:3.1}.  
Or, one can say that we are only
interested in the shapes of height functions
$h(t,x)=\log Z(t,x)$ by identifying its vertical translations:
$h(t,x) \sim h(t,x)+c$ for all $c\in \R$; see Remark \ref{rem:2.1-c}
below.

\begin{rem}  \label{rem:1.1}
One expects $\mu^c$,  $c\in\mathbb{R}$ to be all the extremal
invariant measures (except constant multipliers) for the process
$Z(t)$ as in \cite{FS}, but we will not investigate this here.  \end{rem}

The paper is organized as follows.  In Section 2, we consider
the KPZ approximating equations \eqref{eq:1} and \eqref{eq:1-M},
and study their invariant measures; see Theorems \ref{thm:2}
and \ref{thm:2.7}.  This is accomplished by introducing 
finite dimensional approximations due to the discretization in space, 
and then taking limits.  In Section 3, we consider its Cole-Hopf transform
and pass to the limit.  We need to replace a complicated nonlinear term
by a simple linear term in the limit; see Theorems \ref{thm:3.1} and 
\ref{thm:3.1-M}.  This procedure has a similarity 
to the so-called Boltzmann-Gibbs principle \cite{KL}, \cite{KLO}, 
which plays an important role in 
establishing the equilibrium fluctuation limits
for large scale interacting systems.  We rely
on Wiener-It\^o expansion.

\section{KPZ approximating equation}

This section studies the invariant measures of the KPZ approximating 
equations \eqref{eq:1} and \eqref{eq:1-M}.  For this purpose, we first consider the 
associated tilt process $u=\partial_x h$, which satisfies the SPDEs
\eqref{eq:b-4.1} and \eqref{eq:b-4.1-torus}
stated below, and introduce its finite dimensional 
approximations due to the spatial discretization. Indeed, in view of 
finding invariant measures, it is important to choose a special 
discretization scheme.
The infinitesimal invariance is shown in Lemma \ref{lem:D-IP}.
Then, we find the invariant measures of the SPDEs \eqref{eq:b-4.1} 
and \eqref{eq:b-4.1-torus} by passing to the limits; see Theorem \ref{thm:2}.
This part is standard, especially because $0<\e<1$ is fixed and
the noises in \eqref{eq:b-4.1},  \eqref{eq:b-4.1-torus} or \eqref{eq:1}, 
\eqref{eq:1-M} are smooth.  The results are the same even if
we replace $\xi^\e$ by any other constants.

One can actually check the infinitesimal invariance directly for 
\eqref{eq:b-4.1} without introducing the spatial discretization; see 
Remark \ref{rem:D^3}.   The reason we do not do it that way is that there are not clear enough results in the infinite dimensional setting 
telling us that the infinitesimal invariance implies the global invariance.

\subsection{Approximating equations and invariant measures
for tilt processes}  \label{sec:2.1}

Let $\eta \in C_0^\infty(\R)$ be a function satisfying $\eta(x)
\ge 0, \eta(x) = \eta(-x)$ and $\int_\R \eta(x)dx=1$.  We set
$\eta^\e(x) = \eta(x/\e)/\e$ for $\e>0$, $\eta_2(x) =
 \eta*\eta(x)$, and $\eta_2^\e(x) =\eta_2(x/\e)/\e$.  Note that
$\eta_2^\e(x) = \eta^\e*\eta^\e(x)$.  To fix ideas, we assume
that supp$\,\eta \subset [-1,1]$, so that supp$\,\eta^\e \subset
[-\e,\e]$ and supp$\,\eta_2^\e \subset [-2\e,2\e]$.
Define the smeared noise:
\begin{equation}  \label{eq:W^e}
W^\e(t,x) = \lan W(t), \eta^\e(x-\cdot) \ran,
\end{equation}
and consider the KPZ approximating equation \eqref{eq:1} on $\R$
for $h= h^\e(t,x)$.  By the symmetry of $\eta$, we have that
\begin{equation}  \label{eq:2.3}
\xi^\e = \int_\R \eta^\e(y)^2 dy \;  (=\eta_2^\e(0)),
\end{equation}
in \eqref{eq:1}.
The solution $h$ of the SPDE \eqref{eq:1} is smooth in $x$ and 
we are concerned with the associated tilt process $u=\partial_x h$,
which satisfies the stochastic Burgers' equation:
\begin{equation}  \label{eq:b-4.1}
\partial_t u = \frac12 \partial_x^2 u + \frac12 \partial_x(u^2*\eta_2^\e)
+ \partial_x \dot{W}^\e(t,x), \quad x \in \R.
\end{equation}
In this respect, \eqref{eq:1} is a kind of  stochastic 
Hamilton-Jacobi equation. 
Similarly, the tilt process $u=\partial_x h$ of the solution $h$ of 
the SPDE \eqref{eq:1-M} on $\SS_M$, $M\ge 1$, satisfies
\begin{equation}  \label{eq:b-4.1-torus}
\partial_t u = \frac12 \partial_x^2 u + \frac12 \partial_x(u^2*\eta_2^\e)
+ \partial_x \dot{W}^\e(t,x), \quad x \in \SS_M.
\end{equation}
Note that $\int_{\SS_M} u(t,x)dx=0$ holds for \eqref{eq:b-4.1-torus}.

Let $\nu^\e$ be the distribution of $\partial_x(B*\eta^\e(x))$ on 
$\mathcal{C}$, where $B$ is the two-sided Brownian motion satisfying 
$B(0)=0$; we abuse the notation for $\nu^\e$ compared with that
introduced in Section \ref{section:1}, since the meanings 
are clear.  Note that $\nu^\e$ is a probability measure
which is independent of the choice of the value of $B(0)$. 
Similarly, $\nu^{\e,M}$ is the distribution of $\partial_x(B^M*\eta^\e(x))$
on $\mathcal{C}_{M,0}$, where $B^M$ is the pinned Brownian motion,
$\mathcal{C}_{M,0} =\{u\in \mathcal{C}_M; \int_{\SS_M}u(x)dx=0\}$
and recall $\mathcal{C}_M = C(\SS_M)$.
Then, the first main result of this section is formulated as in the
following theorem.  This will be extended to the height process $h$
in Section \ref{section:2.7}; see Theorem \ref{thm:2.7}.

\begin{thm} \label{thm:2}
{\rm (1)} The probability measure $\nu^{\e,M}$ on $\mathcal{C}_{M,0}$ is invariant
under \eqref{eq:b-4.1-torus}, that is, for the tilt process $\partial_x h$ of 
the solution $h$ of the SPDE \eqref{eq:1-M}.\\
{\rm (2)} The probability measure $\nu^\e$ on $\mathcal{C}$ is invariant 
under \eqref{eq:b-4.1}, that is, for the tilt process determined from 
the SPDE \eqref{eq:1}.
\end{thm}

\subsection{Invariant measure of KPZ approximating equation 
on a discrete torus}

In this section, we introduce the KPZ approximating equation
on a discrete torus $\T_N=\{1,2, \ldots,N\}$ with periodic boundary condition.
To study its invariant measure, it is important to choose a special discretization
scheme as we will explain.  Let $\a: \Z \to [0,\infty)$ be given and
satisfy the conditions $\a(i)=\a(-i)$ and $\a(i)=0$ for 
$i: |i|\ge K$ with some $K\ge 1$.  We naturally regard $\a$ as a function on $\T_N$
assuming that $N$ is sufficiently large compared
with the size of the support of $\a$: $N> 2K$.

For $h=(h(i))_{i\in\T_N}\in \R^{\T_N}$, we define $\De h \in \R^{\T_N}$ by
$\De h(i) = h(i+1)+h(i-1)-2h(i), i \in \T_N$ and two functions $G_1(h)
=(G_1(i,h))_{i\in\T_N}$, $G_2(h)=(G_2(i,h))_{i\in\T_N}$ by
\begin{align*}
& G_1(i,h) = (h_{i+1}-h_i)^2 + (h_i-h_{i-1})^2, \\
& G_2(i,h) = (h_{i+1}-h_i)(h_i-h_{i-1}), \quad i\in \T_N,
\end{align*}
respectively.  We sometimes write $h_i$ for $h(i)$.
These are discrete analogues of $2(\partial_x h)^2$ and $(\partial_x h)^2$,
respectively.  For functions $\b, \ga$ on $\T_N$, we define the convolution 
$\b *\ga$ on $\T_N$ by $(\b*\ga)(i) = \sum_{k\in \T_N} \b(i-k)\ga(k), i\in \T_N$, 
where $i-k$ is defined in modulo $N$.

We consider the stochastic differential equation for $h_t
=(h_t(i))_{i\in\T_N} \in \R^{\T_N}$:
\begin{equation} \label{eq:SDE-discrete}
dh_t(i) = \frac{\la_1}2\De h_t(i)dt + \la_2 \{\a_2*G_1(i,h_t)+\a_2*G_2(i,h_t)\}dt + 
\la_3 dw_t^\a(i), \quad
i\in \T_N,
\end{equation}
where $\la_1, \la_2, \la_3\in\R$ are arbitrary constants,
$\a_2 = \a*\a$, $w_t^\a = \a*w_t$ and $w_t = (w_t(i))_{i\in\T_N}$ 
is a family of independent Brownian motions.  
We consider three operators on $\R^{\T_N}$:
\begin{align*}
& {\mathcal L}_0^\a f(h) = \frac{\la_1}2 \sum_{i\in \T_N} \De h(i)
\frac{\partial f}{\partial h_i} + \frac{\la_3^2}2 \sum_{i,j\in \T_N} 
\a_2(i-j) \frac{\partial^2 f}{\partial h_i\partial h_j}, \\
& {\mathcal A}_1^\a f(h) = \sum_{i\in \T_N} (\a_2*G_1)(i,h) 
\frac{\partial f}{\partial h_i}, \\
& {\mathcal A}_2^\a f(h) = \sum_{i\in \T_N} (\a_2*G_2)(i,h) 
\frac{\partial f}{\partial h_i},
\end{align*}
for $f\in C^2(\R^{\T_N})$.
Then, ${\mathcal L}^\a = {\mathcal L}_0^\a +\la_2 {\mathcal A}_1^\a
+ \la_2 {\mathcal A}_2^\a$
is the generator of the SDE \eqref{eq:SDE-discrete}.

Let $\a^{-1}=\a_N^{-1}$ be
the inverse matrix of $\a=\{\a(i-j)\}_{i,j\in\T_N}$.  Note that
the matrix $\a$ may not be invertible in general, but we can
always make $\det \a\not=0$ by slightly perturbing $\a$ 
and we consider such $\a$.
Let $\mu_N(dh) = e^{- I_N^\a(h)} dh$ be an infinite measure on
$\R^{\T_N}$, where $dh=\prod_{i\in\T_N} dh(i)$ and
$$
I_N^\a(h) = \frac{\la}2 \sum_{j\in\T_N} \{\a^{-1}*h(j+1)- \a^{-1}*h(j)\}^2,
\quad \la=\frac{\la_1}{\la_3^2}.
$$

\begin{lem} \label{lem:D-IP}
For every $f, g\in C_b^2(\R^{\T_N})$, we have that
\begin{equation} \label{eq:D-IP-1}
\int  g(h) {\mathcal L}_0^\a f(h) d\mu_N=\int  f(h) {\mathcal L}_0^\a g(h) d\mu_N.
\end{equation}
In particular, $\int{\mathcal L}_0^\a f(h)d\mu_N=0$.
We also have that
\begin{equation} \label{eq:D-IP-2}
\int  {\mathcal A}_1^\a f(h) d\mu_N= - \int  {\mathcal A}_2^\a f(h) d\mu_N.
\end{equation}
Accordingly, we have that
\begin{equation} \label{eq:D-IP-3}
\int  {\mathcal L}^\a f(h) d\mu_N=0.
\end{equation}
\end{lem}

\begin{proof}
We first compute derivatives of $ I_N^\a$:
\begin{align}  \label{eq:derivative-I}
& \frac{\partial}{\partial h_i} I_N^\a(h) \\
& =\la \sum_{j\in\T_N} \{\a^{-1}*h(j+1)- \a^{-1}*h(j)\}
\frac{\partial}{\partial h_i}\{\a^{-1}*h(j+1)- \a^{-1}*h(j)\} \notag \\
& =\la \sum_{j\in\T_N} \sum_{k\in\T_N} \{\a^{-1}(j+1-k)- \a^{-1}(j-k)\}h(k)\cdot
\{\a^{-1}(j+1-i)- \a^{-1}(j-i)\}  \notag \\
& = \la\sum_{k\in\T_N} \{2\a_2^{-1}(i-k)- \a_2^{-1}(i+1-k)- \a_2^{-1}(i-1-k)\}h(k)
 \notag \\
& = - \la\sum_{k\in\T_N} \a_2^{-1}(i-k) \De h(k)
= -\la (\a_2^{-1}*\De h)(i). \notag
\end{align}
We now prove the symmetry \eqref{eq:D-IP-1} of $ {\mathcal L}_0^\a$.  To this end, 
\begin{align*}
\int g \frac{\partial^2 f}{\partial h_i\partial h_j} d \mu_N
& = - \int \frac{\partial}{\partial h_j} \left(g e^{-I_N^\a(h)}\right)
\frac{\partial f}{\partial h_i} dh \\
& = - \int \left( \frac{\partial g}{\partial h_j} - g
\frac{\partial I_N^\a}{\partial h_j} \right) 
\frac{\partial f}{\partial h_i} d\mu_N \\
& = - \int \left( \frac{\partial g}{\partial h_j} + g \la
(\a_2^{-1}*\De h)(j) \right) 
\frac{\partial f}{\partial h_i} d\mu_N,
\end{align*}
by \eqref{eq:derivative-I}.  Therefore, we have that
\begin{align*}
\int g {\mathcal L}_0^\a f d\mu_N
& = -\frac{\la_3^2}2 \sum_{i,j} \int \a_2(i-j) \frac{\partial g}{\partial h_j}
\frac{\partial f}{\partial h_i} d\mu_N
+ \frac{\la_1}2 \sum_i \int g \De h(i) \frac{\partial f}{\partial h_i} d\mu_N\\
& \qquad - \frac{\la_3^2}2\cdot \la \sum_{i,j}\int g \, \a_2(i-j)
(\a_2^{-1}*\De h)(j) \frac{\partial f}{\partial h_i} d\mu_N \\
& = -\frac{\la_3^2}2 \sum_{i,j} \int \a_2(i-j) \frac{\partial g}{\partial h_j}
\frac{\partial f}{\partial h_i} d\mu_N.
\end{align*}
This shows \eqref{eq:D-IP-1}.
We next prove \eqref{eq:D-IP-2}.  For $\ell =1,2$,
\begin{align*} 
& \int  {\mathcal A}_\ell^\a f(h) d\mu_N=
- \sum_i \int f \frac{\partial}{\partial h_i}\left\{ (\a_2*G_\ell)(i)
 e^{-I_N^\a(h)}    \right\}  dh  \\
& = - \int f \left\{ \sum_i \frac{\partial}{\partial h_i} (\a_2*G_\ell)(i) 
- \sum_i  (\a_2*G_\ell)(i) \frac{\partial  I_N^\a}{\partial h_i}
   \right\}  d\mu_N.
\end{align*}
Here, noting that $\frac{\partial G_\ell}{\partial h_i}(j)=0$ if $j\not= i,
i\pm 1$, the first sum vanishes both for $\ell=1,2$:
\begin{align*} 
& \sum_i \frac{\partial}{\partial h_i} (\a_2*G_\ell)(i) \\
& = \sum_i \frac{\partial}{\partial h_i} \{\a_2(-1) G_\ell(i+1) + 
  \a_2(0) G_\ell(i) + \a_2(1) G_\ell(i-1)\} \\
& = \left\{
\begin{aligned}
& -2 (\a_2(0)+\a_2(1)) \sum_i \De h(i) = 0, \quad \text{ for } \ell=1, \\
& \a_2(0) \sum_i \De h(i) + \a_2(1) \sum_i\{
-(h_{i+2}-h_{i+1})+(h_{i-1}-h_{i-2})\} = 0, \quad \text{ for } \ell=2.
\end{aligned}
\right.
\end{align*}
On the other hand, from \eqref{eq:derivative-I}, the second sum
can be rewritten as
\begin{align*} 
- \sum_i  (\a_2*G_\ell)(i) \frac{\partial I_N^\a}{\partial h_i} 
 = \la \sum_i  (\a_2*G_\ell)(i)  (\a_2^{-1}* \De h)(i) 
= \la \sum_i G_\ell(i) \De h(i),
\end{align*}
for $\ell=1,2$.  However, 
\begin{align*} 
& \sum_i G_1(i) \De h(i) \\
& = \sum_i \{ (h_{i+1}-h_i)^2 + (h_i-h_{i-1})^2\} \{ (h_{i+1}-h_i)- (h_i-h_{i-1})\} \\
& = \sum_i (h_{i+1}-h_i)^3 - \sum_i  (h_{i+1}-h_i)^2(h_i-h_{i-1})
 + \sum_i (h_i-h_{i-1})^2(h_{i+1}-h_i) - \sum_i (h_i-h_{i-1})^3 \\
& = \sum_i  G_2(i) \{(h_i-h_{i-1}) - (h_{i+1}-h_i)\}  \\
& = - \sum_i  G_2(i) \De h(i).
\end{align*}
This implies \eqref{eq:D-IP-2}.  \eqref{eq:D-IP-3} is immediate
from \eqref{eq:D-IP-1} and \eqref{eq:D-IP-2}.
\end{proof}

We can apply Echeverria's result \cite{Ech} for the finite dimensional
SDE \eqref{eq:SDE-discrete}  and Lemma \ref{lem:D-IP} proves the
invariance of $\mu_N$. 

We define the tilt variables $u=(u(i))_{i\in\T_N}$ associated with 
$h=(h(i))_{i\in\T_N}$ by $u(i) = \nabla h(i) := h(i+1)-h(i)$.
Note that $u$ always satisfies $\sum_{i\in\T_N} u(i)=0$, and
$I_N^\a(h) = \tilde I_N^\a(u) := \frac{\la}2 \sum_j \{\a^{-1}*u(j)\}^2$.
Let $\nu_N(du) = e^{-\tilde I_N^\a(u)} du/ Z_N^\a$ be
a probability measure on $\R_0^{\T_N} := \{u\in
\R^{\T_N}; \sum_{i\in\T_N} u(i)=0\}$, where $du$ is the Lebesgue 
measure on this space and $Z_N^\a$ is a normalizing constant.  
Then, our result for the height process
$h$ can be transformed into that for the tilt process $u$:

\begin{prop}
The probability measure $\nu_N$ on $\R_0^{\T_N}$ is invariant
for the tilt process $u=\nabla h$ of the SDE \eqref{eq:SDE-discrete}.
\end{prop}

Note that we will later consider $u^N\equiv \nabla_Nh = N\nabla h$;
see Lemma \ref{lem:2.5-a} and \eqref{eq:SDE-u} below.

\subsection{Invariant measure of KPZ approximating equation 
on a continuous torus}

Under a proper scaling in space $i\mapsto x = i/N$, parameters 
$\la_1, \la_2, \la_3$ and $\a(\cdot)$, one can show that
the stationary solution of \eqref{eq:SDE-discrete} converges 
weakly to that of the SPDE \eqref{eq:1-M} with $M=1$, i.e.,
\eqref{eq:1-M} for $x \in \SS=\R/\Z (=[0,1))$, or, for the 
corresponding tilt process, to the SPDE \eqref{eq:b-4.1-torus}
with $M=1$ for fixed $\e>0$, by showing the tightness
of the sequence of stationary solutions of the SDE \eqref{eq:SDE-discrete}.
The goal is to show the following proposition, whose proof will be 
completed in Section \ref{section:2.5}.  This proposition proves 
Theorem \ref{thm:2}-(1).

\begin{prop}\label{prop:2.4-M}
The probability measure $\nu^{\e,1}$ on $\mathcal{C}_{1,0}$
is invariant under the SPDE \eqref{eq:b-4.1-torus} with $M=1$.  
By a simple scaling argument, we see that the probability 
measure $\nu^{\e,M}$ on $\mathcal{C}_{M,0}$ is invariant 
under the SPDE \eqref{eq:b-4.1-torus}.
\end{prop}

We first show the convergence of the invariant measure.
For $u\in \R_0^{\T_N}$, we define $u^N=\{u^N(x);x\in\SS\}$ by a linear
interpolation of $\{u^N(\frac{i}N) := Nu(i)\}_{i\in \T_N}$, that is
\begin{align}  \label{eq:interpolation}
u^N(x) & = u^N(\tfrac{i+1}N)\cdot N(x-\tfrac{i}N)+u^N(\tfrac{i}N)
  \cdot N(\tfrac{i+1}N-x)  \\
& = N^2 u(i+1)(x-\tfrac{i}N)+N^2u(i)(\tfrac{i+1}N-x), 
\quad x \in [\tfrac{i}N,\tfrac{i+1}N).  \notag
\end{align}

\begin{lem}  \label{lem:2.5-a}
Consider $\nu_N$ on $\R_0^{\T_N}$ by choosing $\a: \a(i) = 
\frac1N\eta^\e(\frac{i}N)$ and $\la=N$.  Then, as $N\to\infty$,
the distribution of $u^N$ under $\nu_N$  weakly converges to
$\nu^{\e,1}$ on the space $\mathcal{C}_1=C(\SS)$.
\end{lem}

\begin{proof}
We first observe that the law of  $\{\nabla_N(\a*B^1(\frac{\cdot}N))(i)\}$
coincides with that of $\{u^N(\frac{i}N) = Nu(i)\}_{i\in \T_N}$
under $\nu_N$, where $\{B^1(x);x\in \SS\}$ is the pinned Brownian 
motion such that $B^1(0)=B^1(1)=0$.  In fact, for every $f\in C_b(\R_0^{\T_N})$,
\begin{align*}
E^{\nu_N}[f(u)] 
& =\frac1{Z_N^\a}  \int_{\R_0^{\T_N}} f(u) 
   e^{-\frac{N}2\sum_j \{\a^{-1}*u(j)\}^2} du \\
& =\frac1{\tilde Z_N^\a}  \int_{\R_0^{\T_N}} f(\a*\tilde u) 
   e^{-\frac{N}2\sum_j \tilde u(j)^2} d\tilde u \\
& = E[f(\a*\nabla B^1(\tfrac{\cdot}N))],
\end{align*}
where we have applied the change of variables:
$u= \a*\tilde u$, that is, $\tilde u = \a^{-1}*u$ and 
$d\tilde u= C_Ndu$ for the second line and note that the distribution
of $\{\tilde u(j)\}$ under the probability measure
$e^{-\frac{N}2\sum_j \tilde u(j)^2} d\tilde u/\tilde Z_N^\a$
is equal to that of $\{(\nabla B^1(\frac{\cdot}N))(j)\}$ for the
third line.  Since $\a*\nabla B^1(\frac{\cdot}N)
= \nabla(\a* B^1(\frac{\cdot}N))$, the above computation
implies that the law of $\{\nabla_N (\a*B^1(\frac{\cdot}N))(i)\}$ 
is equal to that of $\{u^N(\frac{i}N)=Nu(i)\}$ under $\nu_N$.
However, it is easy to see that the linear interpolation of
$\{\nabla_N (\a*B^1(\frac{\cdot}N))(i)\}$ converges in
$C(\SS)$ to $\{\partial_x(\eta^\e*B^1)\}$ as $N\to\infty$ a.s., 
and this completes the proof.
\end{proof}

We choose $\la_1=N^2, \la_2=\frac16 N^2, \la_3= \sqrt{N}$
and $\a(i) = \frac1N\eta^\e(\frac{i}N)$ in the SDE 
\eqref{eq:SDE-discrete}, and set
$$
U_N(t) = \frac1N \sum_{i\in \T_N} \{u_t(i)^2 + (\nabla_Nu_t(i))^2\},
$$
where $u_t(i) := \nabla_Nh_t(i) = N(h_t(i+1)-h_t(i))$.  Note that 
$u_t=\{u_t(i)\}$ satisfies the following SDE:
\begin{align}\label{eq:SDE-u}
du_t(i) =& \frac12\De_Nu_t(i)dt + \frac16 \nabla_N[\a_2* \{
u_t(\cdot)^2+u_t(\cdot-1)^2+ u_t(\cdot)u_t(\cdot-1)\}](i) dt \\
&+ \sqrt{N} \nabla_N dw_t^\a(i),  \notag
\end{align}
where $\De_Nu(i):=N^2\De u(i)$.  We define 
$\{u_t^N(x); x \in \SS\}$ by the linear interpolation
of $\{u_t^N(\frac{i}N) := u_t(i)\}_{i\in\T_N}$
as in \eqref{eq:interpolation}.  Then 
$c_1 \|u_t^N\|_{H^1(\SS)}^2 \le U_N(t) \le c_2 \|u_t^N\|_{H^1(\SS)}^2$
with some $0<c_1$ and $c_2<\infty$.  We denote Sobolev spaces
of order $s\ge 0$ on $\SS$ by $H^s(\SS)$.

We consider the stationary
solution of \eqref{eq:SDE-u}, that is, the initial value
before scaling is taken as $\{u_0(i)/N\}_{i\in\T_N}
\stackrel{\rm{law}}{=} \nu_N$.

\begin{lem}  \label{lem:2.6-a}
{\rm (1)} For every $T>0$, we have the uniform bound:
$$
\sup_{N\in\N}E\left[\sup_{0\le t \le T} U_N(t) 
\right] <\infty.
$$
{\rm (2)} For every $T>0, \fa\in C^\infty(\SS)$ and $0\le s < t\le T$,
$$
E[\lan u_t^N-u_s^N,\fa\ran_{\SS}^4] \le C(\fa) (t-s)^2,
$$
holds with $C(\fa)=C_T(\fa)>0$, 
where $\lan u^N,\fa\ran_{\SS} = \int_\SS u^N(x)\fa(x)dx$.  \\
{\rm (3)}  In particular, $\{u_t^N\}_{N\in\N}$ is tight on
$C([0,T],C(\SS))$ for every $T>0$.
\end{lem}

\begin{proof}
The tightness on $C([0,T],H^s(\SS))$ with $s<1$
follows from (1) and (2) noting that the embedding $H^1(\SS)
\subset H^s(\SS)$ is compact by Rellich's theorem; see, e.g.,
the proof of Proposition 3.1 in \cite{F92}.  Therefore, (3) follows
by noting $H^s(\SS)\subset C(\SS)$ continuously embedded
if $s>1/2$.

To show (1), we apply It\^o's formula to see that
\begin{align*}
d U_N(t) & = \frac1N\sum_i \{2u_t(i) d u_t(i)+ (du_t(i))^2
+ 2\nabla_N u_t(i) d\nabla_N u_t(i)+ (d\nabla_N u_t(i))^2\}   \\
 & = \frac1N\sum_i  \big[ u_t(i) \big\{\De_N u_t(i)+
\frac13 \nabla_N[\a_2* \{u_t(\cdot)^2+u_t(\cdot-1)^2+
 u_t(\cdot)u_t(\cdot-1)\}](i) \big\} \\
& \qquad + \nabla_N u_t(i) \big\{\nabla_N \De_N u_t(i)+
\frac13 \nabla_N^2[\a_2* \{u_t(\cdot)^2+u_t(\cdot-1)^2+
 u_t(\cdot)u_t(\cdot-1)\}](i) \big\} \\
& \qquad + (-N \nabla_N^2 \a_2(0) + N \nabla_N^4\a_2(0))\big] 
 \cdot dt\\
& \qquad + \frac{\sqrt{N}}N \sum_i \{2u_t(i) \nabla_N dw_t^\a(i)
+ 2\nabla_Nu_t(i) \nabla_N^2 dw_t^\a(i)\} \\
&=: b_N(t)dt+dm_N(t),
\end{align*}
where $m_N(t)$ denotes the martingale part.
Therefore, we have that
\begin{align} \label{eq:U-N}
E[\sup_{0\le t \le T} U_N(t)] \le E[U_N(0)] + \int_0^TE[|b_N(t)|]dt
+ E[\sup_{0\le t \le T} m_N(t)].
\end{align}
However, by the stationarity of $u_t$, we easily see that
\begin{align*}
& E[U_N(0)] = \frac1N\sum_i E^{\nu_N}[\bar u(i)^2+(\nabla_N\bar u(i))^2] \le C,
\end{align*}
uniformly in $N$, where $\bar u(i)=u^N(\frac{i}N)
(\stackrel{\rm{law}}{=} \nabla_N(\a*B(\frac{\cdot}N))(i))$, and
\begin{align*}
& E[|b_N(t)|] \le \frac1N\sum_i \big(
E^{\nu_N}[ |\bar u(i) \big\{\De_N \bar u(i)+
\frac13 \nabla_N[\a_2* \{\bar u(\cdot)^2+\bar u(\cdot-1)^2+
 \bar u(\cdot)\bar u(\cdot-1)\}](i) \big\} \\
& \qquad\qquad\qquad
  + \nabla_N \bar u(i) \big\{\nabla_N \De_N \bar u(i)+
\frac13 \nabla_N^2[\a_2* \{\bar u(\cdot)^2+\bar u(\cdot-1)^2+
 \bar u(\cdot)\bar u(\cdot-1)\}](i) \big\}| ]  \\
&  \qquad\qquad\qquad
  +|-N \nabla_N^2 \a_2(0) + N \nabla_N^4\a_2(0)|\big) \\
& \qquad\quad\;\;  \le C,
\end{align*}
since $|-N \nabla_N^2 \a_2(0) + N \nabla_N^4\a_2(0)|$ is bounded in $N$
(asymptotically converging to $|-(\eta_2^\e)''(0)+(\eta_2^\e)''''(0)|$
as $N\to\infty$), and
$E^{\nu_N}[ |\nabla_N^\ell \bar u(i)|^p], \ell =0,1,2,3, p\ge 1$
are all independent of $i$ (because of the shift invariance of
$\nu_N$) and uniformly bounded in $N$.
Moreover, by Doob's inequality and then by the stationarity of $u_t$,
\begin{align*}
E[\sup_{0\le t \le T} m_N(t)]^2
& \le E[\sup_{0\le t \le T} m_N(t)^2] \le 4E[m_N(T)^2] \\
& = 4\int_0^Tdt \frac4{N}\sum_j E\big[\big\{\sum_i(u_t(i)\nabla_N\a(i-j)
+ \nabla_Nu_t(i)\nabla_N^2\a(i-j))\big\}^2\big] \\
& \le CT.
\end{align*}
Note that $\nabla_N\a(i-j) = \eta^\e(\frac{i+1-j}N)- \eta^\e(\frac{i-j}N)$
and $\nabla_{N}^2 \a(i-j) = N\{ \eta^\e(\frac{i+2-j}N) -2 \eta^\e(\frac{i+1-j}N)
+ \eta^\e(\frac{i-j}N)\}$ are both $O(1/N)$.  This proves (1).

To show (2), from the definition \eqref{eq:interpolation} of the linear
interpolation for $u_t^N(x)$, 
one can rewrite $\lan u_t^N,\fa\ran_{\SS}$ as a sum in $i$.  Then, applying the
 summation by parts in $i$, we obtain that
\begin{equation} \label{eq:u^N_t}
\lan u_t^N,\fa\ran_{\SS} = \frac1N \sum_i u_t(i) \tilde \fa^N(i),
\end{equation}
where
$$
\tilde\fa^N(i) = N^2 \int_\SS \{1_{[\tfrac{i-1}N,\tfrac{i}N)}(x) (x-\tfrac{i-1}N)
+ 1_{[\tfrac{i}N,\tfrac{i+1}N)}(x) (\tfrac{i+1}N-x) \}\fa(x)dx.
$$
However, the Taylor expansion of $\fa(x)$ around $x=i/N$ in the right 
hand side up to the third order leads to 
$$
\tilde\fa^N(i) = \fa(\tfrac{i}N) + \tfrac1{12N^2} \fa''(\tfrac{i}N) +
 r^N(i), \quad i\in \T_N,
$$
with remainder terms $r^N(i)$ satisfying $|r^N(i)|\le C/N^3$.
This implies that $|R^N(i)|$, $|\nabla_N R^N(i)|$, $|\De_N R^N(i)| \le C/N$
for $R^N(i) := \tilde\fa^N(i)-\fa(\tfrac{i}N)$.  Then, from \eqref{eq:SDE-u}
and \eqref{eq:u^N_t}, for $0\le s \le t$, we have that
\begin{align}  \label{eq: u_t-u_s}
\lan u_t^N-u_s^N,\fa\ran_{\SS}= I^{(1)}+  I^{(2)} +  I^{(3)},
\end{align}
where
\begin{align*}
 I^{(1)}&= \frac1{2N} \int_s^t \sum_i u_r(i) \{ (\De_N\fa(\tfrac{\cdot}N))(i) 
 + \De_N R^N(i)\} dr, \\
 I^{(2)}& = - \frac1{6N} \int_s^t \sum_i 
\a_2* \{u_r(\cdot)^2+u_r(\cdot-1)^2+ u_r(\cdot)u_r(\cdot-1)\}(i)  \\
& \qquad\qquad\qquad\qquad  \times
\{ (\nabla_N\fa(\tfrac{\cdot}N))(i)  + \nabla_N R^N(i)\} dr, \\
 I^{(3)}&= \frac1{\sqrt{N}}\sum_i\nabla_N(w_t^\a-w_s^\a)(i)\{\fa(\tfrac{i}N)+R^N(i)\}.
\end{align*}
Noting that $E[|u_r(i)|^p] = E^{\nu_N}[|\bar u(i)|^p]$ (by stationarity) are bounded
in $N$ for $p\ge 1$, we easily see that
$
E[(I^{(1)})^4], E[(I^{(2)})^4]  \le C(\fa)(t-s)^4.
$
Moreover, it is also easy to see that
$
E[(I^{(2)})^4] \le C(\fa)(t-s)^2.
$
This proves (2).
\end{proof}

\subsection{The martingale problems associated 
with the SPDEs \eqref{eq:b-4.1} and \eqref{eq:b-4.1-torus}}

To complete the proofs of Proposition \ref{prop:2.4-M} and then
Theorem \ref{thm:2}, we introduce the martingale formulations for the
SPDEs \eqref{eq:b-4.1}  and \eqref{eq:b-4.1-torus}.  To this end,
we first introduce the (pre) generators of the processes $h(t)$ 
determined by \eqref{eq:1} or \eqref{eq:1-M}.

Let $\mathcal{D} = \mathcal{D}(\mathcal{C})$ be the class of 
all tame functions $\Phi$ on 
$\mathcal{C}=C(\R)$, that is, those of the form:
\begin{equation} \label{eq:2.tame}
\Phi(h) = f(\lan h,\fa_1\ran, \ldots, \lan h,\fa_n\ran), \quad
h\in \mathcal{C},
\end{equation}
with $n= 1,2,\ldots$, $f=f(z_1,\ldots,z_n)\in C_b^2(\R^n), \fa_1
\ldots, \fa_n \in C_0^\infty(\R)$, where $\lan h,\fa\ran =
\int_\R h(x)\fa(x)dx$.  We define its functional derivatives by
\begin{align}  \label{eq:2.5-D}
& D\Phi(x;h) 
= \sum_{i=1}^n \partial_{z_i}f(\lan h,\fa_1\ran, \ldots, \lan h,
\fa_n\ran)  \fa_i(x), \\
& D^2\Phi(x_1,x_2;h)
 = \sum_{i,j=1}^n \partial_{z_i}\partial_{z_j}
f(\lan h,\fa_1\ran, \ldots, \lan h,\fa_n\ran) \fa_i(x_1)\fa_j(x_2).
     \label{eq:2.6-D}
\end{align}
The class $\mathcal{D}_{\infty} = \mathcal{D}_{\infty}
(\mathcal{C})$ stands for the family
of all $\Phi \in \mathcal{D}$ determined by \eqref{eq:2.tame}
with $f\in C_\infty^2(\R^n)$ such that 
$
\lim_{|z|\to\infty} \left\{|f(z)| + |\partial_{z_i}f(z)| + |\partial_{z_i}\partial_{z_j}f(z)| 
\right\} =0.
$

For $\Phi\in \mathcal{D}$, define two operators $\mathcal{L}_0^\e$
and $\mathcal{A}^\e$ by
\begin{align*}
\mathcal{L}_0^\e\Phi(h) & = \frac12 \int_{\R^2}
 D^2 \Phi(x_1,x_2;h) \eta_2^\e(x_1-x_2) dx_1 dx_2
  + \frac12 \int_\R \partial_x^2h(x) D\Phi(x;h) dx,  \\
\mathcal{A}^\e\Phi(h) & = \frac12\int_\R\big((\partial_x h)^2 -\xi^\e \big)
 *\eta_2^\e(x) D\Phi(x;h) dx.
\end{align*}
Then, $\mathcal{L}^\e := \mathcal{L}_0^\e + \mathcal{A}^\e$ is the (formal)
generator corresponding to the SPDE \eqref{eq:1}.  In fact, by applying
It\^o's formula, we have that
$$
d\Phi(h_t) = \lan D\Phi(x;h_t), d h_t(x)\ran_\R
+ \frac12 \lan D^2\Phi(x_1,x_2;h_t), dW^\e(t,x_1)
dW^\e(t,x_2)\ran_{\R^2}
$$
and note that
\begin{equation} \label{eq:2.cov}
dW^\e(t,x_1)dW^\e(t,x_2) = \eta_2^\e(x_1-x_2) dt.
\end{equation}

The (formal) generator corresponding to the SPDE 
\eqref{eq:b-4.1} for the tilt process $u=\partial_x h$ is given by
$\mathcal{L}^{\e,U} = \mathcal{L}_{0}^{\e,U} + \mathcal{A}^{\e,U}$,
where
\begin{align*}
\mathcal{L}_0^{\e,U}\Phi(u) & = \frac12 \int_{\R^2}
  D^2 \Phi(x_1,x_2;u) \, \partial_{x_1} \partial_{x_2}\{\eta_2^\e(x_1-x_2)\} dx_1 dx_2
  + \frac12 \int_\R \partial_x^2u(x) D\Phi(x;u) dx,  \\
\mathcal{A}^{\e,U}\Phi(u) & = \frac12\int_\R
\partial_x (u^2 *\eta_2^\e)(x)  D\Phi(x;u) dx,
\end{align*}
for $\Phi=\Phi(u)\in \mathcal{D}$, which is given by \eqref{eq:2.tame}
with $u$ in place of $h$.  Note that the derivatives $\partial_x^2$ and 
$\partial_x$ in these operators can be moved to $D\Phi(x;u)$
by the integration by parts.

We similarly define $\mathcal{D}(\mathcal{C}_M)$ and 
$\mathcal{D}(\mathcal{C}_{M,0})$ as the classes of all
$\Phi$ on $\mathcal{C}_M$ and $\mathcal{C}_{M,0}$, respectively,
of the forms \eqref{eq:2.tame} with $\fa_i \in C^\infty(\SS_M)$ 
and $\lan h,\fa\ran_{\SS_M} := \int_{\SS_M} h(x)\fa(x)dx$
in place of $\lan h,\fa\ran$.
Then, operators $\mathcal{L}_{0,M}^\e, \mathcal{A}_M^\e$ together
with $\mathcal{L}_M^\e := \mathcal{L}_{0,M}^\e + \mathcal{A}_M^\e$
on $\mathcal{D}(\mathcal{C}_M)$ and $\mathcal{L}_{0,M}^{\e,U}, 
\mathcal{A}_M^{\e,U}$ together with $\mathcal{L}_M^{\e,U} := 
\mathcal{L}_{0,M}^{\e,U} + \mathcal{A}_M^{\e,U}$ on 
$\mathcal{D}(\mathcal{C}_{M,0})$ are defined as
$\mathcal{L}_{0}^\e, \mathcal{A}^\e, \mathcal{L}^\e$
and $\mathcal{L}_{0,M}^{\e,U}, \mathcal{A}_M^{\e,U}, 
\mathcal{L}_M^{\e,U}$, respectively, by replacing the integrals over 
$\R^2$ and $\R$ by those over  $\SS_M^2$ and $\SS_M$,
respectively.  We also consider the classes of functions
 $\mathcal{D}_\infty(\mathcal{C}_M)$ and 
$\mathcal{D}_\infty(\mathcal{C}_{M,0})$.

\begin{rem}  \label{rem:2.1-c}
We can regard $\mathcal{L}^\e$ as the generator of the tilt process
$u$ by replacing its domain.  In fact, let $\mathcal{D}_{\nabla}=
\mathcal{D}_{\nabla}(\mathcal{C})$ be the class of all
$\Phi\in\mathcal{D}$ with $\fa_i$ satisfying $\int_\R\fa_i dx=0,
1\le i \le n$.  This is a natural class of functions for tilt variables,
since, under the equivalence relation $h \sim h+c$ with some
$c\in\R$, we have
$\Phi(h) = \Phi(h+c)$ if $\Phi \in \mathcal{D}_{\nabla}$ so that
$\Phi$ is a function on the quotient space $\tilde{\mathcal{C}}
= \mathcal{C}/\!\!\sim$.
For the function 
$\Phi \in \mathcal{D}_{\nabla}$, though we write its variable by
$h$, the height $h$ itself has no meaning.  In particular, if $h$ is
differentiable, $\Phi \in \mathcal{D}_{\nabla}$ can be considered
as a function of its tilt $u:=h'\equiv \partial_x h$:  if $\Phi(u) = 
f(\lan u,\psi_1\ran,\ldots, \lan u,\psi_n\ran)$ with $\psi_1,\ldots,
\psi_n\in C_0^\infty(\R)$, then $\lan u,\psi_i\ran =
\lan h, \fa_i\ran$ with $\fa_i:=-\psi_i'$ and $\fa_i$ satisfies the condition
$\int_\R \fa_i dx=0$, which is the additional condition
imposed on $\Phi \in \mathcal{D}_{\nabla}$.  We can also define
$\mathcal{D}_{\nabla}(\mathcal{C}_M)$ as the class of all
$\Phi\in \mathcal{D}(\mathcal{C}_M)$ with $\fa_i \in C^\infty(\SS_M)$ 
satisfying $\int_{\SS_M} \fa_i(x)dx =0$. 
\end{rem}

We now introduce the martingale problems  associated 
with the SPDEs \eqref{eq:b-4.1} and \eqref{eq:b-4.1-torus}
on extended spaces.  Recall that $\e>0$ is fixed so that the noise 
$\dot{W}^\e(t,x)$ is smooth in $x$.  
As a state spaces for the SPDE \eqref{eq:b-4.1},
we take $C(\R)\cap L_r^2(\R), r>0,$ where $L_r^2(\R)$
is the weighted $L^2$-space; recall \eqref{eq:L_r-norm}.

\begin{lem}\label{lem:2.3.1}
{\rm (1)} If the probability measure $P$ on 
$C([0,\infty), C(\R)\cap L_r^2(\R))$
is a solution of the 
$(\mathcal{L}^{\e,U},\mathcal{D}_{\infty})$-martingale problem,
then there exists  $\dot{W}^\e(t,x)$, which is defined on this space
and a Gaussian smeared noise under $P$, such that the
coordinate function $u(t)$ is a solution of the SPDE \eqref{eq:b-4.1}
in the generalized functions' sense; i.e., \eqref{eq:b-4.1} holds multiplied by
any test function $\fa\in C_0^\infty(\R)$ and integrated over $\R$
(as in \eqref{eq:1.2-a}).  \\
{\rm (2)} Similar results hold on $\SS_M$: Under the solution $P$ 
on $C([0,\infty),\mathcal{C}_{M,0})$ of the
$(\mathcal{L}_M^{\e,U},\mathcal{D}_{\infty}(\mathcal{C}_{M,0}))$-martingale
problem, the coordinate function $u(t)$ satisfies the SPDE
\eqref{eq:b-4.1-torus} in the generalized functions' sense
with a certain Gaussian smeared noise $\dot{W}^\e(t,x)$ on
$[0,\infty)\times \SS_M$.
\end{lem}

\begin{proof}
To prove (1), we use two types of functions $\Phi_1(u) = \lan u,\fa\ran$
and $\Phi_2(u) = \lan u,\fa_1\ran \lan u,\fa_2\ran$; more precisely, their
cut-off functions such as $\Phi_{1,N}(u) = g_N(\lan u,\fa\ran)$ with
$g_N\in C_\infty^2(\R)$ satisfying $g_N(x)=x$ for $|x|\le N$
and similarly defined functions $\Phi_{2,N}$ for $\Phi_2$. We denote by
$
b(u,\fa) = \tfrac12 \lan \partial_x^2 u + \partial_x (u^2*\eta_2^\e),\fa\ran.
$
Then,
\begin{align}  \label{eq:2.l.1}
M_t(\fa) := & \lan u(t),\fa\ran - \lan u(0),\fa\ran-
\int_0^t \mathcal{L}^{\e,U} \Phi_1(u(s))ds \\
= & \lan u(t),\fa\ran -\lan u(0),\fa\ran- \int_0^t b(u(s),\fa)ds   \notag
\end{align}
is a local martingale.  Moreover, by noting that  $\Phi_2(u(t))- \int_0^t 
\mathcal{L}^{\e,U} \Phi_2(u(s))ds$ is a local martingale and
$\mathcal{L}^{\e,U} \Phi_2(u) = \lan \partial_{x_1}\partial_{x_2}
\{\eta_2^\e(\cdot-\cdot)\},\fa_1\otimes
\fa_2\ran_{\R^2} + b(u,\fa_1)\lan u,\fa_2\ran +b(u,\fa_2)\lan u,\fa_1\ran$,
and by applying It\^o's formula, we can easily see that
\begin{align}  \label{eq:2.l.2}
M_t(\fa_1) M_t(\fa_2) -
t \int_{\R^2} \partial_{x_1}\partial_{x_2}
\{\eta_2^\e(x_1-x_2)\} \fa_1(x_1) \fa_2(x_2) dx_1dx_2
\end{align}
is a local martingale.  This implies that the cross variation of
two local martingales $M_t(\fa_1)$ and $M_t(\fa_2)$ with 
$\fa_1, \fa_2\in C_0^\infty(\R)$ is given by
\begin{align}  \label{eq:2.l.3}
\lan M(\fa_1), M(\fa_2)\ran_t = t \lan V\fa_1,\fa_2\ran,
\end{align}
where the right hand side denotes the inner product in
$L^2(\R) =L^2(\R,dx)$ and
$$
V\fa(x) = \partial_x \int_{\R} \eta_2^\e(x-x_1) (-\partial_{x_1}) \fa(x_1) dx_1.
$$
Introducing operators: $R=\partial_x$ and
\begin{align*}
& Q\fa(x) = \int_{\R} \eta_2^\e(x-x_1) \fa(x_1) dx_1, \\
& Q^{\frac12}\fa(x) = \int_{\R} \eta^\e(x-x_1) \fa(x_1) dx_1,
\end{align*}
we can rewrite $V$ as $V = (R Q^{\frac12})(R Q^{\frac12})^*$ as operators
on $L^2(\R)$.  Note that $Q$ and $Q^{\frac12}$ are symmetric on $L^2(\R)$, 
$(Q^{\frac12})^2 =Q$, and, in particular, $Q$ is non-negative, but
$\|Q^{\frac12}\|_{\rm HS}^2 = \int_{\R^2}(\eta^\e(x_1-x_2))^2 dx_1dx_2 =\infty$  
so that ${\rm Tr} \, Q =\infty$.

By the martingale representation theorem (see, e.g., \cite{DZ} Theorem 8.2,
actually stated only in case ${\rm Tr} \, Q<\infty$, and also
Remark \ref{rem:cov-Q} below), \eqref{eq:2.l.3} implies that
$$
M_t(x) = R W^Q(t,x) \equiv \partial_x W^Q(t,x),
$$
where $W^Q$ is the $Q$-Wiener process, which has the representation
\eqref{eq:W^e} with a space-time Gaussian white noise.  This with
\eqref{eq:2.l.1} implies the conclusion of (1).  The proof of (2)
is similar.
\end{proof}

\begin{rem}  \label{rem:cov-Q}
In our case, as we pointed out, the assumption ${\rm Tr} \, Q<\infty$
is not satisfied.  To overcome this, we may first define $M_t^N(x)$
by restricting $M_t(x)$ on $[-N,N]$ and periodically extending it to
$[-N-2\e,N+2\e]$.  For $M_t^N$, the corresponding $Q$-operator
is given by
\begin{align*}
Q^N\fa(x) = \int_{-N-2\e}^{N+2\e} \eta_2^\e(x-x_1) \fa(x_1) dx_1, \quad x\in [-N,N],
\end{align*}
with $\fa$ defined on $[-N,N]$ but periodically extended to
$[-N-2\e,N+2\e]$ and this operator becomes of trace class.  Therefore, one can
apply Theorem 8.2 of \cite{DZ} and construct $W^{Q^N}(t,x)$.  Then,
by the consistency of  $W^{Q^N}(t,x)$, one can extend it to the whole 
line $\R$.
\end{rem}

\subsection{Proof of Proposition \ref{prop:2.4-M}} \label{section:2.5}

We may assume $M=1$ without loss of generality.  Recall that 
$\{u_t^N(t); x\in \SS\}$ is defined by the linear interpolation of the 
stationary solution $u_t=\{u_t(i)\}$ of \eqref{eq:SDE-u} in such
a manner that $u_t^N(\tfrac{i}N) = u_t(i), i\in
\T_N$, and it is tight on $C([0,T],C(\SS))$ from Lemma 
\ref{lem:2.6-a}.  Therefore, by Skorohod's representation theorem,
we can realize on a proper probability space such that
$u_t^N$ converges to some $u_t$ in $C([0,T],C(\SS))$ as
$N\to\infty$ a.s.\ for every $T>0$.
We abuse the notation.  Then, for every $\Phi\in
\mathcal{D}(\mathcal{C}(\SS))$, we have that
\begin{equation} \label{eq:dPhi}
d\Phi(u_t^N) = \lan D\Phi(\cdot;u_t^N), du_t^N\ran_{\SS}
+ \frac12 \lan D^2\Phi(x_1,x_2;u_t^N), du_t^N(x_1) du_t^N(x_2)\ran_{\SS^2}.
\end{equation}
However, from \eqref{eq: u_t-u_s}, we have that
\begin{align}  \label{eq:du_t^N}
 d \lan u_t^N,\fa\ran
= &  \frac1{2N} \sum_i u_t^N(\tfrac{i}N) \{ (\De_N\fa(\tfrac{\cdot}N))(i) 
 + \De_N R^N(i)\} dt \\
&  - \frac1{6N} \sum_i \a_2* \{u_t(\frac{\cdot}N)^2+u_t(\tfrac{\cdot-1}N)^2
+ u_t(\tfrac{\cdot}N)u_t(\tfrac{\cdot-1}N)\}(i)  \notag \\
 & \qquad\qquad\qquad\qquad  \times
\{ (\nabla_N\fa(\tfrac{\cdot}N))(i)  + \nabla_N R^N(i)\} dt   \notag  \\
& - \frac1{\sqrt{N}}\sum_i \{(\nabla_N\fa(\tfrac{\cdot}N))(i)+\nabla_NR^N(i)\}
\sum_j \a(i-j) dw_t(j).  \notag  
\end{align}
In particular, recalling that $|\nabla_NR^N(i)| \le C/N$,
\begin{align*}
& d \lan u_t^N,\fa_1\ran d \lan u_t^N,\fa_2\ran  \\
& = \frac1N \sum_{i,j}\a_2(i-j)
\{(\nabla_N\fa_1(\tfrac{\cdot}N))(i)+\nabla_NR_1^N(i)\}
\{(\nabla_N\fa_2(\tfrac{\cdot}N))(j)+\nabla_NR_2^N(j)\} dt\\
& = \frac1{N^3} \sum_{i,j,k}\eta^\e(\tfrac{i-k}N)\eta^\e(\tfrac{k-j}N)
\{(\nabla_N\fa_1(\tfrac{\cdot}N))(i)+\nabla_NR_1^N(i)\}
\{(\nabla_N\fa_2(\tfrac{\cdot}N))(j)+\nabla_NR_2^N(j)\} dt\\
& \to \int_{\SS^3}\eta^\e(x-z)\eta^\e(y-z)\fa_1'(x)\fa_2'(y)dxdydz \cdot dt
= \int_{\SS^2}\eta_2^\e(x-y)\fa_1'(x)\fa_2'(y)dxdy \cdot dt,
\end{align*}
as $N\to\infty$.
Since $u_t^N$ converges to $u_t$ in the space
$C([0,T],C(\SS))$ a.s., for the limit $u_t$, we see 
from \eqref{eq:dPhi} and \eqref{eq:du_t^N} that
$$
\Phi(u_t)- \int_0^t  \mathcal{L}_1^{\e,U}\Phi(u_s)ds
$$
is a martingale for every  $\Phi\in \mathcal{D}(C(\SS))$ and therefore
for $\Phi\in \mathcal{D}(\mathcal{C}_{1,0})$.
This completes the proof of Proposition \ref{prop:2.4-M} with the
help of Lemma \ref{lem:2.3.1}-(2) and Lemma \ref{lem:2.5-a}.

\subsection{Invariant measure of KPZ approximating equation 
on $\R$}

Let $u_t^M=\{u_t^M(x); x\in \SS_M=[0,M)\}$ be the stationary solution of
the SPDE \eqref{eq:b-4.1-torus}, that is, $u_0^M \stackrel{\rm{law}}{=}
\nu^{\e,M}$,  constructed in  Proposition \ref{prop:2.4-M}.
We extend $u_t^M$ periodically  on $\R$.

\begin{lem}  \label{lem:2.8-a}
{\rm (1)} For every $T>0$ and  $r>0$, we have
$$
\sup_{M\ge 1}E\left[\sup_{0\le t \le T} \|u_t^M\|_{H_r^1(\R)}^2 
\right] <\infty,
$$
where $\|u\|_{H_r^1(\R)}^2 = \|u\|_{L_r^2(\R)}^2 
+ \|\partial_xu\|_{L_r^2(\R)}^2$.  \\
{\rm (2)} For every $T>0, \fa\in C_0^\infty(\R)$ and $0\le s<t\le T$,
$$
E[\lan u_t^M-u_s^M,\fa\ran^4] \le C(\fa) (t-s)^2,
$$
holds with $C(\fa)=C_T(\fa)>0$.  \\
{\rm (3)}  In particular, $\{u_t^M\}_{M\ge 1}$ is tight on
$C([0,T],C(\R)\cap L_r^2(\R))$ for every $T, r>0$.
\end{lem}

\begin{proof}
The proof is parallel to that of Lemma \ref{lem:2.6-a}.  Indeed,
(3) follows from (1) and (2) noting that the embedding $H_r^1(\R)
\subset H_{r'}^s(\R)$ is compact if $r'>r>0$ and $s<1$, and also
$H_r^s(\R) \subset C(\R)$ if $s>1/2$; see \cite{F95}, p.284 for
the weighted Sobolev spaces $H_r^s(\R)$.

To show (1), set $U^M(t) = \|u_t^M\|_{H_r^1(\R)}^2$.  Then, by
It\^o's formula,
\begin{align*}
d U^M(t) & = \int_\R\{2u_t(x) d u_t(x)+ (du_t(x))^2
+ 2\partial_x u_t(x) d\partial_x u_t(x)+ (d\partial_x u_t(x))^2\}
e^{-2r\chi(x)}dx \\
 & = \int_\R [ u_t(x) \{\partial_x^2 u_t(x)+\partial_x (u_t^2*\eta_2^\e)\}
+ \partial_x u_t(x) \{\partial_x^3 u_t(x)+\partial_x^2 (u_t^2*\eta_2^\e)\} \\
& \qquad + (-(\eta_2^\e)''(0)+(\eta_2^\e)''''(0)) ] 
e^{-2r\chi(x)}dx  \cdot dt\\
& \qquad + \int_\R \{2u_t(x) d\partial_xW^\e(t,x) + 
2\partial_x u_t(x) d\partial_x^2 W^\e(t,x)\} e^{-2r\chi(x)}dx\\
&=: b^M(t)dt+dm^M(t),
\end{align*}
where $u_t=u_t^M$ and $W^\e(t,x)$ originally defined on $\SS_M$
is periodically extended on $\R$.  We can bound $E[\sup_{0\le t \le T} U^M(t)]$
by the sum of three terms similarly to \eqref{eq:U-N}.  However, we easily see that
\begin{align*}
& E[U^M(0)] = E^{\nu^{\e,M}}[\|\partial_x B^M*\eta^\e\|_{H_r^1(\R)}^2] \le C
\quad (\text{uniformly in }M),\\
\intertext{where $\{B^M(x);x\in \SS_M\}$ is periodically extended on $\R$, and}
& E[|b^M(t)|] \le \int_\R e^{-2r\chi(x)}dx \big(
E^{\nu^{\e,M}}[ | u(x) \{\partial_x^2 u(x)+\partial_x (u^2*\eta_2^\e)\} \\
& \qquad \qquad \qquad
+ \partial_x u(x) \{\partial_x^3 u(x)+\partial_x^2 (u^2*\eta_2^\e)\}|]
+ (-(\eta_2^\e)''(0)+(\eta_2^\e)''''(0))\big) \\
& \qquad\qquad\;\,  \le C,
\end{align*}
since $E^{\nu^{\e,M}}[ |\partial_x^\ell u(x)|^p], \ell =0,1,2,3, p\ge 1$
are all independent of $x$ (because of the shift invariance of
$u(x)$ under $\nu^{\e,M}$) and uniformly bounded in $M$.
Moreover, by Doob's inequality and then by the stationarity of $u_t$,
\begin{align*}
E[\sup_{0\le t \le T} m^M(t)]^2
& \le 4\int_0^Tdt \int_{\R^2} 8e^{-2r(\chi(x)+\chi(y))}dxdy \big\{
 E^{\nu^{\e,M}}[u(x)u(y)] \partial_x \partial_y\eta_{2,M}^\e(x-y) \\
& \qquad\qquad \qquad\qquad \qquad
+ E^{\nu^{\e,M}}[\partial_x u(x)\partial_y u(y)] \partial_x^2 \partial_y^2
\eta_{2,M}^\e(x-y) \big\}  \\
& \le CT,
\end{align*}
where $\eta_{2,M}^\e(x-y)$ is defined in the sense of modulo $M$
in $x-y$; note that $\partial_x \partial_y \eta_{2,M}^\e$ and
$\partial_x^2 \partial_y^2\eta_{2,M}^\e$ are bounded in $M, x, y$. 
We have estimated as $m^M(T)^2 \le 2(m_1^M(T)^2+m_2^M(T)^2)$
by decomposing $m^M(T)$ into the sum of two stochastic integrals
$m_1^M(T)$ and $m_2^M(T)$.  This proves (1).

For (2), denoting $u_t=u_t^M$ again, we see that
\begin{align*}
\lan u_t-u_s,\fa\ran
&= \frac12 \int_s^t \{\lan u_r,\fa''\ran - 
\lan u_r^2*\eta_2^\e,\fa'\ran \}dr
-\{W^\e(t,\fa')-W^\e(s,\fa')\} \\ 
&=: I^{(1)}+  I^{(2)}.
\end{align*}
However, we easily see that
$
E[(I^{(1)})^4] \le C(\fa)(t-s)^4
$
by the stationarity of $u_t$, and
$
E[(I^{(2)})^4] = C(\fa)(t-s)^2,
$
since $W^\e(t,\fa')$ is a Brownian motion multiplied by a certain
constant.  This proves (2).
\end{proof}

Let $u=\{u(x);x\in \SS_M\}$ be a $\mathcal{C}_{M,0}$-valued random variable
distributed under $\nu^{\e,M}$ and, by periodically extending $u$
on $\R$, we can regard $\nu^{\e,M}$ as a probability distribution
on $\mathcal{C}$.  Then, the following lemma is easy and the proof 
is omitted.

\begin{lem} \label{lem:2.9-b}
The distribution  $\nu^{\e,M}$ weakly converges to $\nu^\e$
on the space $\mathcal{C}$ as $M\to\infty$.
\end{lem}

We are now ready to give the proof of Theorem \ref{thm:2}.

\begin{proof}[Proof of Theorem \ref{thm:2}]
The assertion (1) is already shown by Proposition \ref{prop:2.4-M}.
Let us prove (2).  We have shown in Lemma \ref{lem:2.8-a} that  the periodically
extended stationary solution  $\{u_t^M\}_{M\ge 1}$ of
the SPDE \eqref{eq:b-4.1-torus} is tight on
$C([0,T],C(\R)\cap L_r^2(\R))$ for every $r>0$.
Therefore, by Skorohod's representation theorem, we can
realize on a proper probability space that $u_t^M$ converges to
some $u_t$ in $C([0,T],C(\R)\cap L_r^2(\R))$ for every $T, r>0$
as $M\to\infty$ a.s.    Then, for every
$\Phi\in \mathcal{D}(\mathcal{C})$,
$$
\Phi(u_t^M)- \int_0^t \mathcal{L}_M^{\e,U}\Phi(u_s^M)ds
$$
is a martingale.  Here, in the operator $\mathcal{L}_M^{\e,U}$,
the function $\eta_2^\e$ should be understood in
the sense of modulo $M$.  However, noting that the supports of the functions
$\fa_1,\ldots,\fa_n$ appearing in $\Phi$ are compact, we see that
$\mathcal{L}_M^{\e,U}\Phi(u^M)$ converges to $\mathcal{L}^{\e,U}\Phi(u)$
as $M\to\infty$ if $u^M$ converges to $u$ in $C([0,T],C(\R)\cap L_r^2(\R))$.
Thus, one can prove that, for the limit $u_t$,
\begin{equation} \label{eq:mart-d}
\Phi(u_t)- \int_0^t  \mathcal{L}^{\e,U}\Phi(u_s)ds
\end{equation}
is a martingale for every  $\Phi\in \mathcal{D}(\mathcal{C})$.
This completes the proof of Theorem \ref{thm:2}-(2) with the
help of Lemma \ref{lem:2.3.1}-(1) and Lemma \ref{lem:2.9-b}.
\end{proof}

As a corollary, we can prove the infinitesimal invariance of
$\mathcal{L}^{\e,U}$, the symmetry of $\mathcal{L}_0^{\e,U}$ and the
asymmetry of $\mathcal{A}^{\e,U}$ under $\nu^\e$, respectively, or integration
by parts formulas, and similar results on $\SS_M$.

\begin{cor}  \label{cor:asymmetry}
{\rm (1)}
For every $\e>0$ and $\Phi\in \mathcal{D}(\mathcal{C})$, we have that
\begin{equation}  \label{eq:Inv-7}
\int \mathcal{L}^{\e,U} \Phi d\nu^\e =0.
\end{equation}
The operators $\mathcal{L}_0^{\e,U}$ and $\mathcal{A}^{\e,U}$ are symmetric
and asymmetric with respect to $\nu^\e$, respectively, that is,
for every $\Phi, \Psi\in \mathcal{D}(\mathcal{C})$,
\begin{align}  \label{eq:Inv-aaa}
\int \Psi\mathcal{L}_0^{\e,U} \Phi d\nu^\e 
&= \int \Phi\mathcal{L}_0^{\e,U} \Psi d\nu^\e,\\
\intertext{and}
\int \Psi \mathcal{A}^{\e,U} \Phi d\nu^\e 
&= -\int \Phi \mathcal{A}^{\e,U} \Psi d\nu^\e.  \label{Inv-bbb}
\end{align}
{\rm (2)} Similar results hold on $\SS_M$ with $\mathcal{L}_M^{\e,U}$,
$\mathcal{L}_{0,M}^{\e,U}$, $\mathcal{A}_M^{\e,U}$ and $\nu^{\e,M}$
in place of $\mathcal{L}^{\e,U}$, $\mathcal{L}_0^{\e,U}$, 
$\mathcal{A}^{\e,U}$ and $\nu^\e$, respectively.
\end{cor}

\begin{proof}
We give the proof of (1) only.
\eqref{eq:Inv-7} follows by taking the average of the martingale
\eqref{eq:mart-d} and noting that $u_t\stackrel{\rm{law}}{=} \nu^\e$.
\eqref{eq:Inv-aaa} can be shown from \eqref{eq:D-IP-1}
rewritten at the tilt level and by taking the limits twice as we did,
or it can be directly shown by noting that $\nu^\e$ is reversible
for the Ornstein-Uhlenbeck process determined by the SPDE:
\begin{equation*}
\partial_t u = \frac12 \partial_x^2 u 
+ \partial_x \dot{W}^\e(t,x), \quad x \in \R.
\end{equation*}
Since $\mathcal{A}^{\e,U} = \mathcal{L}^{\e,U} - \mathcal{L}_0^{\e,U}$,
\eqref{eq:Inv-7} and \eqref{eq:Inv-aaa} with $\Psi=1$ prove that
\begin{equation}  \label{eq:Inv-6}
\int \mathcal{A}^{\e,U}\Phi d\nu^\e =0,
\end{equation}
and \eqref{Inv-bbb} follows from this by noting that 
$\mathcal{A}^{\e,U}(\Phi \Psi) = \Psi \mathcal{A}^{\e,U} \Phi 
+\Phi \mathcal{A}^{\e,U} \Psi$.
\end{proof}

\begin{rem}\label{rem:D^3}
We can alternatively prove the infinitesimal invariance \eqref{eq:Inv-7} directly 
using the Wiener-It\^o expansion, see \cite{F14}.
\end {rem}

\begin{rem}  \label{rem:Yaglom}
(Yaglom reversibility)
Corollary \ref{cor:asymmetry}
suggests that the generator of the time reversed process
under $\nu^\e$ is given by $\mathcal{L}_0^{\e,U} - 
\mathcal{A}^{\e,U}$.  Coming back to the level of the height processes,
a simple computation shows that $\mathcal{L}_0^\e \check{\Phi}(h) = 
\mathcal{L}_0^\e \Phi(\check{h})$ and $\mathcal{A}^\e \check{\Phi}(h) = 
-\mathcal{A}^\e \Phi(\check{h})$ for $\Phi\in \mathcal{D}$, where
$\check{h}$ and $\check{\Phi}$ are defined by the transformations
 $\check{h}(x) = -h(-x)$ and $\check{\Phi}(h) = \Phi(\check{h})$,
respectively.   This means that
$\check{h}(t,x):= - h(t,-x)$ determined from the solution 
$h(t,x)$ of the SPDE \eqref{eq:1} admits the (pre) generator   
$\mathcal{L}_0^\e - \mathcal{A}^\e$.
\end{rem}

\subsection{Invariant measure for the height process}
\label{section:2.7}

Theorem \ref{thm:2} deals with the tilt processes only, but this can
be easily extended to the height process.  Theorem \ref{thm:2.7}
will not be used later, but we state it for its own interest.  Set
\begin{equation*}
X_t^\e = \frac12\int_0^t \partial_x^2 h^\e(s,0)ds
+ \frac12 \int_0^t \big((\partial_x h^\e(s))^2
 -\xi^\e \big)* \eta_2^\e(0)ds + W^\e(t,0),
\end{equation*}
for the solution $h^\e(t,x)$ of \eqref{eq:1}.
The key point is that, as functions of $h^\e$, the first and
second terms of $X_t^\e$ are defined on the quotient space
$\tilde{\mathcal{C}}$ defined in Remark \ref{rem:2.1-c}.
Therefore, once $h^\e(t)\in 
\tilde{\mathcal{C}}$ is determined by solving the SPDE 
\eqref{eq:b-4.1}, we can recover its height at $x=0$ as
\begin{equation}\label{smoothed2}
h^\e(t,0) = h^\e(0,0) + X_t^\e.
\end{equation}

\begin{thm}  \label{thm:2.7}
For any bounded, integrable and continuous function $G=G(h_0,h)$
on $\R\times\tilde{\mathcal{C}}$ and for any $\e>0$, $t\ge 0$,
we have that
\begin{equation}\label{smoothed3}
\int_{\R\times\tilde{\mathcal{C}}} G(h^\e(t,0),h^\e(t))dh_0d\nu^\e
= \int_{\R\times\tilde{\mathcal{C}}} G(h^\e(0,0),h^\e(0))
dh_0 d\nu^\e,
\end{equation}
where $dh_0$ means that $h^\e(0,0)$ is distributed under the
Lebesgue measure on $\R$.  Note that
$\R\times\tilde{\mathcal{C}}$ can be identified with 
$\mathcal{C}$.  Similar results hold on $\SS_M$.
\end{thm}

\begin{proof}
From \eqref{smoothed2} and then by the translation-invariance
of the Lebesgue measure and performing the integral in $dh_0$
first, the left hand side of \eqref{smoothed3} is equal to
$$
\int_{\R\times\tilde{\mathcal{C}}} G(h^\e(0,0)+X_t^\e,h^\e(t))
 dh_0d\nu^\e
= \int_{\R\times\tilde{\mathcal{C}}} G(h^\e(0,0),h^\e(t)) 
dh_0 d\nu^\e.
$$
But, this is equal to the right hand side of \eqref{smoothed3}
by the invariance of $\nu^\e$ under $h^\e(t)\in
\tilde{\mathcal{C}}$  due to Theorem \ref{thm:2}-(2).
\end{proof}

\section{Cole-Hopf transform of KPZ approximating equation
and proofs of Theorems \ref{thm:0} and \ref{thm:1.1}}

Our goal is to study the limit of the KPZ approximating equation
\eqref{eq:1} on $\R$ or \eqref{eq:1-M} on $\SS_M$ 
as $\e\downarrow 0$.  To this end,
we move to the level of the corresponding Cole-Hopf transformed
process $Z^\e(t)$ rather than staying with \eqref{eq:1} or \eqref{eq:1-M}, 
and show that $Z^\e(t)$ converges to the solution $Z(t)$ of the 
SPDE \eqref{eq:8-b} on $\R$ or \eqref{eq:8-b-M} on $\SS_M$
at least if the corresponding
tilt process is stationary.  This implies that the solution $h^\e(t)$ of
\eqref{eq:1}  or \eqref{eq:1-M} converges to $h(t) + \tfrac1{24}t$
as $\e\downarrow 0$, where $h(t)$ is the Cole-Hopf solution of 
the KPZ equation defined by \eqref{eq:1.6}.  We can actually
do this only for \eqref{eq:1-M}; due to a technical reason,
we do not have Proposition \ref{prop:3.12}  on $\R$.
Since all arguments except this do work  on $\R$,
we state the results on $\R$ in Sections 
\ref{section:3.1}--\ref{section:3.3}.  Then, we study the
SPDEs on $\SS_M$ in Sections \ref{section:3.4} and
\ref{subsection:3.4.3}.  Finally in Section \ref{section:3.6},
letting $M\to\infty$, as a byproduct, we find an invariant measure
of the SHE \eqref{eq:1.1} on $\R$.

\subsection{The equation for $Z^\e(t)$}\label{section:3.1}

Under the transformation $h\mapsto Z$ defined by $Z=e^h$, the
KPZ approximating equation \eqref{eq:1} is transformed into the
equation \eqref{eq:8} for $Z=Z^\e(t)$.
In fact, by applying It\^o's formula and recalling \eqref{eq:2.cov}
with $x_1=x_2=x$,
\begin{align*}
dZ & = e^h dh + \frac12 e^h (dW^\e)^2 \\
& = \frac12 Z \Big(\partial_x^2 h + \big((\partial_x h)^2
 -\xi^\e \big)* \eta_2^\e\Big) dt + Z dW^\e  + \frac12 Z \xi^\e dt \\
& = \frac12 Z \Big(\partial_x^2 h + (\partial_x h)^2* \eta_2^\e\Big) dt
 +  Z dW^\e.
\end{align*}
Thus, \eqref{eq:8} is obtained noting that $\partial_x^2 h + (\partial_x h)^2
= Z^{-1} \partial_x^2 Z$ and $\partial_x h = \partial_x Z/Z$. 
The derivation of \eqref{eq:8-M} from \eqref{eq:1-M} is the same.

We define the notion of tilt variables associated with the process
$Z(t)$.  This is a reformulation of those defined for $h$
above Remark \ref{rem:1.1} or in Remark \ref{rem:2.1-c}.
For $Z^1, Z^2\in \mathcal{C}_+$, we say $Z^1 \sim Z^2$
if there exists $c>0$ such that $Z^1(x)=cZ^2(x)$ for all $x\in \R$.
Then, by the linearity and uniqueness of solutions of the SPDE
\eqref{eq:1.1}, we see that $Z^1(t) \sim Z^2(t)$ holds if $Z^1(0)
\sim Z^2(0)$ for two solutions $Z^1(t), Z^2(t)$ of \eqref{eq:1.1}. 
Thus, \eqref{eq:1.1} defines a stochastic evolution $\tilde{Z}(t)$
on the quotient space $\tilde{\mathcal{C}}_+ := 
\mathcal{C}_+/\!\!\sim$.  The SPDE \eqref{eq:8} has the same
character, though it is nonlinear.

\subsection{Wrapped processes}  \label{sec:3.2}

To avoid the complexity arising from the infiniteness of the
invariant measure of $h^\e(t,x)$, we introduce a
modified process $g^\e(t,x)$ of $h^\e(t,x)$.
Let us take $\rho\in C_0^\infty(\R)$ satisfying $\rho\ge 0$,
supp $\rho\subset [-1,1]$, supp $\rho$ is connected,
and $\int_\R\rho(x)dx=1$, and fix it in the rest of the paper
except the last step of Section \ref{subsection:3.4.3},
where we take two different $\rho$'s.
We define a wrapped process $g^\e(t,x)$ of $h^\e(t,x)$ by
$g^\e(t,x)= h^\e(t,x) + N^\e(t)$ with $N^\e(t) = - [h^\e(t,\rho)]$,
more precisely its right continuous modification, where 
$[h]\in \Z$ stands for the integer part of $h\in \R$ and
$h^\e(t,\rho) = \int_\R h^\e(t,x)\rho(x)dx$.  
In particular,  $g^\e(x,\rho)$ defined from $g^\e(t,x)$ 
similarly to $h^\e(t,\rho)$ always satisfies
$g^\e(t,\rho) \in [0,1]$ a.s.\ and $g^\e(t,\rho) \equiv h^\e(t,\rho)$
modulo $1$. 
 
In the next lemma, the initial distribution of $h^\e(0,\cdot)$ is
taken to be $\pi\otimes\nu^\e$, where $\pi$ is a uniform
measure on $[0,1]$, under the decomposition of the height:
\begin{equation} \label{eq:3.1-map}
g \mapsto (g(\rho),\{g(x)-g(\rho);x\in\R\}),
\end{equation}
into the height averaged by $\rho$ and the tilt variable.  
Then, $g^\e(t)$ considered as a $[0,1]\times\tilde{\mathcal{C}}$-valued
process under the map \eqref{eq:3.1-map} is stationary in $t$:

\begin{lem}  \label{lem:3.1-a}
The probability measure $\pi\otimes\nu^\e$ on $[0,1]\times
\tilde{\mathcal{C}}$ is invariant under $g^\e(t,x)$.
\end{lem}

\begin{proof}
Take a periodic and smooth function $f$ on $[0,1]$ and
set $\Psi(h) = f(h(\rho))$ for $h\in \mathcal{C} \cong
[0,1]\times\tilde{\mathcal{C}}$ under the map 
\eqref{eq:3.1-map}, where $h(\rho) = \int_\R h(x)\rho(x)dx$.  Then, since
$$
D\Psi(x;h) = f'(h(\rho)) \rho(x), \quad
D^2\Psi(x_1,x_2;h) = f''(h(\rho)) \rho(x_1)\rho(x_2),
$$
we have that
\begin{align*}
L_{\nabla h}^\e f(h(\rho)) :=& \mathcal{L}^\e \Psi(h) 
= \frac{\xi_\rho^\e}2 f''(h(\rho)) + \frac12 b^\e(\nabla h)
 f'(h(\rho)),
\end{align*}
where 
\begin{align*}
& \xi_\rho^\e = \int_{\R^2} \rho(x_1)\rho(x_2) \eta_2^\e(x_1-x_2)
dx_1 dx_2, \\ 
& b^\e(\nabla h) = h(\rho'') + \int_\R
((\partial_x h)^2-\xi^\e)*\eta_2^\e(x) \rho(x)dx.
\end{align*}
Note that $b^\e(\nabla h)$ is a tilt variable.
Take another function $\Phi=\Phi(\nabla h)$ of tilt variables
$\nabla h = \{\partial_x h; x \in \R\}$.  Then, since
\begin{align*}
\mathcal{L}^\e(\Psi\Phi)
& = \Phi\mathcal{L}^\e\Psi +  \Psi\mathcal{L}^\e\Phi
+ \int_{\R^2} D\Psi(x_1;h)D\Phi(x_2;h) \eta_2^\e(x_1-x_2) 
dx_1dx_2 \\
& = \Phi L_{\nabla h}^\e f +  f\mathcal{L}^\e\Phi
+ f'(h(\rho)) \lan D\Phi(\cdot;h)* \eta_2^\e,\rho\ran,
\end{align*}
noting that $\mathcal{L}^\e\Phi$ and $\lan D\Phi(\cdot;h)* \eta_2^\e,\rho\ran$
are tilt variables, we have that
\begin{align}  \label{eq:3.aa}
& \int_{[0,1]\times\tilde{\mathcal{C}}}\mathcal{L}^\e(\Psi\Phi)
d\pi\otimes\nu^\e
 = E^{\nu^\e}\left[\Phi(\nabla h) \int_0^1 L_{\nabla h}^\e f(h(\rho))d\pi\right]\\
& \qquad \qquad  \notag+  \int_0^1f(h(\rho))d\pi \,
E^{\nu^\e}[\mathcal{L}^\e\Phi]
+ \int_0^1 f'(h(\rho)) d\pi\,
 E^{\nu^\e}[\lan D\Phi(\cdot;h)* \eta_2^\e,\rho\ran ].
\end{align}
However, we easily see that
$$
\int_0^1 L_{\nabla h}^\e f(h(\rho))d\pi =
\int_0^1 L_{\nabla h}^\e f(a)da =0
$$
for all fixed $\nabla h$ by the periodicity of $f$, and also
$$
\int_0^1 f'(h(\rho)) d\pi = \int_0^1 f'(a) da=0.
$$
Moreover, noting that $\mathcal{L}^\e$ acting on $\Phi=\Phi(\nabla h)$ 
through $h$ coincides with $\mathcal{L}^{\e,U}$ acting on $\Phi=\Phi(u)$
through $u$, Corollary \ref{cor:asymmetry} shows that the second term
in the right hand side of \eqref{eq:3.aa} vanishes, and therefore
we have that
$$
\int_{[0,1]\times\tilde{\mathcal{C}}}\mathcal{L}^\e(\Psi\Phi)
d\pi\!\otimes\!\nu^\e=0.
$$
This can be extended to linear combinations of the functions of the
form $\Psi\Phi$, and concludes the proof of the lemma.
\end{proof}

We next introduce the Cole-Hopf transform $Y^\e(t,x) =e^{g^\e(t,x)}$ of 
the wrapped process $g^\e(t,x)$.  The initial distribution of
$h^\e(0,\cdot)$ is taken as mentioned above
Lemma \ref{lem:3.1-a}.  $Y^\e(t,x)$ is called a wrapped
process of $Z^\e(t,x) = e^{h^\e(t,x)}$ and satisfies $Y_\rho^\e(t)\in [1,e]$
a.s., where we define
\begin{equation}  \label{eq:3.5-c}
Y_\rho= \exp\left\{ \int_\R \log Y(x) \rho(x)dx\right\},
\end{equation}
for  $Y=\{Y(x)>0; x\in \R\}$ and $Y_\rho^\e(t) = (Y^\e(t))_\rho$.

\begin{lem}  \label{lem:3.2-R}
$Y^\e(t,x)$ satisfies the following equation in generalized functions' sense:
\begin{align}  \label{eq:3-Y^e-2}
Y^\e(t,x) =  Y^\e(0,x) + & \frac12 \int_0^t \partial_x^2 Y^\e(s,x)ds 
+\int_0^t A^\e(x,Y^\e(s)) ds \\
 & + \int_0^t Y^\e(s,x) dW^\e(s,x) + N^\e(t,x),   \notag
\end{align}
where
\begin{align}  \label{eq:3.4}
A^\e(x,Y) & = \frac12 Y(x)\left\{\left(
\frac{\partial_x Y}Y \right)^2*\eta_2^\e(x) 
- \left(\frac{\partial_x Y}Y \right)^2(x) \right\},  \\
N^\e(t,x) & =
\int_0^t \left\{ (e-1) 1_{\{ Y_\rho^\e(s-)=1\}}
+  (e^{-1}-1) 1_{\{ Y_\rho^\e(s-)=e\}} \right\}Y^\e(s-,x) N^\e(ds).
  \label{eq:3.5}
\end{align}
\end{lem}

\begin{proof}
Note that $N^\e(t)=-[h^\e(t,\rho)]$ is expressed as
$$
N^\e(t) = \int_0^t \sum_{a=\pm1} a \, n^\e(ds,a)
$$
with a certain point process $n^\e(ds,a)$ on $\mathbb{X}
= \{\pm1\}$.  Thus, applying It\^o's formula for $Y^\e(t,x) =
Z^\e(t,x) e^{N^\e(t)} \equiv F(Z^\e(t,x),N^\e(t))$ with $F(z,n) = ze^n,
z\in\R, n\in \Z$, we have that
\begin{align*} 
Y^\e(t,x) &=  Y^\e(0,x) + \int_0^t \frac{\partial F}{\partial z}
(Z^\e(s,x),N^\e(s)) dZ^\e(s,x) \\
&\qquad +\int_0^{t+} \sum_{a=\pm1} \{ F(Z^\e(s,x),N^\e(s-)+a)
- F(Z^\e(s,x),N^\e(s-))\} n^\e(ds,a) \\
& = Y^\e(0,x) + \int_0^t e^{N^\e(s)} dZ^\e(s,x)
+ N^\e(t,x),
\end{align*}
where $N^\e(t,x)$ is defined by \eqref{eq:3.5}.
The conclusion follows from \eqref{eq:8}.
\end{proof}

\subsection{Asymptotic behavior of the nonlinear term in
\eqref{eq:3-Y^e-2}}  \label{section:3.3}

We need to analyze the limit of the third term in the right hand
side of \eqref{eq:3-Y^e-2} as $\e\downarrow 0$ at least in the 
stationary situation.  The goal of this subsection is to show the
following theorem, by which one can replace $A^\e(x,Y^\e(s))$
with a linear function $\tfrac1{24} Y^\e(s,x)$ if $h^\e(0,\cdot)$
is distributed under $\pi\otimes\nu^\e$.

\begin{thm}  \label{thm:3.1}
For every $\fa\in C_0(\R)$ satisfying $\rm{supp}\,\fa \cap  
\rm{supp}\,\rho = \emptyset$ (so that $\dist(\rm{supp}\,\fa,
\rm{supp}\,\rho)>0$), we have that
$$
\lim_{\e\downarrow 0} E^{\pi\otimes\nu^\e}
\left[ \sup_{0\le t \le T}\left\{ \int_0^t \hat A^\e(\fa,Y^\e(s))ds \right\}^2 
\right] =0,
$$
where
\begin{align*}
\hat A^\e(\fa,Y) & = \int_\R \hat A^\e(x,Y)
\fa(x)dx, \\
\hat A^\e(x,Y) & = A^\e(x,Y) - \frac1{24} Y(x).
\end{align*}
In particular, under the time average, $A^\e(\fa,Y^\e(s))$
can be replaced by $\tfrac1{24}\int_\R Y^\e(s,x)
\fa(x)dx$ in $L^2(\Om)$ in a strong topology as $\e\downarrow 0$
under the equilibrium situation, if $\fa$ satisfies the above
conditions.
\end{thm}

\begin{rem}  \label{rem:3.1}
{\rm (1)} The time average is essential to show this theorem.
At each fixed time, we never have this type of statement;
see Remark \ref{rem:3.1} below.  \\
{\rm (2)} The constant $\frac1{24}$ frequently appears in KPZ computations;
see e.g.\ Theorem 2.3 of \cite{BG}, Theorem 1.1 of \cite{C} and
Proposition 5.1 of \cite{BCF}.
\end{rem}

The proof of Theorem \ref{thm:3.1} will be carried out at the
level of the height processes $g^\e(t,x)$ or $h^\e(t,x)$ not at that
of the transformed processes $Y^\e(t,x)$ or $Z^\e(t,x)$,
and in a similar way
to that of the Boltzmann-Gibbs principle, which is needed in
the study of the equilibrium fluctuation and establishes a
replacement of a certain complex term by a linear term.
In particular, we deduce an equilibrium dynamic problem into
 a static problem.
To this end, we first consider the symmetric part 
$\mathcal{S}^\e := \tfrac12(\mathcal{L}^\e+\mathcal{L}^{\e\ast})$
of the generator $\mathcal{L}^\e$ of the height process
and the corresponding Dirichlet
form.  Since  $\mathcal{L}^\e = \mathcal{L}_0^\e + \mathcal{A}^\e$,
and $\mathcal{L}_0^\e$ is symmetric and $\mathcal{A}^\e$ is asymmetric
with respect to $\pi\otimes \nu^\e$, we see that $\mathcal{S}^\e
= \mathcal{L}_0^\e$; see Corollary \ref{cor:asymmetry} (at least
at the level of tilt variables) and
arguments in the proof of Lemma \ref{lem:3.1-a}.

The corresponding Dirichlet form is given in the next lemma.

\begin{lem}  \label{lem:3.4}
\begin{align*}
\|\Phi\|_{1,\e}^2 := & \lan \Phi, (-\mathcal{L}_0^\e)\Phi\ran_{\pi\otimes\nu^\e}  
=  \frac12 E^{\pi\otimes\nu^\e}\left[ \int_\R
\left( D\Phi(\cdot;h)*\eta^\e\right)^2(x)dx \right].
\end{align*}
\end{lem}

Before giving the proof of this lemma, we note that the limit 
as $\e\downarrow 0$ of $\nu^\e$ for tilt variables
(and therefore defined on $\tilde{\mathcal{C}}$) can be identified
with the Gaussian random measure $\nu$ on $(\R, 
\mathcal{B}(\R))$ determined from $dB$. More precisely, under $\nu$, random variables $\{X(A); A\in 
\mathcal{B}(\R)\}$ are given and
\begin{enumerate}
\item $X(A) \stackrel{\text{law}}{=} N(0,|A|)$ with $|A|=$
the Lebesgue measure of $A$,
\item If $\{A_i \in \mathcal{B}(\R)\}_{i=1}^n$ are disjoint, then
$\{X(A_i)\}_{i=1}^n$ are independent and $X(\cup_{i=1}^nA_i) = \sum_{i=1}^n
X(A_i)$ a.s.
\end{enumerate}
Such $X(A)$ can be constructed from $X((a,b]) := B(b)-B(a)$
in terms of the two-sided Brownian motion $\{B(x);x\in \R\}$
satisfying, for instance, $B(0)=0$.

\begin{proof}[Proof of Lemma \ref{lem:3.4}]
We first note that  $\mathcal{L}_0$ defined as the limit of 
$\mathcal{L}_0^\e$ as $\e\downarrow 0$, that is,
$$
\mathcal{L}_0\Phi(h) = \frac12 \int_{\R}
 D^2 \Phi(x,x;h) dx + \frac12 \int_\R \partial_x^2h(x) D\Phi(x;h) dx,  
$$
is the generator of
the Ornstein-Uhlenbeck process determined by the SPDE
$$
\partial_t h = \frac12 \partial_{x}^2 h + \dot{W}(t,x),\quad x \in \R,
$$
and $\pi\otimes\nu$ is reversible under the wrapped process
$g(t,x)$ of $h(t,x)$ so that it is reversible under $\mathcal{L}_0$.
It is easy to see that
\begin{equation} \label{eq:Lem3-1-a}
\lan \Psi, (-\mathcal{L}_0)\Phi\ran_{\pi\otimes\nu}
=  \frac12 E^{\pi\otimes\nu}\left[ \int_\R
D\Psi(x;h)D\Phi(x;h)dx \right].
\end{equation}
Now, for a given $\Phi$, we set $\tilde\Phi^\e(h) := 
\Phi(h*\eta^\e)$ and take $\tilde\Phi^\e$ and $\tilde\Psi^\e$
in place of $\Phi$ and $\Psi$, respectively, in \eqref{eq:Lem3-1-a}.
Then, noting that
\begin{align*}
D \tilde\Phi^\e(x;h) & = D\Phi(\cdot;h*\eta^\e)*\eta^\e(x) \\
D^2 \tilde\Phi^\e(x_1,x_2;h) & = D^2\Phi(\cdot,\cdot;h*\eta^\e)*(\eta^\e)^{\otimes 2}(x_1,x_2),
\end{align*}
we can show that
$
\mathcal{L}_0 \tilde\Phi^\e(h)
= \mathcal{L}_0^\e \Phi(h*\eta^\e)
$
and therefore
$
\lan \tilde\Psi^\e, (-\mathcal{L}_0) \tilde\Phi^\e\ran_{\pi\otimes\nu}
= \lan \Psi, (-\mathcal{L}_0^\e) \Phi\ran_{\pi\otimes\nu^\e}
$
by the change of variables.  On the other hand, the right hand
side of \eqref{eq:Lem3-1-a} with $\tilde\Phi^\e$ and 
$\tilde\Psi^\e$ in place of $\Phi$ and $\Psi$, respectively,
is rewritten as
\begin{align*}
\frac12 & E^{\pi\otimes\nu}\left[ \int_\R
D\Psi(\cdot;h*\eta^\e)*\eta^\e(x)
D\Phi(\cdot;h*\eta^\e)*\eta^\e(x) dx \right]  \\
& = \frac12 E^{\pi\otimes\nu^\e}\left[ \int_\R
D\Psi(\cdot;h)*\eta^\e(x)
D\Phi(\cdot;h)*\eta^\e(x) dx \right],
\end{align*}
by the change of variables again.
This concludes the proof of the lemma.
\end{proof}

The basic tool of the proof of Theorem \ref{thm:3.1} is the
Wiener-It\^o expansion.
Recall that the multiple Wiener integral of order $n\ge 1$ with
a kernel $\fa_n \in \hat{L}^2(\R^n)$, i.e.\ $\fa_n \in L^2(\R^n)$
and symmetric in $n$-variables, is defined by
\begin{align*}
I(\fa_n) & = \frac1{n!} \int_{\R^n} \fa_n(x_1,\ldots,x_n) dB(x_1)\cdots dB(x_n) \\
& = \int_{-\infty}^\infty dB(x_1) \int_{-\infty}^{x_1} dB(x_2)
\cdots \int_{-\infty}^{x_{n-1}}  \fa_n(x_1,\ldots,x_n) dB(x_n),
\end{align*}
where $B$ is the two-sided Brownian motion on $\R$ introduced
above.
Set $\mathcal{H}_n = \{I(\fa_n) \in L^2(\tilde{\mathcal{C}},\nu); 
\fa_n \in \hat{L}^2(\R^n)\}$ for $n\ge 1$ and $\mathcal{H}_0
= \{\text{const}\}$.   Then, the well-known
Wiener-It\^o (Wiener chaos) expansion of $\Phi \in \mathcal{H}
:= L^2(\tilde{\mathcal{C}},\nu)$ is given by
\begin{equation}\label{eq:11-b}
\Phi = \sum_{n=0}^\infty I(\fa_n) \in \bigoplus_{n=0}^\infty \mathcal{H}_n,
\end{equation}
with some $\fa_0\in\R$ and $\fa_n\in \hat{L}^2(\R^n)$, 
where $I(\fa_0)=\fa_0$, and
\begin{align} \label{eq:2.6}
\|\Phi\|_{L^2(\nu)}^2 
 = \sum_{n=0}^\infty \| I(\fa_n) \|_{L^2(\nu)}^2
 = \sum_{n=0}^\infty \frac1{n!} \|\fa_n \|_{L^2(\R^n)}^2
\end{align}
holds because of the orthogonality and then by It\^o isometry.
The expansion \eqref{eq:11-b} identifies $\Phi \in
L^2(\tilde{\mathcal{C}},\nu)$ with the element 
$
\varphi = \{\varphi_n\}_{n=0}^\infty \in 
\bigoplus_{n=0}^{\infty} \hat{L}^2( \mathbb{R}^n)
$
of the symmetric Fock space.  The reason to do this is that it 
gives an explicit representation of $D$: $D\Phi(x)$ has
representation $\{D\varphi_n\}_{n=1}^\infty$ 
where $D\varphi_n \in \hat{L}^2( \mathbb{R}^{n-1})$ is given by 
\begin{align}\label{five}
D\varphi_n(x; x_1,\ldots,x_{n-1})
& = -\frac1n \sum_{i=1}^n \partial_i \varphi_n 
( x_1,\ldots,x_{i-1}, x, x_i, \ldots, x_{n-1})\\
& = - \partial_1 \varphi_n 
(x, x_1,\ldots,x_{n-1}).  \notag
\end{align}
The factor $\tfrac1n$ arises when we replace $\tfrac1{n!}$
with $\tfrac1{(n-1)!}$, and the second equality is due to the
symmetry of $\fa_n$.

The next task toward the proof of Theorem \ref{thm:3.1}
is to express the norm $\|\Phi\|_{1,\e}$ of
$\Phi\in L^2(\tilde{\mathcal{C}},\nu)$ in terms of its Wiener
chaos expansion \eqref{eq:11-b}.  

\begin{lem}  \label{lem:3.5}
For $\Phi\in L^2(\tilde{\mathcal{C}},\nu)$,
$$
\|\Phi\|_{1,\e}^2 = \frac12 \sum_{n=0}^\infty \frac1{n!}
\int_{\R^{n+1}} \big(D\fa_{n+1}(x;x_1,\ldots,x_n)*
(\eta^\e)^{\otimes (n+1)}\big)^2 dxdx_1\cdots dx_n.
$$
\end{lem}

\begin{proof}
Lemma \ref{lem:3.4} applied for a function $\Phi$ of tilt variables
gives
\begin{align*}
\|\Phi\|_{1,\e}^2
= \frac12 \int_\R E^{\nu}\left[ 
\left( D\Phi(\cdot;B*\eta^\e)*\eta^\e\right)^2(x)\right] dx.
\end{align*}
Here, we see that
\begin{align*}
& \left( D\Phi(\cdot;B*\eta^\e)*\eta^\e\right)(x)\\
& = \sum_{n=1}^\infty \frac1{(n-1)!} \int_{\R^{n-1}} 
 D\fa_n(x;x_1,\ldots,x_{n-1})*(\eta^\e)^{\otimes n}
 dB(x_1)\cdots dB(x_{n-1}).
\end{align*}
Therefore, the conclusion follows from \eqref{eq:2.6}.
\end{proof}

We are now almost ready to start the proof of Theorem \ref{thm:3.1}.
But, before starting, we slightly extend Lemma 2.4 of \cite{KLO}, 
p.48 stated for temporally homogeneous functions
$V(x)$ to more general temporally inhomogeneous $V(s,x)$
as follows.  Recall that this lemma holds generically under the stationary situation:
$L$ is the generator of a process $X_t$, $\pi$ is its invariant probability
measure, $S=(L+L^*)/2$ and semi-norms $\|\cdot\|_{-1}$ and $\|\cdot\|_1$ are
defined based on the operator $S$: $\|f\|_1^2 = E^\pi[f\cdot (-S)f]$ and
$\|f\|_{-1}^2 = \sup_g \{2E^\pi[fg]-\|g\|_1^2\}$.  This extension is actually
needed only for the proof of Lemma \ref{cor:3.10}, and not for that
of Theorem \ref{thm:3.1}.

\begin{lem}  \label{lem:3.6-a}
For $V=V(s,x)$,
$$
E^\pi\left[ \sup_{0\le t \le T} \left(\int_0^tV(s,X_s)ds\right)^2\right]
\le 24 \int_0^T \|V(s,\cdot)\|_{-1}^2ds.
$$
\end{lem}

\begin{proof}
We give the outline of the proof.  For $f=f(t,x)$, let  $M_t$ be
the  martingale 
\begin{equation} \label{eq:Dynkin-1}
M_t = f(t,X_t)-f(0,X_0)-\int_0^t \left(\frac{\partial}{\partial s} +L\right) f(s,X_s)ds.
\end{equation}
Then, we have
\begin{equation} \label{eq:Dynkin-2}
E^\pi[M_t^2] = 2\int_0^t \|f(s,\cdot)\|_1^2ds. 
\end{equation}
In fact, the proof of \eqref{eq:Dynkin-2} is similar to that of (2.16) in \cite{KLO}, p.47.
Note that, because of the temporal inhomogeneity in our situation,  
two terms $E^\pi[f(t,X_t)^2]- E^\pi[f(0,X_0)^2]$ and 
$E^\pi [\int_0^t \frac{\partial}{\partial s} f(s,X_s)^2ds] (= 
\int_0^t \frac{\partial}{\partial s} E^\pi [f(s,X_s)^2]ds)$ appear, but these terms
just cancel.

For given $V(s,\cdot)$, let us take $f(s,\cdot)$ such that 
$E^\pi[(Sf(s,\cdot)-V(s,\cdot))^2] \le \de$ and $\|f(s,\cdot)\|_1 \le 
\|V(s,\cdot)\|_{-1} +\de$ for $\de>0$ and $s\in [0,T]$,
and define $M_t$ as in \eqref{eq:Dynkin-1} and $M_t^-$ by
$$
M_t^- = f(T-t,X_{T-t})-f(T,X_T)-\int_0^t \left(-\frac{\partial}{\partial s} +L^*\right)
 f(T-s,X_{T-s})ds, \quad t \in [0,T].
$$
Then, by a simple computation, we have
\begin{align*}
M_t+M_T^--M_{T-t}^- & = - \int_0^t (L+L^*)f(s,X_s)ds \\
& =-2 \int_0^t \{V(s,X_s) - R(s,X_s)\} ds,
\end{align*}
where $R:=V-Sf$ satisfies $E^\pi[R^2(s,\cdot)]\le \de$.
This combined with \eqref{eq:Dynkin-2} implies the concluding estimate as in \cite{KLO}.
\end{proof}

Due to this lemma, we have the bound:
\begin{align}  \label{eq:thm3-1-1}
& E^{\pi\otimes\nu^\e}
\left[ \sup_{0\le t \le T} \left\{ \int_0^t \hat A^\e(\fa(s,\cdot),Y^\e(s))ds \right\}^2 
\right] \\
& \qquad \le 24 T \sup_{\Phi\in L^2(\pi\otimes\nu^\e)}
\left\{ 2E^{\pi\otimes\nu^\e}
\left[ \hat A^\e(\fa,Y) \Phi \right] - \|\Phi\|_{1,\e}^2\right\}.
 \notag
\end{align}
In fact, the reason we introduced the wrapped process mostly
lies in applying this bound.
For $\Phi = \Phi(h(\rho),\nabla h)\in L^2(\pi\otimes\nu^\e)$,
we can rewrite
\begin{align}  \label{eq:3.14-b}
2 E^{\pi\otimes\nu^\e}
\left[ \hat A^\e(\fa,Y) \Phi \right] 
= E^{\pi}\left[ Y_\rho E^{\nu^\e}
\left[ B^\e(\fa,Y) \Phi(h(\rho),\nabla h) \right] \right],
\end{align}
where
\begin{align*}
B^\e(x,Y) & = \left\{\left(\frac{\partial_x Y}Y \right)^2
*\eta_2^\e(x) - \left(\frac{\partial_x Y}Y \right)^2(x)
-\frac1{12} \right\} \frac{Y(x)}{Y_\rho}
\left(= \frac{2\hat A^\e(x,Y)}{Y_\rho} \right), \\
B^\e(\fa,Y) & = \int_\R B^\e(x,Y) \fa(x) dx,
\end{align*}
and recall that $Y_\rho\equiv e^{h(\rho)}$,
$h(x)=\log Y(x)$  is defined by \eqref{eq:3.5-c}.
The integration of $\Phi(h(\rho),\nabla h)$ under $\nu^\e$ is
performed in $\nabla h$ by fixing $h(\rho)$.  Note that
$B^\e(x,Y)$ is a tilt variable, though $\hat A^\e(x,Y)$ is not.
The bound \eqref{eq:thm3-1-1} reduces the equilibrium
dynamic problem into a static problem.

The key for the proof of Theorem \ref{thm:3.1} is the following static bound:

\begin{prop}  \label{prop:Phi}
For $\Phi = \Phi(\nabla h) \in L^2(\tilde{\mathcal{C}},\nu)$
such that $\|\Phi\|_{1,\e}<\infty$, and $\fa\in C_0(\R)$ satisfying the
condition of Theorem \ref{thm:3.1}, we have that
\begin{equation} \label{eq:B-1}
\left| E^{\nu^\e}\left[ B^\e(\fa,Y) \Phi\right] \right|
\le C(\fa) \sqrt{\e} \|\Phi\|_{1,\e},
\end{equation}
for all $0 < \e <1 \wedge \left(\frac14 \dist(\rm{supp}\,\fa,\rm{supp}\,\rho)\right)$
with some positive constant $C(\fa)$, which depends only on 
$\|\fa\|_\infty$ and the size of $\rm{supp}\,\fa$. In particular, taking $\Phi=1$,
we see $E^{\nu^\e}\left[ B^\e(\fa,Y) \right] =0$.
\end{prop}

As we pointed out, we will work with the height processes
and not with the Cole-Hopf transformed processes.
In this respect, $\left(\frac{\partial_x Y}Y \right)^2 -\xi^\e$
is transformed back to $(\partial_x h)^2 -\xi^\e$. However,
recalling that $\partial_x h=\partial_x(B*\eta^\e)$ under $\nu^\e$
in the stationary situation, by It\^o's formula, we have that
\begin{align}\label{eq:6.1}
(\partial_xh)^2 & = \left\{ \int_\R \eta^\e(x-y)dB(y)\right\}^2 \\
& = \Psi^\e(x) + \int_\R \eta^\e(x-y)^2dy = \Psi^\e(x)+\xi^\e,  \notag
\end{align}
where $\Psi^\e(x)$ is a Wiener functional of second order given by
\begin{align} \label{eq:2.Psi}
\Psi^\e(x) &
= \int_{\R^2} \eta^\e(x-x_1)\eta^\e(x-x_2)dB(x_1)dB(x_2)\\
&\equiv 2 \int_{x_1<x_2} \eta^\e(x-x_1)\eta^\e(x-x_2)dB(x_1)dB(x_2).  \notag
\end{align}
Therefore, $\left(\frac{\partial_x Y}Y \right)^2 -\xi^\e= \Psi^\e(x)$
and $A^\e(x,Y)$ is rewritten as
\begin{equation} \label{eq:10-b}
\frac12 Y(x) \{\Psi^\e*\eta_2^\e(x) - \Psi^\e(x)\}.
\end{equation}

\begin{rem} \label{rem:3.1}
{\rm (1)} The term $\Psi^\e*\eta_2^\e(x) - \Psi^\e(x)$ 
does not converge in a strong sense.  In fact, if we take 
$\eta(x) = \tfrac1{\sqrt{2\pi}} e^{-x^2/2}$ for simplicity, then
an explicit computation shows that
$$
E^{\nu^\e} [\{\Psi^\e*\eta_2^\e(x) - \Psi^\e(x)\}^2]
= \frac1{\pi \e^2} \left( \frac1{4\sqrt{5}}- \frac1{2\sqrt{3}}
+ \frac14 \right).
$$
{\rm (2)} We easily have that
$$
E^{\nu^\e} [\Psi^\e(x)^2] = \frac1{\e^2} \eta_2(0)^2.
$$
Comparing with {\rm (1)}, we see that taking the difference
does not  improve the magnitude in $\e$.\\
{\rm (3)} The term $\Psi^\e*\eta_2^\e(x) - \Psi^\e(x)$ converges
to $0$ in a weak sense.  More precisely, for every
$\Phi\in L^2(\nu)$ whose second order Wiener chaos has a
continuous kernel $\fa_2\in C(\R^2)\cap \hat{L}^2(\R^2)$, 
we have that
$$
\lim_{\e\downarrow 0}
E^{\nu^\e} [\{\Psi^\e*\eta_2^\e(x) - \Psi^\e(x)\}\Phi]=0.
$$
However, this is not sufficient to analyze
the limit of \eqref{eq:10-b} because of the extra $Y(x)$. \\
{\rm (4)} If one could have a bound on $\Psi^\e*\eta_2^\e(x) 
- \Psi^\e(x)$ in a Sobolev norm $\|\cdot\|_{H_r^{-\a}}$ with 
possibly $\a< 1/2$,
then one might be able to control the limit of \eqref{eq:10-b}.
However, we only have that
$$
E^{\nu^\e} \left[\left| \int_\R \{\Psi^\e*\eta_2^\e(x) - \Psi^\e(x)\}
\psi(x)dx \right|\right] \le C \e \|\partial_x^2\psi\|_\infty,
$$
for every $\psi\in C_b^2(\R)$.  This (with interpolation) roughly
implies
$$
E^{\nu^\e} \left[\| \Psi^\e*\eta_2^\e(x) - \Psi^\e(x)
\|_{H^{-\a}_r(\R)} \right] \le C \e^{\a-1},
$$
which is expected to converge to $0$ only if $\a>1$.
Under the multiplication of $Y^\e(x)$, which is roughly in
$C^{\frac12-\de}$ (uniformly in $\e$), the convergence of 
\eqref{eq:10-b} to $0$ cannot be expected. \\
{\rm (5)} Proposition \ref{prop:Phi} 
shows that $\|B^\e(\fa,Y)\|_{-1,\e} \le C(\fa)\sqrt{\e}$.
For its $L^2$-norm, we only have  $\|B^\e(\fa,Y)\|_{L^2(\Om)} 
\le C(\fa)\e^{-1/2}$.
Indeed, in the proof of the proposition stated below, we can estimate 
$|f_{n+2}-f_{n+2}| \le |f_{n+2}| + |f_{n+2}|$ to avoid to have their derivatives,
but we lose the factor $\e$ in doing so.
\end{rem}

The results mentioned in Remark \ref{rem:3.1} are  
not useful in our situation because of the extra term $Y(x)$.
Since $Y(x)$ is not a tilt variable, we need to consider
$\frac{Y(x)}{Y_\rho}$ instead as in $B^\e(x,Y)$.

Once Proposition \ref{prop:Phi} is shown, \eqref{eq:3.14-b} with $\Phi\equiv 1$
implies that $E^{\pi\otimes\nu^\e}\left[ \hat A^\e(\fa,Y) \right] =0$.
This means that we may assume $\Phi\in L^2(\pi\otimes\nu^\e)$
in the right hand side of \eqref{eq:thm3-1-1} satisfies
$E^{\pi\otimes\nu^\e}\left[ \Phi \right] =0$.  For such $\Phi$, 
we have that
\begin{align*}
\left| 2E^{\pi\otimes\nu^\e}
\left[ \hat A^\e(\fa,Y) \Phi \right] \right|
& \le E^\pi\left[  Y_\rho C(\fa) \sqrt{\e} \|\Phi(h(\rho),\cdot)\|_{1,\e}
\right] \\
& \le e \, C(\fa) \sqrt{\e}  \int_0^1 d\pi \left\{\tfrac12
E^{\nu^\e}\left[ \int_\R
\left( D\Phi(\cdot;h(\rho),\nabla h)*\eta^\e\right)^2(x)dx
\right] \right\}^{1/2}  \\
& \le e \, C(\fa) \sqrt{\e}  \left\{\tfrac12 
E^{\pi\otimes\nu^\e}\left[ \int_\R
\left( D\Phi(\cdot;h(\rho),\nabla h)*\eta^\e\right)^2(x)dx
\right] \right\}^{1/2}  \\
& \le e\, C(\fa) \sqrt{\e} \|\Phi\|_{1,\e},
\end{align*}
where the operator $D$ acts only on the tilt variable $\nabla h$.
We have used Proposition \ref{prop:Phi}, Lemma \ref{lem:3.4}
and $Y_\rho\in [1,e]$ for the second line, Schwarz's inequality for
the third line and Lemma  \ref{lem:3.6} stated below for the
fourth line.  Therefore, the right hand side of \eqref{eq:thm3-1-1}
is bounded by
$$
24 T \sup_{\a\in \R} \{ e C(\fa) \sqrt{\e}\a-\a^2\}
= 24 T (\tfrac{e}2\, C(\fa) \sqrt{\e})^2,
$$
in which we write $\a = \|\Phi\|_{1,\e}$.  This tends to $0$ as 
$\e\downarrow 0$, and concludes the proof of Theorem
\ref{thm:3.1}.

\begin{lem} \label{lem:3.6}
For $\Phi=\Phi(h(\rho),\nabla h)$ such that
$E^{\pi\otimes\nu^\e}\left[ \Phi \right] =0$,
\begin{align*}
\|\Phi\|_{1,\e}^2 = & \frac{\xi_\rho^\e}2 E^{\pi\otimes\nu^\e}\left[ 
\left( \frac{\partial\Phi}{\partial h(\rho)}(h(\rho),\nabla h)\right)^2
  \right] \\
& \quad +  \frac12 E^{\pi\otimes\nu^\e}\left[ \int_\R
\left( D\Phi(\cdot;h(\rho),\nabla h)*\eta^\e\right)^2(x)dx \right].
\end{align*}
In the above formula, $D$ acts only on the tilt variables and
$\xi_\rho^\e$ is defined in the proof of Lemma  \ref{lem:3.1-a}.
\end{lem}

\begin{proof}
Denote $D$ acting on $h=(h(\rho),\nabla h)$ by $\tilde D$ for
distinction.  Then, it can be expressed as
$$
\tilde D \Phi(x;h(\rho),\nabla h) = 
\frac{\partial\Phi}{\partial h(\rho)}(h(\rho),\nabla h) \rho(x)
+ D\Phi(x;h(\rho),\nabla h),
$$
so that Lemma \ref{lem:3.4} implies that
\begin{align*}
\|\Phi\|_{1,\e}^2 =  \frac12 E^{\pi\otimes\nu^\e}\left[ 
\int_\R \left\{ \frac{\partial\Phi}{\partial h(\rho)}(h(\rho),\nabla h)
 \rho*\eta^\e(x) 
+ D\Phi(\cdot;h(\rho),\nabla h)*\eta^\e(x)\right\}^2dx \right].
\end{align*}
We expand the square inside the integration, then the cross term
becomes
\begin{align} \label{eq:cross}
E^{\pi\otimes\nu^\e}\left[ 
\frac{\partial\Phi}{\partial h(\rho)}(h(\rho),\nabla h)
\int_\R \rho*\eta^\e(x) 
\big( D\Phi(\cdot;h(\rho),\nabla h)*\eta^\e\big)(x)dx \right] =0
\end{align}
and this shows the conclusion. \eqref{eq:cross} is shown, 
first for $\Phi$ of the form $\Phi = \sum_{i=1}^\ell f_i(h(\rho))
\Phi_i(\nabla h)$, where we choose $\{f_i(a)=\sqrt{2}\sin
\pi i a \}_{i=1}^\infty$ which is a complete orthonormal system
of $L^2([0,1],\pi)$.  Since $E^{\pi\otimes\nu^\e}\left[ \Phi \right] =0$,
we may assume $E^\pi[f_i]=0$ so that $i$ are even.
Indeed for such $\Phi$, the left hand side of
\eqref{eq:cross} can be rewritten as
\begin{align*}
\sum_{i,j=2: \text{ even}}^\ell \int_0^1 f_i' f_j d\pi \,
E^{\nu^\e}\left[ \Phi_i \int_\R \rho*\eta^\e(x) 
\big( D\Phi_j(\cdot;\nabla h)*\eta^\e\big)(x)dx \right].
\end{align*}
However, we easily see that
$
\int_0^1 f_i' f_j d\pi =0
$
for all even $i, j \ge 2$ and this proves \eqref{eq:cross}. The general
$\Phi \in L^2(\pi\!\otimes\!\nu^\e)$ can be
approximated by the functions of the above form.
\end{proof}

Only the proof of Proposition \ref{prop:Phi} is left.  Before
giving it, we recall the diagram formula in order to compute 
the second order chaos in the products of two
Wiener functionals $\frac{Y(x)}{Y_\rho}$ and $\Phi$.
Let $n_1, \ldots, n_m \in\N$ be given.  We call $\Ga$ the set of all 
diagrams $\ga$, which are collections of (undirected) edges connecting
vertices in $V :=\{(j,\ell); \ell=1,\ldots, m, j=1,\ldots, n_\ell\}$,
in such a way that each edge in $\ga$ connects two vertices $(j_1,\ell_1)$ 
and $(j_2,\ell_2)$ only when $\ell_1 \not= \ell_2$ and each 
vertex is attached to at most one edge.  We denote 
$N= \sum_{\ell=1}^m n_\ell$ and the number of edges in $\ga$ by
$|\ga|$.  We define the function $\fa_\ga\in \hat{L}^2(\R^{N-2|\ga|})$ 
as follows: We first introduce a function $\fa$ of $N$ variables $\{x_{j,\ell}\}$ by
$$
\fa(x_{j,\ell}, \ell=1,\ldots, m, j=1,\ldots, n_\ell)
:= \prod_{\ell=1}^m \fa_\ell(x_{j,\ell}, j=1,\ldots, n_\ell).
$$
We call the variables $\{x_{j,\ell}, \ell=1,\ldots, m, j=1,\ldots, n_\ell\}$ 
simply as $\{x_1,\ldots,x_N\}$ and call $\fa$ again its symmetrization.
Then, $\fa_\ga$ is defined from $\fa$ by truncating the last
 $2|\ga|$-variables:
$$
\fa(x_1,\ldots,x_{N-2|\ga|}) := \int_{\R^{|\ga|}} 
\fa(x_1,\ldots,x_{N-2|\ga|}, p_1,p_1, \ldots. p_{|\ga|},p_{|\ga|})
dp_1\cdots dp_{|\ga|}.
$$
Then, we have the following diagram formula; see Major
\cite{Ma}, Section 5, under a slightly different setting.

\begin{lem}  
For $\fa_1\in \hat{L}^2(\R^{n_1}), \ldots, \fa_m\in \hat{L}^2(\R^{n_m})$
with $n_1, \ldots, n_m \ge 1$, we have
$$
I(\fa_1) \cdots I(\fa_m) = \sum_{\ga\in \Ga} \frac{(N-2|\ga|)!}
{n_1! \cdots n_m!} I (\fa_\ga).
$$
\end{lem}

We are now ready to give the proof of Proposition \ref{prop:Phi}.

\begin{proof}[Proof of Proposition \ref{prop:Phi}]
We first notice that, under $\nu$, $\frac{Y(x)}{Y_\rho}$ has an expression:
$$
\frac{Y(x)}{Y_\rho} = e^{B(x)-\int_\R B(y)\rho(y)dy},
$$
and the exponent can be rewritten as
$$
B(x)-\int_\R B(y)\rho(y)dy= \int_\R \phi_x(u)dB(u),
$$
where
\begin{equation}  \label{eq:3.17-x}
\phi_x(u) = 1_{(-\infty,x]}(u) + \th(u), \quad \th(u) = - \int_u^\infty\rho(y)dy.
\end{equation}
Note that, from the condition of $\rho$, $\phi_x(\cdot)$ has a compact support: 
supp $\phi_x \subset [x\wedge(-1),x\vee 1]$, $|\phi_x(u)|\le 1$, has
a jump only at $u=x$ and $\th$ is smooth.  Therefore, $\frac{Y(x)}{Y_\rho}$ has 
the following Wiener-It\^o expansion under $\nu$:
\begin{align} \label{eq.WIexp-Y}
\frac{Y(x)}{Y_\rho} 
= e^{a(x)} \left\{ 1 + \sum_{n=1}^\infty  \frac1{n!}\int_{\R^n}
  \phi_x^{\otimes n} (u_1,\ldots,u_n) dB(u_1) \cdots dB(u_n)\right\},
\end{align}
where $a(x) = \tfrac12 \int_\R \phi_x^2(u)du$.  In fact, one can apply the well-known
result for the expansion of exponential martingales 
written, e.g., in \cite{KS}, p.167 for $M_t= \int_{-\infty}^t\phi_x(u)dB(u)$,
and letting $t\to\infty$.

Since $\Psi^\e$ and $\Psi^\e*\eta^\e$ are both second order Wiener chaoses, 
to compute the expectation
\begin{equation}\label{eq:Exp}
E^{\nu^\e} \left[\{\Psi^\e*\eta_2^\e(x)
- \Psi^\e(x)\} \frac{Y(x)}{Y_\rho} \Phi \right],
\end{equation}
what we need to obtain is the kernel $\bar\fa_2(x_1,x_2)$
of the second order Wiener chaos in the product $\Phi\cdot
\frac{Y(x)}{Y_\rho}$.  Denoting the kernel of the $n$th order Wiener
chaos in $\frac{Y(x)}{Y_\rho}$ except the factor $e^{a(x)}$
by $\psi_n(u_1,\cdots,u_n) =  
\phi_x^{\otimes n}(u_1,\cdots,u_n)$, in the expansion of the product
\begin{equation}  \label{eq:3.product}
\Phi\cdot \frac{Y(x)}{Y_\rho} = e^{a(x)}\left( \sum_{m_1=0}^\infty I(\fa_{m_1})\right)
 \left( \sum_{m_2=0}^\infty I(\psi_{m_2})\right)
= e^{a(x)}\sum_{m_1,m_2=0}^\infty I(\fa_{m_1})  I(\psi_{m_2}),
\end{equation}
we apply the diagram formula to get the explicit formula for 
$\bar\fa_2(x_1,x_2)$:
\begin{align*}
\bar\fa_2(x_1,x_2)&= e^{a(x)} \sum_{n=0}^\infty 
\begin{pmatrix} n+2 \\ 2 \end{pmatrix} n!
\times
\frac{2!}{(n+2)!n!} \int_{\R^n} \fa_{n+2}(\uu,x_1,x_2) 
\psi_n(\uu) d\uu \\
&\;\; + e^{a(x)} \sum_{n=0}^\infty 
\begin{pmatrix} n+2 \\ 2 \end{pmatrix} n!
\times
\frac{2!}{(n+2)!n!} \int_{\R^n} \fa_n(\uu) 
\psi_{n+2}(\uu,x_1,x_2) d\uu \\
&\;\; +e^{a(x)} \sum_{n=0}^\infty (n+1)^2 n!\times
\frac{2!}{(n+1)!(n+1)!} \int_{\R^n} \fa_{n+1}(\uu,x_1) 
\psi_{n+1}(\uu,x_2) d\uu \\
&\;\; +e^{a(x)} \sum_{n=0}^\infty (n+1)^2 n!\times
\frac{2!}{(n+1)!(n+1)!} \int_{\R^n} \fa_{n+1}(\uu,x_2) 
\psi_{n+1}(\uu,x_1) d\uu \\
&=: \bar\fa_2^{(1)}(x_1,x_2) + \bar\fa_2^{(2)}(x_1,x_2) +
\bar\fa_2^{(3)}(x_1,x_2) + \bar\fa_2^{(4)}(x_1,x_2),
\end{align*}
where we denote $\uu = (u_1,\ldots,u_n)$ and $d\uu = du_1\cdots du_n$.
Note that, to obtain the second order term, $(m_1,m_2)$ which
we need to take care are only of the forms $\{(n+2,n), (n,n+2),
(n+1,n+1), n \ge 0\}$, in which case $N= 2n+2$, and we may only 
consider $\ga$ satisfying $N-2|\ga| = 2$, i.e., $|\ga|=n$.  For example,
when $(m_1,m_2) = (n+2,n)$, the prefactor in the diagram formula
becomes as above, since $|\Ga| = \begin{pmatrix} n+2 \\ 2 \end{pmatrix}n!$,
and when $(m_1,m_2)=(n+1,n+1)$, $|\Ga|=(n+1)^2n!$

In the rest, we will show that the contributions of $\bar\fa_2^{(1)}$,
$\bar\fa_2^{(3)}$, $\bar\fa_2^{(4)}$ to the expectation \eqref{eq:Exp}
are small, while that of $\bar\fa_2^{(2)}$ exactly cancels with 
$\frac1{12}  E^{\nu^\e} \left[\frac{Y(x)}{Y_\rho} \Phi \right]$,
when integrated with a test function $\fa=\fa(x)$.

Under $\nu^\e$, the kernels $\fa_n$ and $\psi_n$ are replaced by
$\fa_n*(\eta^\e)^{\otimes n}$ and $\psi_n*(\eta^\e)^{\otimes n}$,
respectively, where $\fa_n*(\eta^\e)^{\otimes n}$ is defined by
$$
\fa_n*(\eta^\e)^{\otimes n}(\uu)
:= \int_{\R^n} \fa_n(\underline{v}) \eta^\e(u_1-v_1)\cdots
\eta^\e(u_n-v_n) d\underline{v},
$$
where $\underline{v}= (v_1,\ldots,v_n)$ and $d\underline{v}=
 dv_1\cdots dv_n$.  Recall that $\Psi^\e*\eta_2^\e(x) - \Psi^\e(x)$
is a second order Wiener chaos with the kernel:
\begin{equation}\label{eq:ker-a}
2 \left\{ \int_\R \eta^\e(y-x_1) \eta^\e(y-x_2) \eta_2^\e(x-y) dy
- \eta^\e(x-x_1) \eta^\e(x-x_2) \right\}.
\end{equation}

Then, the contribution of the first term $\bar\fa_2^{(1)}$ to the
expectation \eqref{eq:Exp}
is given by, neglecting the factor $e^{a(x)} \sum_{n=0}^\infty
\frac1{n!}$ ($2$ in \eqref{eq:ker-a} cancels with $\frac12$
appearing in \eqref{eq:2.6} for $n=2$),
\begin{align*}
& \int_{\R^2}dx_1dx_2 \left\{
\int_{\R^n} \left( \fa_{n+2}*(\eta^\e)^{\otimes (n+2)} \right)
(\uu,x_1,x_2) \left(\phi_x^{\otimes n}*(\eta^\e)^{\otimes n}\right)
(\uu) d\uu \right\}  \\
& \qquad \times 
\left\{ \int_\R \eta^\e(y-x_1) \eta^\e(y-x_2) \eta_2^\e(x-y) dy
- \eta^\e(x-x_1) \eta^\e(x-x_2) \right\} \\
& = \int_{\R^n} \phi_x^{\otimes n}(\uu) d\uu \left\{ \int_\R \left(\fa_{n+2}*(\eta_2^\e)^{\otimes (n+2)} \right)
(\uu,y,y) \eta_2^\e(x-y)dy \right. \\
& \qquad\qquad\qquad\qquad\qquad \left. \phantom{\int_{\R^n}}
-  \left(\fa_{n+2}*(\eta_2^\e)^{\otimes (n+2)} \right)(\uu,x,x)\right\}.
\end{align*}
In the above computation, we first move $\eta^\e$ in 
$\phi_x^{\otimes n}*(\eta^\e)^{\otimes n}$ to $\fa_{n+2}*(\eta^\e)^{\otimes (n+2)}$, which gives $(\fa_{n+2}*(\eta_2^\e)^{\otimes n}*(\eta^\e)^{\otimes 2})(\uu,x_1,x_2)$.  Then, we integrate in $x_1$ and $x_2$ and obtain the above formula.

Therefore, the contribution of $\bar\fa_2^{(1)}$ to the
left hand side of \eqref{eq:B-1} is given by
$
\sum_{n=0}^\infty \frac1{n!} I_n^{(1)},
$
where
\begin{align*}
I_n^{(1)} := & \int_\R \fa(x) e^{a(x)}dx \int_{\R^n} \phi_x^{\otimes n}
 (\uu) d\uu\\
& \times \left\{ \int_R (f_{n+2}*(\eta^\e)^{\otimes (n+2)})
(\uu,y,y) \eta_2^\e(x-y)dy - (f_{n+2}*(\eta^\e)^{\otimes (n+2)})
(\uu,x,x) \right\},
\end{align*}
with $f_{n+2} :=  \fa_{n+2}*(\eta^\e)^{\otimes (n+2)}$.
We have rewritten as $\fa_{n+2}*(\eta_2^\e)^{\otimes (n+2)}=
f_{n+2}*(\eta^\e)^{\otimes (n+2)}$ in the above formula.
The difference in the braces in the right hand side of
 $I_n^{(1)}$ can be expressed as the
expectation:
$$
E[f_{n+2}(\uu+U^\e, x+Y^\e+X_1^\e, x+Y^\e+X_2^\e) -
f_{n+2}(\uu+U^\e, x+X_1^\e, x+X_2^\e)],
$$
where $U^\e = (U_i^\e)_{i=1}^n$ is an $\R^n$-valued random
variable with independent components $U_i^\e$ distributed
under $\eta^\e(y)dy$ ($\eta^\e$ in short), $X_1^\e, X_2^\e$
are $\R$-valued random variables distributed under $\eta^\e$,
$Y^\e$ is an  $\R$-valued random variable distributed under
$\eta_2^\e$, and all random variables are mutually independent.
We estimate
\begin{align*}
& | f_{n+2}(\uu+U^\e, x+Y^\e+X_1^\e, x+Y^\e+X_2^\e) -
f_{n+2}(\uu+U^\e, x+X_1^\e, x+X_2^\e) | \\
& =\left| \int_0^1  \frac{\partial}{\partial \la}
 f_{n+2}(\uu+U^\e, x+\la Y^\e+X_1^\e, x+\la Y^\e+X_2^\e) d\la\right|     \\
& =\left| \int_0^1  Y^\e \sum_{k=n+1}^{n+2} \partial_k
 f_{n+2}(\uu+U^\e, x+\la Y^\e+X_1^\e, x+\la Y^\e+X_2^\e) d\la
  \right|\\
& \le 4\e \int_0^1 \left| \partial_{n+2}
 f_{n+2}(\uu+U^\e, x+\la Y^\e+X_1^\e, x+\la Y^\e+X_2^\e) \right|
d\la, \quad \text{ a.s.}.
\end{align*}
Note that supp$\,\eta \subset [-1,1]$ implies 
supp$\,\eta_2^\e \subset [-2\e,2\e]$ and therefore
$|Y^\e| \le 2\e$ a.s.  We also used the symmetry of $f_{n+2}$
in $(n+2)$-variables.  Accordingly, since $\fa\in C_0(\R)$
implies supp $\fa\subset [-K,K]$ with some $K\ge 1$ and
supp $\phi_x \subset [x\wedge(-1),x\vee 1] \subset [-K,K]$
for $x\in [-K,K]$, we have
\begin{align*}
|I_n^{(1)}| & \le 4\e \|\fa e^{a(x)}\|_\infty \int_{-K}^K dx \int_{\R^n}     
  1_{[-K,K]^n} (\uu) d\uu  \\
& \qquad \times \int_0^1 E[|\partial_{n+2}
 f_{n+2}(\uu+U^\e, x+\la Y^\e+X_1^\e, x+\la Y^\e+X_2^\e)|] d\la\\
& \le 4\e \|\fa e^{a(x)}\|_\infty \int_{-K-3\e}^{K+3\e} dx \int_{\R^n}     
  1_{[-K-\e,K+\e]^n} (\uu) d\uu  \\
& \qquad \times E[|\partial_{n+2}
 f_{n+2}(\uu, x, x+X_2^\e-X_1^\e)|].
\end{align*}
The last line is obtained by putting the integrals in $(x,\uu)$
inside the expectation and then applying the change of variables $\uu'=\uu+U^\e$ and $x'=x+\la Y^\e+X_1^\e$.  We enlarge
the domains of the integrations a little bit.  Since $X_2^\e-X_1^\e$
is distributed under $\eta_2^\e$, this is equal to
\begin{align*}
4\e \|\fa e^{a(x)}\|_\infty \int_{\R^{n+2}}  1_{[-K-3\e,K+3\e]}(x)
  1_{[-K-\e,K+\e]^n} (\uu) 
 |\partial_{n+2} f_{n+2}(\uu, x, x+y)| d\uu dx \eta_2^\e(y)dy.
\end{align*}
Apply Schwarz's inequality, and we obtain that
\begin{align*}
|I_n^{(1)}| & \le 4\e \|\fa e^{a(x)}\|_\infty 
  \left\{ \int_{\R^{n+2}}  1_{[-K-3\e,K+3\e]}(x)
  1_{[-K-\e,K+\e]^n} (\uu)  d\uu dx \eta_2^\e(y)dy \right\}^{1/2}  \\
& \qquad \times  \left\{ \int_{\R^{n+2}} 
   |\partial_{n+2} f_{n+2}(\uu, x, x+y)|^2 d\uu dx \eta_2^\e(y)dy
 \right\}^{1/2}  \\
& \le 4 \sqrt{\e} \|\fa e^{a(x)}\|_\infty \|\eta_2\|_\infty^{1/2}
  (2K+6)^{(n+1)/2} \left\{\int_{\R^{n+2}} 
   |\partial_{n+2} f_{n+2}(\uu, x, y)|^2 d\uu dx dy
 \right\}^{1/2},
\end{align*}
since $0<\e<1$.  We have used a rough estimate: $|\eta_2^\e(y)| \le 
\|\eta_2\|_\infty/\e$.  Thus, we obtain that
\begin{align*}
\sum_{n=0}^\infty \frac1{n!} |I_n^{(1)}| & 
\le C \sqrt{\e} \sum_{n=0}^\infty \frac{ (2K+6)^{(n+1)/2}}{n!} 
   \left\{\int_{\R^{n+2}} 
   |\partial_{n+2} f_{n+2}|^2 d\ux \right\}^{1/2} \\
& \le C \sqrt{\e} \left\{\sum_{n=0}^\infty \frac{ (n+1)^2(2K+6)^{(n+1)}}{(n+1)!} 
   \right\}^{1/2} \left\{\sum_{n=0}^\infty \frac1{(n+1)!} 
   \int_{\R^{n+2}}    |\partial_{n+2} f_{n+2}|^2 d\ux \right\}^{1/2} \\
& \le C' \sqrt{\e}  \|\Phi\|_{1,\e},
\end{align*}
by Schwarz's inequality and Lemma \ref{lem:3.5}.

Next, the contribution of the second term $\bar\fa_2^{(2)}$ to the
expectation \eqref{eq:Exp} is given by, again neglecting the factor $e^{a(x)} \sum_{n=0}^\infty\frac1{n!}$,
\begin{align*}
& \int_{\R^2}dx_1dx_2 \left\{
\int_{\R^n} \left( \fa_{n}*(\eta^\e)^{\otimes n} \right)
(\uu) \left(\phi_x^{\otimes (n+2)}*(\eta^\e)^{\otimes (n+2)}\right)
(\uu,x_1,x_2) d\uu \right\}  \\
& \qquad\qquad \times 
\left\{ \int_\R \eta^\e(y-x_1) \eta^\e(y-x_2) \eta_2^\e(x-y) dy
- \eta^\e(x-x_1) \eta^\e(x-x_2) \right\} \\
& = \int_{\R^n} \left(f_{n}*(\eta^\e)^{\otimes n} \right)
(\uu) \phi_x^{\otimes n}(\uu) d\uu \\
& \times \left\{ \int_\R 
 \left(\phi_x^{\otimes 2}*(\eta_2^\e)^{\otimes 2}\right)(y,y)
 \eta_2^\e(x-y)dy - \left(\phi_x^{\otimes 2}*(\eta_2^\e)^{\otimes 2}\right)(x,x)
\right\}.
\end{align*}
Here, the last term in the braces can be represented by means of expectations:
\begin{align*}
& \int_\R  \left(\phi_x^{\otimes 2}*(\eta_2^\e)^{\otimes 2}\right)(y,y)
 \eta_2^\e(x-y)dy - \left(\phi_x^{\otimes 2}*(\eta_2^\e)^{\otimes 2}\right)(x,x)\\
&= E[\phi_x^{\otimes 2}(x+R_1+R_3, x+R_2+R_3)]
- E[\phi_x^{\otimes 2}(x+R_1, x+R_2)],
\end{align*}
where $\{R_1, R_2, R_3\}$ are independent random
variables distributed under $\eta_2^\e$.  By expanding
$$
\phi_x^{\otimes 2}= 1_{(-\infty,x]}^{\otimes 2} +1_{(-\infty,x]}\otimes \th
+ \th \otimes 1_{(-\infty,x]}+ \th^{\otimes 2},
$$
the above difference of two expectations can be rewritten as
$$
J_1+2J_2^\e(x)+J_3^\e(x),
$$
where
\begin{align*}
& J_1 = E[1_{(-\infty,0]}^{\otimes 2}(R_1+R_3, R_2+R_3)]
-  E[1_{(-\infty,0]}^{\otimes 2}(R_1, R_2)] \\
& J_2^\e(x) = E[1_{(-\infty,0]}(R_1+R_3) \th(x+ R_2+R_3)]
-  E[1_{(-\infty,0]}(R_1)\th(x+R_2)] \\
& J_3^\e(x) = E[\th(x+R_1+R_3) \th(x+ R_2+R_3)]
-  E[\th(x+R_1)\th(x+R_2)].
\end{align*}
However, we see that ${\rm supp}\, J_2^\e \subset ({\rm supp}\, \rho)^{4\e}
:= \{y\in\R; |y-x|\le 4\e \text{ for some } x\in {\rm supp}\, \rho\}$ and
${\rm supp}\, J_3^\e \subset ({\rm supp}\, \rho)^{4\e}$; recall that
${\rm supp}\, \rho$ is connected and $E[1_{(-\infty,0]}(R_1+R_3)] =
E[1_{(-\infty,0]}(R_1)] = \frac12$ by the symmetry of $R_i$ for $J_2^\e(x)$.
Therefore, from our assumption: $\dist({\rm supp}\, \fa, {\rm supp}\, \rho)>4\e$,
$\fa(x) J_2^\e(x) = \fa(x) J_3^\e(x) =0$
for all $x\in\R$.  On the other hand, we have
$$
J_1=\frac1{12}.
$$
In fact, by the symmetry of $\eta_2^\e$,
\begin{align}\label{eq:1/3}
P(R_1+R_3>0, R_2+R_3>0)
& = P(R_1-R_3>0, R_2-R_3>0)  \\
& = P \left(R_3 = \min_{i=1,2,3} R_i\right) = \tfrac13, \notag
\end{align}
for the first expectation and the second one is $\frac14$
as easily seen. The constant $\frac1{12}$ is
universal in the sense that it does not depend on the specific
distributions of independent random variables $\{R_1, R_2, R_3\}$
if they are symmetric (and have densities).  
Summarizing these, we see that the contribution of $\bar\fa_2^{(2)}$
to the expectation \eqref{eq:Exp}, when multiplied by $\fa$, is given by
\begin{equation}\label{eq:ker-b}
\tfrac1{12}e^{a(x)} \sum_{n=0}^\infty \frac1{n!} 
\int_{\R^n} \left(f_{n}*(\eta^\e)^{\otimes n} \right)(\uu)
 \phi_x^{\otimes n}(\uu) d\uu.
\end{equation}

However, the series in \eqref{eq:ker-b} just cancels with the expectation 
$E^{\nu^\e} \left[\frac{Y(x)}{Y_\rho} \Phi \right]$ divided by $12$. Indeed,
this expectation can be computed by picking up the $0$th order term in the product \eqref{eq:3.product}, and it is again an application of the diagram formula.  Indeed, we may take care $(m_1,m_2)$ of the forms $(n,n), n\ge 0$ only, in which case $N=2n$, and may consider only
$\ga$ such that $|\ga|=n$, i.e., diagrams connecting all vertices of the forms $(j_1,\ell_1)$ and $(j_2,\ell_2)$ with $\ell_1\not= \ell_2$.  The number of such $\ga$'s is given by $|\Ga|=n!$.  Thus, we have
\begin{align*}
E^{\nu^\e} \left[\frac{Y(x)}{Y_\rho} \Phi \right]
& = e^{a(x)} \sum_{n=0}^\infty \frac1{n!} 
\int_{\R^n} \left( \fa_{n}*(\eta^\e)^{\otimes n} \right)
(\uu) \left(\phi_x^{\otimes n}*(\eta^\e)^{\otimes n}\right)(\uu) d\uu  \\
&  = e^{a(x)} \sum_{n=0}^\infty \frac1{n!} 
\int_{\R^n} \left(f_{n}*(\eta^\e)^{\otimes n} \right)
(\uu) \phi_x^{\otimes n}(\uu) d\uu,
\end{align*}
which is exactly the same series in \eqref{eq:ker-b} except
the factor $\frac1{12}$.  

The contribution from the third term $\bar\fa_2^{(3)}$ to the
expectation \eqref{eq:Exp} is given by, neglecting the factor
$a(x) \sum_{n=0}^\infty \frac{2}{n!}$, 
\begin{align*}
& \int_{\R^2}dx_1dx_2 \left\{
\int_{\R^n} \left( \fa_{n+1}*(\eta^\e)^{\otimes (n+1)} \right)
(\uu,x_1) \left( \phi_x^{\otimes (n+1)}*(\eta^\e)^{\otimes (n+1)}\right)
(\uu,x_2) d\uu \right\}  \\
& \qquad\qquad \times 
\left\{ \int_\R \eta^\e(y-x_1) \eta^\e(y-x_2) \eta_2^\e(x-y) dy
- \eta^\e(x-x_1) \eta^\e(x-x_2) \right\} \\
& = \int_{\R^n} \phi_x^{\otimes n}(\uu) d\uu 
\left\{ \int_\R 
 \left(\fa_{n+1}*(\eta_2^\e)^{\otimes (n+1)} \right)
(\uu,y) \left(\phi_x*\eta_2^\e\right)(y)
 \eta_2^\e(x-y)dy  \right. \\
& \qquad \qquad\qquad\qquad  \left. \phantom{\int_\R }
- \left(\fa_{n+1}*(\eta_2^\e)^{\otimes (n+1)} \right)
(\uu,x) \left(\phi_x*\eta_2^\e\right)(x) \right\}.
\end{align*}
However, we see that
\begin{align*}
& \int_\R  \left(\phi_x*\eta_2^\e\right)(y)
 \eta_2^\e(x-y)dy  = E[\phi_x(x+R_1+R_2)]  = \frac12+\th*\eta_4^\e(x), \\
& \phi_x*\eta_2^\e(x)  = E[\phi_x(x+R_1)]  = \frac12+\th*\eta_2^\e(x),
\end{align*}
where $\{R_1, R_2\}$ are independent random variables
distributed under $\eta_2^\e$, from the symmetry of $R_i$.

Therefore, the contribution of $\bar\fa_2^{(3)}$ in the
left hand side of \eqref{eq:B-1} is given by
$
\sum_{n=0}^\infty \frac2{n!} I_n^{(3)},
$
where
\begin{align*}
I_n^{(3)} := & \int_\R \fa(x) e^{a(x)}dx \int_{\R^n} \phi_x^{\otimes n}
 (\uu) d\uu\\
& \; \times \int_\R \left\{ f_{n+1}*(\eta^\e)^{\otimes (n+1)}
(\uu,y)  - f_{n+1}*(\eta^\e)^{\otimes (n+1)}(\uu,x) \right\}
(\phi_x*\eta_2^\e)(y) \eta_2^\e(x-y)dy  \\
& + \int_\R \fa(x) e^{a(x)} \{\th*\eta_4^\e(x)- \th*\eta_2^\e(x)\}dx 
\int_{\R^n}  f_{n+1}*(\eta^\e)^{\otimes (n+1)}(\uu,x) \phi_x^{\otimes n}
 (\uu) d\uu.
\end{align*}
The second term in $I_n^{(3)}$
vanishes from our assumption by noting that
${\rm supp}\{\th*\eta_4^\e(x)- \th*\eta_2^\e(x)\} \subset
({\rm supp} \,\rho)^{4\e}$.  For the first term, 
estimating $|(\phi_x*\eta_2^\e)(y)| \le 1$, the absolute value of 
the last integral in $y$ can be bounded by
$$
E[|f_{n+1}(\uu+U^\e, x+Y^\e+X_1^\e) -f_{n+1}(\uu+U^\e, x+X_1^\e)|],
$$
where $U^\e, X_1^\e, Y^\e$ are the same as before, and this can be estimated further by
\begin{align*}
2\e \int_0^1 E\left[ \left| \partial_{n+1}
 f_{n+1}(\uu+U^\e, x+\la Y^\e+X_1^\e) \right|\right] d\la.
\end{align*}
Thus,
\begin{align*}
|I_n^{(3)}| & \le 2\e \|\fa e^{a(x)}\|_\infty \int_{\R^{n+1}}
1_{[-K-3\e,K+3\e]}(x) 1_{[-K-\e,K+\e]^n} (\uu) 
  |\partial_{n+1} f_{n+1}(\uu, x)| d\uu dx \\
& \le 2\e \|\fa e^{a(x)}\|_\infty (2K+6)^{(n+1)/2}
\left\{ \int_{\R^{n+1}}  |\partial_{n+1} f_{n+1}(\ux)|^2 d\ux
\right\}^{1/2},
\end{align*}
by Schwarz's inequality.  Therefore,
we get that
\begin{align*}
\sum_{n=0}^\infty \frac2{n!} |I_n^{(3)}| 
& \le C\e \sum_{n=0}^\infty \frac1{n!} (2K+6)^{(n+1)/2}
\left\{ \int_{\R^{n+1}}  |\partial_{n+1} f_{n+1}(\ux)|^2 d\ux
\right\}^{1/2}  \\
& \le C'\e \|\Phi\|_{1,\e}.
\end{align*}
The contribution of $\bar\fa_2^{(4)}$ can be estimated similarly,
and this concludes the proof of the proposition.
\end{proof}

\begin{rem} \label{rem:3.3}
{\rm (1)} 
The assumption that ${\rm supp}\, \fa$ and
${\rm supp}\, \rho$ separate is needed to treat the terms 
$\bar\fa_2^{(3)}$  and $\bar\fa_2^{(4)}$.  For $\bar\fa_2^{(2)}$,
this assumption is unnecessary by changing $\frac1{12}$
into $\frac1{12} + 2J_2^\e(x) + J_3^\e(x)$ in the definition of
$B^\e(x,Y)$. \\
{\rm (2)} Due to the symmetry of $\eta(x)$, we obtain the constant
$\frac1{12}$.  For general asymmetric $\eta$, if it satisfies
$\int_0^\infty \eta_4(y)dy= \int_0^\infty \eta_2(y)dy$
(i.e., $P(R_1+R_2>0) = P(R_1>0)$ to treat $\bar\fa_2^{(3)}$ and
$\bar\fa_2^{(4)}$), this constant $\frac1{12}$ is replaced by
$$
J_1 = P(R_1+R_3>0, R_2+R_3>0) - P(R_1>0)^2,
$$
where $\{R_1, R_2, R_3\}$ are independent random variables
distributed under $\eta_2(y)dy$.  For example, if 
$\text{supp}\,\eta\subset [0,\infty)$ (or $(-\infty,0]$), then
$R_1, R_2, R_3>0$ a.s.\ and we have $J_1=0$.
\end{rem} 

\subsection{SPDE on $\SS_M$}  \label{section:3.4}

The arguments developed in Sections 
\ref{section:3.1}--\ref{section:3.3} work for the
SPDE's on $\SS_M, M\ge 2,$ instead of $\R$ in a similar
way.  We outline it.  Let $h^{\e,M}(t,x)$, $x\in \SS_M$ be
the solution of the SPDE \eqref{eq:1-M}.
It is periodically extended on $\R$.  Then, the Cole-Hopf
transformed process $Z^{\e,M}(t,x) = e^{h^{\e,M}(t,x)}$
satisfies the SPDE \eqref{eq:8-M}.
Taking $\rho$ as in Section \ref{sec:3.2}, we consider
wrapped processes $g^{\e,M}(t,x)$ and $Y^{\e,M}(t,x)$ of
$h^{\e,M}(t,x)$ and $Z^{\e,M}(t,x)$, respectively. Note that
the integral $h^{\e,M}(t,\rho)  =\int_{\SS_M}h^{\e,M}(t,x)
\rho(x)dx$ is defined in a periodic sense; that is, 
$\int_{-M/2}^{M/2} h^{\e,M}(t,x)\rho(x)dx$ recalling that $M\ge 2$.

We take the initial distribution of  $h^{\e,M}(0,\cdot)$ to be
$\pi\otimes\nu^{\e,M}$ under the map \eqref{eq:3.1-map}
with $\R$ replaced by $\SS_M$.  Then, the probability measure
$\pi\otimes \nu^{\e,M}$ is
invariant under $g^{\e,M}(t,x)$ as in Lemma \ref{lem:3.1-a}.
Lemma \ref{lem:3.2-R} is also parallel:

\begin{lem}  \label{lem:3.2-R-M}
$Y^{\e,M}(t,x)$ satisfies the following equation in generalized functions' sense:
\begin{align}  \label{eq:3-Y^e-2-M}
Y^{\e,M}(t,x) =  Y^{\e,M}(0,x) + & \frac12 \int_0^t \partial_x^2 Y^{\e,M}(s,x)ds 
+\int_0^t A^\e(x,Y^{\e,M}(s)) ds \\
 & + \int_0^t Y^{\e,M}(s,x) dW^\e(s,x) + N^{\e,M}(t,x),  \quad x\in \SS_M, \notag
\end{align}
where $A^\e(x,Y)$ is defined by \eqref{eq:3.4} with convolution $*\eta_2^\e$ 
considered in a periodic sense, 
\begin{align} \label{eq:3.5-M}
N^{\e,M}(t,x) & =
\int_0^t \left\{ (e-1) 1_{\{ Y_\rho^{\e,M}(s-)=1\}}
+  (e^{-1}-1) 1_{\{ Y_\rho^{\e,M}(s-)=e\}} \right\}\\
&\hspace{5cm} \times Y^{\e,M}(s-,x) N^{\e,M}(ds),  \notag
\end{align}
and $N^{\e,M}(t) = -[h^{\e,M}(t,\rho)]$.
\end{lem}

The limit $\nu$ of $\nu^\e$ as $\e\downarrow 0$ was 
identified with the distribution of Gaussian random measure
$X=\{X(A); A\in \mathcal{B}(\R)\}$.  The limit $\nu^M$
of $\nu^{\e,M}$ as $\e\downarrow 0$ is nothing but
$X=\{X(A); A\in \mathcal{B}([0,M))\}$ restricted on
$[0,M)$ and conditioned to be $X([0,M))=0$.  Such conditional
random variables $X^M=\{X^M(A); A\in \mathcal{B}([0,M))\}$ 
can be realized by $X^M((a,b]) =B^M(b)-B^M(a)$,
$0\le a<b\le M$, where $B^M = \{B^M(x);x\in [0,M)\}$ is
the pinned Brownian motion such that $B^M(0)=B^M(M)=0$.

The Wiener-It\^o expansion \eqref{eq:11-b} is modified in the following
way under $\nu^M$:  Taking a usual standard Brownian motion
$B=\{B(x); x\in [0,M]\}$ such that $B(0)=0$, the pinned Brownian
motion is realized as
\begin{equation}  \label{eq:3.B^M}
B^M(x) = B(x)- \frac{x}MB(M), \quad x\in [0,M],
\end{equation}
so that $B^M(0)=B^M(M)=0$.  By applying It\^o's formula and
noting $B(M) = \int_0^MdB(x)$, the multiple Wiener integral
with respect to $B^M$ can be rewritten into that with respect
to $B$ as follows:
\begin{align*}
I^M(\fa_n) := & \frac1{n!} \int_{\SS_M^n} \fa_n(x_1,\ldots,x_n) dB^M(x_1)\cdots dB^M(x_n) \\
=& \frac1{n!} \int_{[0,M)^n} \fa_n^M(x_1,\ldots,x_n) dB(x_1)\cdots dB(x_n) = I(\fa_n^M),
\end{align*}
where $\fa_n^M$ is defined by
\begin{align}  \label{eq:3.fa^M}
\fa_n^M(x_1,\ldots,x_n) & = \sum_{k=0}^n \frac{(-1)^k}{M^k} \sum_{1\le i_1<\cdots<i_k\le n}
\int_{[0,M)^k} \check{\fa}_{n; i_1,\ldots,i_k} dy_{i_1}\cdots dy_{i_k},  \\
\check{\fa}_{n; i_1,\ldots,i_k} &= \fa_n(x_1,\ldots,x_{i_1-1},y_{i_1},x_{i_1+1},\ldots,x_{i_k-1},y_{i_k},x_{i_k+1},
\ldots,x_n).  \notag
\end{align}
Note that $\fa_n^M$ is symmetric in $\ux=(x_1,\ldots,x_n)$ and satisfies
\begin{equation}  \label{eq:fa_n=0}
\int_{[0,M)} \fa_n^M(x_1,\ldots,x_n)dx_i=0,
\end{equation}
for every $1\le i \le n$.  Thus, $\Phi \in L^2(\tilde{\mathcal{C}}_M,\nu^M)$
has a Wiener-It\^o expansion
\begin{equation*}
\Phi = \sum_{n=0}^\infty I(\fa_n^M),
\end{equation*}
with some $\fa_0\in\R$ and $\fa_n^M\in \hat{L}_0^2([0,M)^n)
:= \{\fa_n^M\in \hat{L}^2([0,M)^n); \fa_n^M$ satisfies the condition
\eqref{eq:fa_n=0}$\}$.  The quotient space $\tilde{\mathcal{C}}_M$
is defined similarly to $\tilde{\mathcal{C}}$.

Note that, for two symmetric functions $f_n=f_n(\ux)$ and $\fa_n=\fa_n(\ux)$,
from \eqref{eq:3.fa^M}, we see that
\begin{equation}  \label{eq:P^M-dual}
\int_{[0,M)^n} f_n^M(\ux)\fa_n(\ux)d\ux= 
\int_{[0,M)^n} f_n(\ux)\fa_n^M(\ux)d\ux,
\end{equation}
and $(\fa_n^M)^M = \fa_n^M$, or $\fa_n^M =\fa_n$ for 
$\fa_n\in \hat{L}_0^2([0,M)^n)$.    Moreover, under the representation
\eqref{eq:3.B^M} of $B^M$, \eqref{eq:6.1} is modified as
\begin{align*}
(\partial_x h)^2 &= \left\{\int_{[0,M)} \eta^\e(x-y) dB^M(y)\right\}^2 \\
&= \left\{\int_{[0,M)} \left(\eta^\e(x-y) -\tfrac1M\right) dB(y)\right\}^2
= \Psi^{\e,M}(x)+\xi^{\e,M},
\end{align*}
where $\xi^{\e,M}= \xi^\e- \frac1{M}$ and
\begin{align}  \label{eq:Psi^M}
\Psi^{\e,M}(x)
& = \int_{[0,M)^2} \left(\eta^\e(x-x_1)-\tfrac1M\right)
\left(\eta^\e(x-x_2)-\tfrac1M\right) dB(x_1)dB(x_2)  \\
& = \int_{[0,M)^2} \Big(\eta^\e(x-x_1)\eta^\e(x-x_2)\Big)^M dB(x_1)dB(x_2).  \notag
\end{align}

Theorem \ref{thm:3.1} can be reformulated as follows in the present setting:

\begin{thm}  \label{thm:3.1-M}
For every $\fa\in C(\SS_M)$ satisfying $\rm{supp}\,\fa \cap  
\rm{supp}\,\rho = \emptyset$ (so that $\dist(\rm{supp}\,\fa,
\rm{supp}\,\rho)>0$ in a periodic sense), we have that
$$
\lim_{\e\downarrow 0} E^{\pi\otimes\nu^{\e,M}}
\left[ \sup_{0\le t \le T}\left\{ \int_0^t \hat A^\e(\fa,Y^{\e,M}(s))ds \right\}^2 
\right] =0,
$$
where $\hat A^\e(\fa,Y)$ is defined similarly to that in Theorem \ref{thm:3.1} 
by taking the integral over $\SS_M$.
\end{thm}

\begin{proof}
We modify the proofs of Proposition \ref{prop:Phi} and Theorem \ref{thm:3.1}.
By \eqref{eq:Psi^M}, the expectation \eqref{eq:Exp} under $\nu^{\e,M}$
in the present setting is equal to
\begin{align}  \label{eq:3.XYZ}
\int_{[0,M)^2} \bar\fa_2^\e(x_1,x_2) &  \left(
\int_{\SS_M} \Big(\eta^\e(y-x_1)\eta^\e(y-x_2)\Big)^M\eta_2^\e(x-y)dy 
\right.\\
& \qquad\qquad\qquad \left. - \Big(\eta^\e(x-x_1)\eta^\e(x-x_2)\Big)^M\right)dx_1dx_2,  
\notag
\end{align}
where $\bar\fa_2^\e$ is defined from $\bar\fa_2$ given below \eqref{eq:3.product}
with $\R^n$ replaced by $\SS_M^n$ and $\fa_n$, $\psi_n$ replaced by
$P^M\fa_n*(\eta^\e)^{\otimes n}$, $P^M\psi_n*(\eta^\e)^{\otimes n}$,
respectively.  The operator $P^M$ is defined by $P^M\fa_n=\fa_n^M$ and note 
that the constant $\xi^{\e,M}$ plays no role in \eqref{eq:3.XYZ}.

By \eqref{eq:P^M-dual}, the operator 
$P^M$ acting on $\psi_n, \psi_{n+1}, \psi_{n+2}$ can be moved to 
$P^M\fa_{n+2}$, $P^M\fa_{n+1}$, $P^M\fa_{n}$ under the integrals in the
variable $\uu$ and to $(\eta^\e(y-x_1)\eta^\e(y-x_2))^M$ under the 
integrals in the variable $(x_1,x_2)$.  Note that we can separate $P^M$
in these two variables.  Thus, we can remove $P^M$ from
$\psi_n, \psi_{n+1}, \psi_{n+2}$.  Writing $P^M\fa_{n}, P^M\fa_{n+1}, P^M\fa_{n+2}$
simply by $\fa_{n}, \fa_{n+1}, \fa_{n+2}$ again and noting $(P^M)^2 = P^M$,
we can also drop $P^M$ from $\fa_{n}, \fa_{n+1}, \fa_{n+2}$.

Now, we expand
\begin{equation}  \label{eq:eta-eta}
\Big(\eta^\e(x-x_1)\eta^\e(x-x_2)\Big)^M
= \eta^\e(x-x_1)\eta^\e(x-x_2) - \frac1M\big(\eta^\e(x-x_1)+\eta^\e(x-x_2) \big)
+ \frac1{M^2}.
\end{equation}

The contribution to \eqref{eq:3.XYZ} from the first term of \eqref{eq:eta-eta}
is exactly the same as in the proof of Proposition \ref{prop:Phi}.
Note that the expectation $\frac1{12} E^{\nu^{\e,M}}[\frac{Y(x)}{Y_\rho}\Phi]$
can be computed similarly in the present setting.  The contribution of the last constant 
$\frac1{M^2}$ cancels when we take the difference in \eqref{eq:3.XYZ}.

The contribution to \eqref{eq:3.XYZ} from the second term of \eqref{eq:eta-eta}
becomes
\begin{align*}
-\frac2M & \int_{[0,M)^2} \bar\fa_2^\e(x_1,x_2)
\left( \int_{\SS_M} \eta^\e(y-x_1)\eta_2^\e(x-y)dy- \eta^\e(x-x_1)\right) dx_1dx_2\\
& = -\frac2M \int_{[0,M)^2} \bar\fa_2^\e(x_1,x_2)
\left( \eta_3^\e(x-x_1)- \eta^\e(x-x_1)\right) dx_1dx_2
\end{align*}
by the symmetry of $\bar\fa_2^\e(x_1,x_2)$.  However, since
$\int_{[0,M)}\fa_{n+2}(\uu,x_1,x_2)dx_2= \int_{[0,M)}\fa_{n+1}(\uu,x_2)dx_2=0$
by \eqref{eq:fa_n=0}, the contributions of $\bar\fa_2^{(1)}$ and $\bar\fa_2^{(4)}$
vanish.  It is therefore enough to compute the contributions of $\bar\fa_2^{(2)}$ 
and $\bar\fa_2^{(3)}$ only.

The computation of the contribution of 
$$
\int_{[0,M)^2} \bar\fa_2^{(2),\e}(x_1,x_2)
\left( \eta_3^\e(x-x_1)- \eta^\e(x-x_1)\right) dx_1dx_2,
$$
with $\bar\fa_2^{(2),\e}$ defined from $\bar\fa_2^{(2)}$ by replacing $\fa_n$, $\psi_{n+2}$ by
$\fa_n*(\eta^\e)^{\otimes n}$, $\psi_{n+2}*(\eta^\e)^{\otimes (n+2)}$, respectively,
can be essentially reduced to the computation of
\begin{align*}
& \int_{[0,M)^2} \phi_x^{\otimes 2}*(\eta^\e)^{\otimes 2}(x_1,x_2)
\left( \eta_3^\e(x-x_1)- \eta^\e(x-x_1)\right) dx_1dx_2\\
& = \{\phi_x*\eta_4^\e(x)-\phi_x*\eta_2^\e(x)\} \int_{[0,M)} (\phi_x*\eta^\e)(x_2)dx_2\\
& = \{\th*\eta_4^\e(x)-\th*\eta_2^\e(x)\} \int_{[0,M)} (\phi_x*\eta^\e)(x_2)dx_2.
\end{align*}
As we saw in the proof of Proposition \ref{prop:Phi}, the support of this function 
is contained in $(\text{supp}\,\rho)^{4\e}$ so that it vanishes when we multiply
the test function $\fa=\fa(x)$.

The computation of the contribution of 
$$
\int_{[0,M)^2} \bar\fa_2^{(3),\e}(x_1,x_2)
\left( \eta_3^\e(x-x_1)- \eta^\e(x-x_1)\right) dx_1dx_2,
$$
with $ \bar\fa_2^{(3),\e}$ defined from $ \bar\fa_2^{(3)}$ similarly as above
can be essentially reduced to the computation of
\begin{align*}
\int_{\SS_M^{n+1}} & \phi_x^{\otimes n}(\uu)(\phi_x*\eta^\e)(x_2)d\uu dx_2\\
& \times \big\{f_{n+1}*((\eta^\e)^{\otimes n}\otimes \eta_3^\e)(\uu,x)
- f_{n+1}*((\eta^\e)^{\otimes n}\otimes \eta^\e)(\uu,x)\big\}.
\end{align*}
However, the difference in the braces can be rewritten as
$$
E[f_{n+1}(\uu+U^\e,x+Y^\e+X_1^\e) - f_{n+1}(\uu+U^\e,x+X_1^\e)],
$$
where $U^\e=\{U_i^\e\}_{i=1}^n$, $U_i^\e\stackrel{\rm{law}}{=}\eta^\e$,
$X_1^\e\stackrel{\rm{law}}{=}\eta^\e$, $Y^\e\stackrel{\rm{law}}{=}\eta_2^\e$
and $\{U_i^\e, X_1^\e, Y^\e\}$ are independent.  This can be estimated exactly
in the same way as $I_n^{(3)}$ in the proof of Proposition \ref{prop:Phi}.
Therefore, we have a parallel assertion to Proposition \ref{prop:Phi}
in the present setting and this completes the proof of the theorem.
\end{proof}

The proof of Theorems \ref{thm:3.1} or  \ref{thm:3.1-M} can be extended 
to obtain the following lemma for $h^{\e,M}(t,x)$.  To prove this lemma,
we need Poincar\'e inequality so that this holds only on $\SS_M$ and
not on $\R$.

\begin{lem} \label{cor:3.10}
Assume that a measurable and bounded function $\fa=\fa(s,x)$ on
$[0,T]\times \SS_M$ is given.  Then, we have
\begin{align} \label{eq:cor3.10}
\sup_{0<\e<1} & E^{\pi\otimes\nu^{\e,M}}\left[ \sup_{0\le t \le T} \left( 
\int_0^t H^{\e,M}(\fa(s,\cdot),\partial h^{\e,M}(s))ds\right)^2 \right] \\
& \qquad  \le 12M^2 \int_0^T \|\fa(s,\cdot)\|_{L^\infty(\SS_M)}^2 ds,   \notag
\end{align}
where
\begin{equation}  \label{eq:H-e}
H^{\e,M}(x,\partial h) = \frac12 \big((\partial_x h)^2 -\xi^{\e,M} \big)
*\eta_2^\e(x)
\end{equation}
and $H^{\e,M}(\fa,\partial h) = \int_{\SS_M} H^{\e,M}(x,\partial h) \fa(x)dx$.
\end{lem}

\begin{proof}
By Lemma \ref{lem:3.6-a}, noting that $H^{\e,M}(\fa,\partial h)$ is a
tilt variable, the expectation in \eqref{eq:cor3.10} is bounded by
\begin{align}  \label{eq:24sup}
24\int_0^T \sup_{\Phi\in L^2(\nu^{\e,M})}
\left\{ 2E^{\nu^{\e,M}}
\left[ \tilde\Psi^{\e,M}(\fa(s,\cdot)) \Phi \right] - \|\Phi\|_{1,\e}^2\right\} ds,
\end{align}
where  $\tilde\Psi^{\e,M}(\fa) = \frac12\int_{\SS_M} \Psi^{\e,M}*\eta_2^\e(x)
\fa(x)dx$, $\Psi^{\e,M}(x)$ is given by \eqref{eq:Psi^M} and 
$\|\Phi\|_{1,\e}^2$ is defined on $\SS_M$.  Since $\Psi^{\e,M}$
is a second order Wiener chaos, we have that
\begin{align*}
2 E^{\nu^{\e,M}}
\left[ \tilde\Psi^{\e,M}(\fa) \Phi \right] 
& = \int_{\SS_M}  \fa*\eta_2^\e(x)dx
\int_{\SS_M^2} f_2(x_1,x_2) \Big(\eta^\e(x-x_1)\eta^\e(x-x_2)\Big)^Mdx_1dx_2  \\
& = \int_{\SS_M^2} f_2^M(x_1,x_2)\psi(x_1,x_2)dx_1dx_2,
\end{align*}
where 
$$
\psi(x_1,x_2) = \int_{\SS_M}\fa*\eta_2^\e(x)\eta^\e(x-x_1)\eta^\e(x-x_2)dx,
$$
and $f_2=\fa_2\otimes (\eta^\e)^{\otimes 2}$ with 
the kernel $\fa_2\in \hat L_0^2([0,M)^2)$ of the second order Wiener chaos of $\Phi$.
Note that  \eqref{eq:fa_n=0} implies $\int_{\SS_M} \fa_2(x_1,x_2)dx_i=0$,
$i=1,2$,  so that
\begin{equation}  \label{eq:vanish}
\int_{\SS_M} f_2(x_1,x_2)dx_i=0.
\end{equation}
This shows $f_2^M=f_2$.  Moreover, Lemma \ref{lem:3.5} (similar on $\SS_M$) implies
\begin{equation}  \label{eq:Sob-Poin}
\frac12 \int_{\SS_M^2} \left( \frac{\partial f_2}{\partial x_1}(x_1,x_2)\right)^2 dx_1dx_2 
\le \| \Phi\|_{1,\e}^2.
\end{equation}
We now estimate
\begin{align}  \label{eq:3.33}
\left| 2E^{\nu^{\e,M}}
\left[ \tilde\Psi^{\e,M}(\fa) \Phi \right]  \right|
\le \int_{\SS_M} \|f_2(\cdot,x_2)\|_{L^\infty(\SS_M)} \|\psi(\cdot,x_2)\|_{L^1(\SS_M)} dx_2.
\end{align}
However, we easily see that
$$
\|\psi(\cdot,x_2)\|_{L^1(\SS_M)} \le \|\fa\|_{L^\infty(\SS_M)},
$$
for every $x_2\in \SS_M$ and, by \eqref{eq:vanish} with $i=1$,
\begin{align*}
|f_2(x_1, x_2)| & = \left| f_2(x_1, x_2) - \frac1M \int_{\SS_M} f_2(y, x_2)dy \right| \\
& = \frac1M \left| \int_{\SS_M} dy \int_y^{x_1} \frac{\partial f_2}{\partial z}(z,x_2)dz\right|
\le \int_{\SS_M} \left| \frac{\partial f_2}{\partial z}(z,x_2) \right| dz.
\end{align*}
Thus, by \eqref{eq:3.33}, Schwarz's inequality and \eqref{eq:Sob-Poin}, we obtain
\begin{align*}
\left| 2E^{\nu^{\e,M}}
\left[ \tilde\Psi^{\e,M}(\fa) \Phi \right]  \right|
& \le \|\fa\|_{L^\infty(\SS_M)} 
  \int_{\SS_M^2} \left| \frac{\partial f_2}{\partial x_1}(x_1,x_2) \right| dx_1dx_2 \\
& \le M \|\fa\|_{L^\infty(\SS_M)} 
  \left( \int_{\SS_M^2} \left| \frac{\partial f_2}{\partial x_1}(x_1,x_2) \right|^2 dx_1dx_2
\right)^{1/2} \\
& \le \sqrt{2} M \|\fa\|_{L^\infty(\SS_M)} \| \Phi\|_{1,\e}.
\end{align*}
Combining this with \eqref{eq:24sup}, the conclusion is shown
similarly to the last part of the proof of Theorem \ref{thm:3.1}.
\end{proof}

\begin{rem}  \label{rem:3.4-c}
{\rm (1)} For $p>1$, $L^p$-norms of $\psi$ diverge as $\e\downarrow 0$ in general.\\
{\rm (2)} On $\R$, we have the same estimate as \eqref{eq:3.33} and, if 
$\text{supp}\, \fa \subset [-K,K]$, then $\|\psi(\cdot,x_2)\|_{L^1(\R)} =0$
if $|x_2|\ge K+3\e$.  Therefore, Morrey's theorem (\cite{Adams}, p.97 (8)),
which implies $\|f_2(\cdot,x_2)\|_{L^\infty(\R)} \le C \|f_2(\cdot,x_2)\|_{H^1(\R)}$,
shows that
$$
\left| 2E^{\nu^{\e}}
\left[ \tilde\Psi^\e(\fa) \Phi \right]  \right|
\le C \sqrt{K} \|\fa\|_{L^\infty(\R)} \{\|\partial f_2/\partial x_1\|_{L^2(\R^2)} + 
\|f\|_{L^2(\R^2)}\}.
$$
However, since Poincar\'e inequality is missing, this cannot be bounded by
$\|\Phi\|_{1,\e}$, because of the last term $\|f\|_{L^2(\R^2)}$.  This estimate
can be easily extended to $\fa$ such that $\sup_{x\in\R}|\fa(x)|\sqrt{1+|x|}
< \infty$.
\end{rem}

This lemma is applied to prove the following proposition, which shows that 
the height average cannot move very quickly.   Recall that 
the initial distribution of $h^{\e,M}(t,x)$ is given by $h^{\e,M}(0,\cdot)
\stackrel{\rm{law}}{=} \pi\otimes\nu^{\e,M}$ under the map \eqref{eq:3.1-map}
defined on $\SS_M$.
Note that the wrapping is not introduced for $h^{\e,M}(t,x)$.
This result will be used to remove the wrapping from
$Y^{\e,M}(t,x)$ in the next section.

\begin{prop} \label{prop:3.12}
For every $T>0$ and $\fa\in C^2(\SS_M)$,
\begin{equation}  \label{eq:H-3}
\sup_{0<\e<1} E\left[\sup_{0\le t \le T}h^{\e,M}(t,\fa)^2\right] <\infty.
\end{equation}
\end{prop}

\begin{proof}
For every $\fa\in C^2(\SS_M)$, from the SPDE \eqref{eq:1-M}, we have
\begin{align}  \label{eq:H-1}
h^{\e,M}(t,\fa) = & h^{\e,M}(0,\fa)  +  \frac12 \int_0^t h^{\e,M}(s,\fa'')ds
+ \int_0^t H^{\e,M}(\fa,\partial h^{\e,M}(s))ds \\
&+ \frac12(\xi^{\e,M}-\xi^\e)t \lan\fa,1\ran_{\SS_M}
+ \lan W(t),\fa*\eta^\e\ran_{\SS_M}.   \notag
\end{align}
By Lemma \ref{cor:3.10}, it is shown that
$$
\sup_{0<\e<1} E\left[ \sup_{0\le t \le T} \left( 
\int_0^t H^{\e,M}(\fa,\partial h^{\e,M}(s))ds\right)^2 \right] <\infty.
$$
It is easy to see that $\xi^{\e,M}-\xi^\e= -\frac1M$ is finite,
$$
E\left[ \sup_{0\le t \le T}  \lan W(t),\fa*\eta^\e\ran_{\SS_M}^2\right]
\le 4 E[\lan W(T),\fa*\eta^\e\ran_{\SS_M}^2]
= 4T \|\fa*\eta^\e\|_{L^2(\SS_M)}^2  
\le 4T \|\fa\|_{L^2(\SS_M)}^2 <\infty,
$$
and
$$
E\left[ \sup_{0\le t \le T} \left( 
\int_0^t h^{\e,M}(s,\fa'')ds\right)^2 \right] \le T \int_0^T E[h^{\e,M}(s,\fa'')^2]ds.
$$
Therefore, once we can prove
\begin{equation}  \label{eq:H-2}
\sup_{0<\e<1} \sup_{0\le t \le T}E[h^{\e,M}(s,\fa)^2] <\infty,
\end{equation}
for every $\fa\in C(\SS_M)$, \eqref{eq:H-1} shows \eqref{eq:H-3}.

The next task is to give the proof of \eqref{eq:H-2}.  To this end,
we rewrite the equation for $h^{\e,M}(t,x)$ into a mild form:
\begin{align*}
h^{\e,M}(t,\fa) = & \int_{\SS_M} h^{\e,M}(0,y)\psi(t,y)dy + \int_0^t
\int_{\SS_M} \psi(t-s,y) H^{\e,M}(y,\partial h^{\e,M}(s))dsdy \\
& + \frac12 (\xi^{\e,M}-\xi^\e) \int_0^t\int_{\SS_M} \psi(t-s,y) dsdy
+ \int_0^t \int_{\SS_M} \psi(t-s,y) W^\e(dsdy),
\end{align*}
where $\psi(t,y) = \int_{\SS_M} p^M(t,x,y)\fa(x)dx$ and
$p^M$ is the heat kernel on $\SS_M$.  We apply Lemma \ref{cor:3.10}
by dropping $\sup_{0\le t\le T}$ and then regarding $\fa(s,y) =\psi(t-s,y)$
for fixed $t$, we have
\begin{align*}
& E^{\pi\otimes \nu^{\e,M}}\left[\left( \int_0^t 
\int_{\SS_M} \psi(t-s,y) H^{\e,M}(y,\partial h^{\e,M}(s))dsdy \right)^2\right] \\
& \le 12M^2 \int_0^t \|\psi(t-s,\cdot)\|_{L^\infty(\SS_M)}^2ds 
\le 12M^2 t \|\fa\|_{L^\infty(\SS_M)}^2<\infty.
\end{align*}
We also see
$$
E^{\pi\otimes \nu^{\e,M}}\left[\left( \int_0^t \int_{\SS_M}
 \psi(t-s,y) W^\e(dsdy) \right)^2 \right]
= \int_0^t \int_{\SS_M} (\psi(t-s,\cdot)*\eta^\e(y))^2dsdy
\le C <\infty.
$$
Since the first term has a uniform bound recalling
 $h^{\e,M}(0,\cdot) \stackrel{\rm{law}}{=} \pi\otimes\nu^{\e,M}$
and the integral in the third term is simply $t\lan\fa,1\ran_{\SS_M}$, 
this completes the proof of  \eqref{eq:H-2}. 
\end{proof}

\begin{rem}  \label{rem:3.5}
As we will see in the next subsection, $N^{\e,M}(t,\fa) = \int_{\SS_M}
N^{\e,M}(t,x)\fa(x)dx$ 
in \eqref{eq:3-Y^e-2-M} converges weakly in $L^2(\Om)$ as $\e\downarrow 0$,
since all other terms do.  Therefore, its $L^2$-norm is automatically bounded in $\e$:
\begin{equation} \label{eq:3.4-a}
\sup_{0<\e<1} E\big[\big(N^{\e,M}(t,\fa)\big)^2\big] < \infty.
\end{equation}
Proposition \ref{prop:3.12} gives a stronger estimate with supremum in $t$
inside the expectation. 
\end{rem}

\subsection{Proof of Theorem \ref{thm:0}} \label{subsection:3.4.3}

We are now at the position to complete the proof of Theorem \ref{thm:0}.
We fix $M\ge 2$ (or $M\ge 1$ by making the support of $\rho$ smaller)
and denote $h^{\e,M}(t,x)$, $Z^{\e,M}(t,x)$, $Y^{\e,M}(t,x)$
simply by $h^{\e}(t,x)$,  $Z^{\e}(t,x)$, $Y^{\e}(t,x)$, respectively, in
this subsection.

\vskip 2mm
\noindent
{\it Step 1.} $\,$ 
Let $h^{\e}(t,x)$ be the solution of the SPDE \eqref{eq:1-M} such that 
$h^{\e}(0,0)=h_0\in \R$ and $\nabla h^{\e}(0)\stackrel{\rm{law}}{=}\nu^{\e,M}$.
We may assume $h_0=0$ without loss of generality.
Then, $h^\e(0,\cdot) (\in \mathcal{C}_M) \stackrel{\rm{law}}{=} 
\de_0\otimes \nu^{\e,M}$, $0<\e<1$ is tight.

\begin{lem}  \label{lem:3.12-ub}
{\rm (1)} The uniform bound \eqref{eq:H-3} holds also for this $h^{\e}(t,x)$.  \newline
{\rm (2)} The tilt variable of $h^\e(t,\cdot)$ under the map \eqref{eq:3.1-map}
on $\SS_M$ is $\nu^{\e,M}$-distributed for all $t\ge 0$.
\end{lem}

\begin{proof}
The bound \eqref{eq:H-3} was shown for 
the solution $\tilde h^{\e}(t,x)$ of the SPDE \eqref{eq:1-M} such that 
$\tilde h^{\e}(0,\cdot) = \big(\tilde h^{\e}(0,\rho), \tilde h^{\e}(0,\cdot)
-\tilde h^{\e}(0,\rho)\big) \stackrel{\rm{law}}{=}\pi\otimes\nu^{\e,M}$
under the map \eqref{eq:3.1-map} on $\SS_M$.  However, two solutions have a 
simple relation: $\tilde h^{\e}(t,x) = h^{\e}(t,x) - h^{\e}(0,\rho)+m$ with 
a $\pi$-distributed random variable $m$ (which is independent of
tilt variables of $h^\e(0,\cdot)$).  Therefore, we see that
$|h^\e(t,\fa)-\tilde h^\e(t,\fa)| \le \{1+|h^\e(0,\rho)|\}\int_{\SS_M}|\fa(x)|dx$
so that \eqref{eq:H-3} holds for the solution $h^\e(t,x)$ 
we are now considering.  The second assertion (2) follows by
noting that the tilt variables of $h^\e(t,\cdot)$ and $\tilde h^\e(t,\cdot)$
coincide: $h^\e(t,x)-h^\e(t,\rho) = \tilde h^\e(t,x) - \tilde h^\e(t,\rho)$.
\end{proof}

Lemma \ref{lem:3.12-ub} taking $\fa=\rho$ in \eqref{eq:H-3}
shows that, for every fixed $t>0$, $h^\e(t,\cdot)$, $0<\e<1$, is tight on 
$\mathcal{C}_M$.
We fix $T>0$ arbitrarily.  Then, from Lemma \ref{lem:3.12-ub} again,  
$(h^\e(0,\cdot), h^\e(T,\cdot), X^\e:=\sup_{0\le t \le T}|h^\e(t,\rho)|, W(t))$, 
$0<\e<1$, is jointly tight on
$\mathcal{C}_M\times \mathcal{C}_M\times \R\times C([0,T], H^{-\a}(\SS_M))$
with suitably chosen $\a>0$.  In particular, every subsequence $\{\e'\downarrow 0\}$ of 
$\{\e\in (0,1)\}$ contains a weakly convergent subsequence (in law sense)
$(h^{\e''}(0,\cdot), h^{\e''}(T,\cdot), X^{\e''}, W(t))$.   Thus, by Skorohod's
theorem, we can find a probability space $(\tilde\Om,\tilde P)$ and
$(\tilde h^{\e''}(0,\cdot), \tilde h^{\e''}(T,\cdot), \tilde X^{\e''}, \tilde W_{\e''}(t))$
defined on this space having the same law as 
$(h^{\e''}(0,\cdot), h^{\e''}(T,\cdot), X^{\e''}, W(t))$,
and $(\tilde h(0,\cdot), \tilde h(T,\cdot), \tilde X, \tilde W(t))$ also
defined on this space such that
$(\tilde h^{\e''}(0,\cdot), \tilde h^{\e''}(T,\cdot), \tilde X^{\e''}, \tilde W_{\e''}(t))$
converges to 
$(\tilde h(0,\cdot), \tilde h(T,\cdot), \tilde X, \tilde W(t))$
$\tilde P$-a.s.\ $\tilde \om$ as $\e'' \downarrow 0$.

From the space-time Gaussian white noise $\tilde W_{\e''}(t)$,
one can construct a smeared noise $\tilde W_{\e''}^{\e''}(t) := \tilde W_{\e''}(t)*
\eta^{\e''}$ and consider the SPDE \eqref{eq:8-M} with initial value 
$\tilde Z^{\e''}(0,\cdot) = e^{\tilde h^{\e''}(0,\cdot)}$.  Then, we have its solution
$\tilde Z^{\e''}(t,\cdot)$.  $\tilde Z^{\e''}(T,\cdot)$ is consistent with
$\tilde h^{\e''}(T,\cdot)$ and $\tilde Z^{\e''}(t,0)$ is consistent with
$\tilde h^{\e''}(t,0)$, respectively, $\tilde P$-a.s.

For every $L\ge 1$, we define the wrapped process  $\tilde Y_L^{\e''}(t,x) 
\equiv \tilde Y_L^{\e''}(t,x;\tilde \om,m)$
on the extended probability space $(\tilde\Om\times [-L,L],\tilde P\otimes\pi_L)$
in such a manner that $\tilde Y_L^{\e''}(t,x) = e^m \tilde Z^{\e''}(t,x)$ modulo
$2L$ averaged by $\rho$ in an exponential sense such as \eqref{eq:3.5-c}, 
i.e., $\log (\tilde Y_L^{\e''}(t))_\rho \in [e^{-L},e^L]$ (instead of $[1,e]$)
and $\log (\tilde Y_L^{\e''}(t))_\rho = m+ \log (\tilde Z^{\e''}(t))_\rho$
modulo $2L$ with $\pi_L$-valued random variable $m$, where $\pi_L$
is a uniform probability measure on $[-L,L]$.
In other words, the initial value
$\log (\tilde Y_L^{\e''})_\rho(0)$ is distributed under $\pi_L$.

\noindent
{\it Step 2}. $\,$ Since $\tilde X^{\e''} = \sup_{0\le t \le T}|\tilde h^{\e''}(t,\rho)|$
 converges to $\tilde X$ as $\e''\downarrow 0$
in $\R$, $\tilde P$-a.s., and $0\le \tilde X<\infty$, $\tilde P$-a.s.,
we see that $\lim_{L\to\infty} \tilde P(\tilde \Om_L)=1$ holds, where
$$
\tilde \Om_L := \left\{ \sup_{0<\e''<1} \sup_{0\le t \le T}
|\tilde h^{\e''}(t,\rho)| \le \tfrac{L}2 \right\} \; (\subset \tilde \Om).
$$
In particular, we have that
\begin{equation} \label{eq:A-1}
\tilde Y_L^{\e''}(t,x;\tilde \om,m) = e^m \tilde Z^{\e''}(t,x;\tilde \om)
\end{equation}
for all $t\in [0,T], m: |m|\le \tfrac{L}2, 0<\e''<1$ and $\tilde \om \in
\tilde \Om_L$.

We prepare the following uniform bound on $\tilde Y_L^{\e''}(t,x)$ defined on $\tilde\Om$.

\begin{lem}
We have that
\begin{equation} \label{eq:3.Y}
\tilde E[\tilde Y_L^{\e''}(t,x;\cdot,m)^{2p}] \le e^{4p^2(|x|+c)+2Lp},
\end{equation}
for every $0<\e''<1$, $t\ge 0$, $p\ge 1$, $x\in [-\tfrac{M}2,\tfrac{M}2]$ and $m\in [-L,L]$ with some
$c>0$. 
\end{lem}

\begin{proof}
Since the tilt variable for $\tilde Y^{\e''}(t)$ is distributed under $\nu^{\e'',M}$, 
we have that
\begin{align*}
\tilde E[\tilde Y_L^{\e''}(t,x)^{2p}] 
& \le \tilde E[(\tilde Y_L^{\e''}(t))_\rho^{4p}]^{1/2} 
\tilde E \left[ \frac{\tilde Y_L^{\e''}(t,x)^{4p}}{(\tilde Y_L^{\e''}(t))_\rho^{4p}}\right]^{1/2}\\
& \le e^{2Lp} E[e^{4p\big(B^M(\cdot)-\int_{\SS_M} B^M(y)\rho(y)dy\big)*\eta^{\e''}(x)}]^{1/2} \\
& = e^{2Lp} E[e^{4p\int_{\SS_M} \phi_\cdot(u)*\eta^{\e''}(x)dB^M(u)}]^{1/2} \\
& = e^{2Lp} e^{\tfrac14 (4p)^2\int_{\SS_M} \big(\phi_\cdot(u)*\eta^{\e''}(x)
- \frac1M \int_{\SS_M} \phi_\cdot(v)*\eta^{\e''}(x)dv\big)^2 du} \\
& \le e^{2Lp} e^{4p^2\int_{\SS_M} \big(\phi_\cdot(u)*\eta^{\e''}(x)\big)^2 du} \\
& \le e^{4p^2(|x|+ c)+2Lp},
\end{align*}
where $\phi_x(u)$ is defined in \eqref{eq:3.17-x}.  Note that 
$\phi_\cdot(u)*\eta^{\e''}(x) = 1_{[u,\infty)}*\eta^\e(x)+\th(u)$,
$1_{[u,\infty)}*\eta^\e(x) =0$ $(u\ge x+\e$), $1 (u\le x-\e$),
$\in [0,1]$ (otherwise), and $\th(u)=0 (u\ge K)$, $\th(u)=-1
(u\le -K)$, where $K>0$ is taken such that $\text{supp}\, \rho 
\subset [-K,K]\subset [-1,1]$.  Recall that supp$\,\eta^\e \subset [-\e,\e]$.
This shows \eqref{eq:3.Y}.  
\end{proof}

Since \eqref{eq:3.Y} with $p=1$ implies the weak compactness of 
$\tilde Y_L^{\e''}(\cdot,\cdot)$, by the similar argument to Krylov and 
Rozovskii \cite{KR}, p.1264 and by the diagonal argument
in $L\in \N$, one can find a subsequence $\e'''\downarrow 0$ of $\e''$
such that
\begin{align}  \label{eq:3.25}
\tilde Y_L^{\e'''}(t,x) \to \tilde Y_L(t,x;\tilde\om,m),
\end{align}
weakly in $L^2([0,T]\times\tilde\Om\times [-L,L] ,\bar{\mathcal{T}}_L,
dtd\tilde Pd\pi_L; L^2(\SS_M))$ with some $\tilde Y_L(t,x;\tilde\om,m)$,
where $\bar{\mathcal{T}}_L$ is the completion of the $\si$-field 
of progressively measurable sets on $[0,T]\times\tilde\Om\times [-L,L]$ 
with respect to $dtd\tilde Pd\pi_L$ for every $L\in \N$.
 This combined with \eqref{eq:A-1} shows
\begin{align}  \label{eq:3.26}
\tilde Z^{\e'''}(t,x) \to \tilde Z(t,x),  
\end{align}
weakly in $L^2([0,T]\times\tilde\Om_L, \bar{\mathcal{T}}, 
dtd\tilde P; L^2(\SS_M))$,
for all $L\in \N$ with some $\tilde Z(t,x)$, where $\bar{\mathcal{T}}$ is 
defined similarly to $\bar{\mathcal{T}}_L$.  Furthermore, by definition,
we see that
\begin{equation} \label{eq:A-2}
\tilde Y_L(t,x;\tilde \om,m) = e^m \tilde Z(t,x;\tilde \om)
\end{equation}
in $L^2([0,T],L^2(\SS_M))$ for a.e.\ $m: |m|\le \tfrac{L}2, L\in \N$ 
and $\tilde P$-a.s.\ $\tilde \om \in\tilde 
\Om_L$. 

The identity \eqref{eq:A-2} can be rewritten as
\begin{equation*}
 \tilde Z(t,x;\tilde \om) = e^{-m}\tilde Y_L(t,x;\tilde \om,m),
\end{equation*}
for a.e.\ $m: |m|\le \tfrac{L}2, L\in \N$ and $\tilde P$-a.s.\
$\tilde \om \in\tilde \Om_L$. 

\noindent
{\it Step 3}.$\,$ 
From \eqref{eq:3-Y^e-2-M} considered on $\tilde\Om$ and modulo $2L$
rather than modulo $1$, under a multiplication of $\fa = \fa(x)
\in C^\infty(\SS_M)$ and dropping superscripts $M$ as above, we have that
\begin{align}  \label{eq:3-Y^e-3}
& \lan \tilde Y_L^{\e'''}(t),\fa\ran_{\SS_M}
= \lan \tilde Y_L^{\e'''}(0),\fa\ran_{\SS_M} \\
& \quad + \int_0^t 
   \lan \tilde Y_L^{\e'''}(s), \tfrac12 \partial_x^2\fa 
+ \tfrac1{24}\fa\ran_{\SS_M} ds
 + \int_0^t \int_{\SS_M} \hat A^{\e'''}(x,\tilde Y_L^{\e'''}(s)) \fa(x) dsdx
 \notag \\
& \quad 
+ \int_0^t \int_{\SS_M} \tilde Y_L^{\e'''}(s,x) \fa(x) \tilde W^{\e'''}(dsdx)
+ \int_{\SS_M} \tilde N_L^{\e'''}(t,x) \fa(x) dx, \notag
\end{align}
where $\tilde W^{\e'''} := \tilde W_{\e'''}*\eta^{\e'''}$ and 
$\tilde N_L^{\e'''}$ is defined correspondingly.

By Theorem \ref{thm:3.1-M}, for each fixed $L\in \N$,
the third term in the right hand side 
of \eqref{eq:3-Y^e-3} converges to $0$ strongly in $L^2(\tilde\Om\times [-L,L],
d\tilde Pd\pi_L)$ as $\e (=\e''') \downarrow 0$, if $\text{supp}\,\fa \cap
\text{supp}\,\rho=\emptyset$.
The fourth term in the right hand side of \eqref{eq:3-Y^e-3}
involving stochastic integrals converges in the following sense:

\begin{lem}\label{lem:3.1-ab}
For every $0\le t\le T$, as $\e (=\e''') \downarrow 0$,
\begin{equation} \label{eq:3.stoch-int}
\int_0^t \int_{\SS_M} \tilde Y_L^{\e'''}(s,x)\fa(x) \tilde W^{\e'''}(dsdx)
\rightarrow \int_0^t \int_{\SS_M} \tilde Y_L(s,x)\fa(x) \tilde W(dsdx),
\end{equation}
weakly in $L^2(\tilde \Om\times [-L,L],\tilde{\mathcal F}_t,\tilde P\otimes\pi_L)$, 
where $\tilde{\mathcal F}_t$ is the $\si$-field generated by
$\{\tilde{W}(s); 0\le s \le t\}$.
\end{lem}

\begin{proof}
Rewriting as
\begin{align*}
\int_0^t \int_{\SS_M} \tilde Y_L^{\e'''}(s,x)\fa(x) \tilde W^{\e'''}(dsdx)
& =  \int_0^t \int_{\SS_M} 
\left(\tilde Y_L^{\e'''}(s,\cdot)\fa(\cdot)\right)* \eta^{\e'''}(x) \tilde W_{\e'''}(dsdx)\\
& =  \int_0^t \int_{\SS_M^2} \tilde Y_L^{\e'''}(s,y)\fa(y)
\eta^{\e'''}(x-y) \tilde W_{\e'''}(dsdx)dy,
\end{align*}
the difference of the both sides of \eqref{eq:3.stoch-int} is given by
\begin{align*}
&  \int_0^t \int_{\SS_M^2} \tilde Y_L^{\e'''}(s,y) \fa(y)\eta^{\e'''}(x-y)
\{ \tilde W_{\e'''}(dsdx) -  \tilde W(dsdx)\} dy  \\
& \quad + 
\int_0^t \int_{\SS_M^2} \{\tilde Y_L^{\e'''}(s,y)-\tilde Y_L(s,y)\}
\fa(y)\eta^{\e'''}(x-y) \tilde W(dsdx)dy  \\
& \quad + \int_0^t \int_{\SS_M^2} \{\tilde Y_L(s,y)-\tilde Y_L(s,x)\}
\fa(y)\eta^{\e'''}(x-y) \tilde W(dsdx)dy \\
& \quad + \int_0^t \int_{\SS_M} \tilde Y_L(s,x) 
\{ \fa*\eta^{\e'''}(x)-\fa(x)\} \tilde W(dsdx) \\
& \quad =: I^{(1,\e''')} + I^{(2,\e''')} +  I^{(3,\e''')} + I^{(4,\e''')}.
\end{align*}

For the first term $I^{(1,\e''')}$, since both $\tilde W_{\e'''}$ and $\tilde W$ are 
the space-time Gaussian white noises and $\tilde W_{\e'''}$ converges to $\tilde W$
a.s., we see that $\tilde W_{\e'''} - \tilde W$ is also a space-time Gaussian white noise
multiplied by a constant $c_{\e'''}$ converging to $0$ as $\e'''\downarrow 0$.  Thus,
\begin{align*}
\tilde E[\{I^{(1,\e''')}\}^2] 
& = c_{\e'''}^2 \int_0^t \int_{\SS_M} \tilde E \left[ \left(
\int_{\SS_M} \tilde Y_L^{\e'''}(s,y) \fa(y)\eta^{\e'''}(x-y) dy 
\right)^2 \right]ds dx\\
& \le c_{\e'''}^2 \int_0^t \int_{\SS_M} \fa(y)^2 \tilde E [ 
\tilde Y_L^{\e'''}(s,y)^2 ]ds dy,
\end{align*}
which converges to $0$ as $\e'''\downarrow 0$ by
noting \eqref{eq:3.Y}.  For the last term,
since $\|\fa*\eta^{\e'''}-\fa\|_{L^\infty(\SS_M)} \to 0$, we easily see that
$$
\tilde E[\{I^{(4,\e''')}\}^2] \to 0 \quad \text{as } \e'''\downarrow 0.
$$

For other terms, note that functions $\Phi$ of the forms
$$
\Phi = \int_0^t \int_{\SS_M} \fa(s,x,\tilde\om) \tilde W(dsdx)
$$
with bounded and continuous kernels $\fa(s,x,\tilde\om)$ form a dense
family in  $L^2(\tilde\Om,\tilde{\mathcal F}_t,\tilde P)\ominus \{\text{const}\}$. 
Then, we have that
$$
\tilde E[I^{(2,\e''')}\Phi] = \int_0^t ds \int_{\SS_M^2} 
 \tilde E[\{\tilde Y_L^{\e'''}(s,y)-\tilde Y_L(s,y)\} \fa(s,x,\tilde\om)] 
\fa(y)\eta^{\e'''}(x-y) dxdy.
$$
If we can replace $\fa(s,x,\tilde\om)$ with $\fa(s,y,\tilde\om)$, then
$\eta^{\e'''}(x-y)$ disappears under the integration in $x$ and
this expectation converges to $0$ by \eqref{eq:3.25}
and applying Lebesgue's convergence theorem noting 
\eqref{eq:3.Y}.  The replacement of $\fa(s,x,\om)$ with
$\fa(s,y,\om)$ with a small error negligible as $\e'''\downarrow 0$
can be justified by the continuity of $\fa(\cdot,\cdot,\om)$ 
noting \eqref{eq:3.Y}.  

Similarly for $I^{(3,\e''')}$, we have that
$$
\tilde E[I^{(3,\e''')}\Phi] = \int_0^t ds \int_{\SS_M^2} 
 \tilde E[\{\tilde Y_L(s,y)-\tilde Y_L(s,x)\} \fa(s,x,\tilde\om)] \fa(y)\eta^{\e'''}(x-y) dxdy.
$$
First, by the continuity of $\fa(\cdot,\cdot,\tilde\om)$ and noting
\eqref{eq:3.Y}, which implies a corresponding bound on $\tilde Y_L(s,y)$, 
we can replace
this by
\begin{equation}  \label{eq:3.YYY}
\int_0^t ds \int_{\SS_M^2}  \tilde E[\tilde Y_L(s,y)\fa(s,y,\tilde\om)
-\tilde Y_L(s,x) \fa(s,x,\tilde\om)] \fa(y)\eta^{\e'''}(x-y) dxdy
\end{equation}
with a small error negligible as $\e'''\downarrow 0$.  Then, noting the
symmetry of $\eta^{\e'''}$, we can rewrite \eqref{eq:3.YYY} as
$$
\int_0^t ds \int_{\SS_M^2}  \tilde E[\tilde Y_L(s,y)\fa(s,y,\tilde\om)] 
(\fa(y) -\fa(x)) \eta^{\e'''}(x-y) dxdy.
$$
However, since $\tilde E[\tilde Y_L(s,y)\fa(s,y,\tilde\om)] \in L^2([0,t]\times \SS_M)$,
this tends to $0$.  This shows that $E[I^{(3,\e''')}\Phi]$ converges to $0$ as $\e'''\downarrow 0$.
\end{proof}

From Theorem \ref{thm:3.1-M} applied on $\tilde\Om$, \eqref{eq:3.25} and Lemma
\ref{lem:3.1-ab}, the last term in the right hand side of 
\eqref{eq:3-Y^e-2-M} (on $\tilde\Om$ and modulo $2L$, or 
\eqref{eq:3-Y^e-3} integrated with $\fa(x)$)
must converge weakly 
to a certain $\tilde{N}_L(t,x) = \tilde{N}_L(t,x;\tilde\om,m)$ in 
$L^2(\tilde\Om\times [-L,L], \tilde{\mathcal F}_T,\tilde P\otimes\pi_L)$;
c.f.\ Remark \ref{rem:3.5}.  Thus we obtain the equation:
\begin{align}  \label{eq:3-12-a}
\tilde Y_L(t,x) =  \tilde Y_L(0,x) + & \frac12 \int_0^t \partial_x^2 \tilde Y_L(s,x)ds 
+ \frac1{24} \int_0^t \tilde Y_L(s,x) ds \\
 & + \int_0^t \tilde Y_L(s,x) d\tilde W(s,x) + \tilde{N}_L(t,x),  \notag
\end{align}
for a.e.\ $x\in (\text{supp}\, \rho)^c$, and a.e.\ $(t,\tilde\om,m)$ in the limit.  

Recalling the definition of $\tilde\Om_L$, we see that the limit 
$\tilde{N}_L(t,x) = \tilde{N}_L(t,x;\tilde\om,m)$ of the jump part of $\tilde Y_L^{\e'''}$
vanishes on $\tilde\Om_L\times [-\frac{L}2,\frac{L}2]$, i.e., $\tilde N_L(t,x)=0$
for all $t\in [0,T]$ on $\tilde \Om_L\times [-\frac{L}2,\frac{L}2]$.  Moreover, 
since $\tilde P(\cup_{L\in\N}\tilde\Om_L)=1$ holds, from 
\eqref{eq:3-12-a} and \eqref{eq:A-2},
we see that $\tilde Z(t,x;\tilde \om)$ satisfies the SPDE
\eqref{eq:8-b-M} on $(\text{supp}\, \rho)^c$.  However, choosing another $\bar\rho$,
which satisfies the same condition as $\rho$ stated at the beginning of Section 
\ref{sec:3.2}, such that $\text{supp}\, \rho \cap \text{supp}\, \bar\rho =
\emptyset$, we see that $\tilde Z(t,x;\tilde \om)$ satisfies the SPDE
\eqref{eq:8-b-M} on the whole torus $\SS_M$.  Note that $\tilde Z$ is defined
independently of the choice of $\rho$.

Since the solution of \eqref{eq:8-b-M} is unique and continuous in $(t,x)$,
we find that  $\tilde Z(t,x;\tilde \om)$ is continuous in $(t,x)$, $\tilde P$-a.s.
In particular, $\tilde Z(T,x;\tilde \om)$ is consistent with $\tilde h(T,\cdot)$,
which was introduced in Step 1 of this subsection.  Since $\tilde Z^{\e'''}(T)$
converges to $\tilde Z(T)$ $\tilde P$-a.s.\ and the limit is uniquely
characterized by the SPDE \eqref{eq:8-b-M}, we don't need to
take the subsequences.  This concludes the proof of Theorem 
\ref{thm:0}.

\subsection{Proof of Theorem \ref{thm:1.1}}  \label{section:3.6}

We first prove that the solution $Z^M(t,x), x\in \R$ of the SPDE
\eqref{eq:8-b-M} on $\SS_M$ periodically extended to the whole
line $\R$ weakly converges to the solution $Z(t,x), x\in \R$ of the 
SPDE \eqref{eq:8-b}:

\begin{prop}  \label{prop:SPDE-conv}
Assume that $Z(0,\cdot)\in L_r^2(\R), r>0,$ is given and the initial
value of $Z^M$ is determined by $Z^M(0,x) = Z(0,x)$ for $|x|\le M/2$
and periodically extended to $\R$.  Then, we have the followings.\\
{\rm (1)} $\{Z^M(t,x)\}_{M\ge 1}$ is tight on $C([0,\infty),L_r^2(\R))$.\\
{\rm (2)} $Z^M(t,x)$ weakly converges to $Z(t,x)$ on
the space $C([0,\infty),L_r^2(\R))$ as $M\to\infty$.
\end{prop}

To prove this proposition, we prepare a lemma.  Recall that
the fundamental solution $p^M$ of the parabolic operator 
$\tfrac{\partial}{\partial t} - \tfrac12 \tfrac{\partial^2}{\partial x^2}$ 
on $\SS_M$ is given by
$$
p^M(t,x,y) = \sum_{n=-\infty}^\infty p(t,x,y+nM),
\quad x,y \in \SS_M = [0,M).
$$
We define a function $\chi_M(x), x\in \R$ by $\chi_M(x)=|x|$
for $|x|\le \frac{M}2$ and then by periodically extending it on $\R$.
The following lemma is an extension of Lemma 6.2 of \cite{Shiga}
to $\SS_M$; see also \cite{F91}.

\begin{lem}  \label{lem:tight-Z}
{\rm (1)}  For every $0<\de<1$ and $T>0$, there exists 
$C=C_{\de,T}>0$ such that
$$
\int_0^{t'}\int_{\SS_M} \left( p^M(t'-s,x',y) - 1_{\{s\le t\}}p^M(t-s,x,y)
\right)^2dy \le C(|t-t'|^{1/2}+|x-x'|^{1-\de}),
$$
holds for $0\le t < t'\le T, x,x'\in \SS_M$.  \\
{\rm (2)} For every $r\in \R$ and $T>0$,
$$
\sup_{M\ge 1}\sup_{0\le t \le T} \sup_{x\in \SS_M}
e^{-r\chi_M(x)}\int_{\SS_M} p^M(t,x,y)e^{r\chi_M(y)}dy<\infty.
$$
\end{lem}

\begin{proof}
To show (1), we first expand the square inside the integral in $y$
and apply Chapman-Kolmogorov's equality for the integral of 
products of $p^M$.  Then, (1) is shown by the following easily
shown uniform bounds on $p^M$:
$$
0<p^M(t,x,y)\le \frac{C}{\sqrt{t}}, \quad
\left|\frac{\partial p^M}{\partial x}(t,x,y)\right| \le \frac{C}t,
$$
for $0<t\le T$ with $C=C_T>0$ which is independent of $M$.
For (2), we note that
$$
\chi_M(x)-|a|\le \chi_M(x+a)\le \chi_M(x)+|a|, \quad x, a \in \R,
$$
and therefore
$$
\int_{\SS_M} p^M(t,x,y)e^{r\chi_M(y)}dy
= E\left[ e^{r\chi_M(x+B(t))}\right]
\le e^{r\chi_M(x)}E\left[ e^{|rB(t)|}\right] \le Ce^{r\chi_M(x)},
$$
where $B(t)$ is a Brownian motion starting at $0$.
\end{proof}

\begin{proof}[Proof of Proposition \ref{prop:SPDE-conv}]
For (1), we follow the proof of Theorem 2.2 in \cite{Shiga}.
The SPDE \eqref{eq:8-b-M} can be rewritten into the mild form:
\begin{align}  \label{eq:Z-mild}
Z^M(t,x) = & \int_{\SS_M} Z^M(0,y) p^M(t,x,y)dy  \\
& + \int_0^t ds  \int_{\SS_M}  e^{\frac1{24}(t-s)}p^M(t-s,x,y)
Z^M(s,y) W(dsdy).   \notag
\end{align}
We first show that
$$
\sup_{M\ge 1}\sup_{0\le t \le T} \int_{\SS_M} e^{-r\chi_M(x)}
E[|Z^M(t,x)|^{2p}] dx<\infty,
$$
for $p\ge 1$ and $T>0$ by using Lemma \ref{lem:tight-Z}-(2).
Then, by Lemma \ref{lem:tight-Z}-(1), we have the bound:
$$
E[|X^M(t,x)- X^M(t',x')|^{2p}]
\le C\{|t-t'|^p+|x-x'|^{2p(1-\de)}\},  
$$
for every $0\le t<t'\le T, \; x, x'\in \SS_M: |x-x'|\le 1$,
where $X^M(t,x)$ is the second term in the right hand side
of \eqref{eq:Z-mild}.  Since the first term is easily treated, this
proves the tightness.

To show (2), we rely on the martingale formulation as in the
proof of Theorem \ref{thm:2}.  We consider the operator
$$
\mathcal{L}^Z\Phi(Z) =\int_\R \left(  \frac12 Z^2(x)
D^2\Phi(x,x;Z) +\Big(\frac12 \partial_x^2Z(x) +\frac1{24} Z(x)\Big)
D\Phi(x;Z) \right) dx,
$$
for $\Phi\in \mathcal{D}(\mathcal{C})$ and 
$\mathcal{L}_M^Z\Phi$ for $\Phi\in \mathcal{D}(\mathcal{C}_M)$
by replacing the integral over $\R$ by $\SS_M$.  Then, the distribution
$P^M$ of  $Z^M(\cdot,\cdot)$ is a solution of 
$(\mathcal{L}_M^Z, \mathcal{D}(\mathcal{C}_M))$-martingale 
problem and $\{P^M\}$ is tight.  Then, it is easy to see that
any weak limit of $P^M$ as $M\to\infty$ is a solution of
$(\mathcal{L}^Z, \mathcal{D}(\mathcal{C}))$-martingale 
problem.  Since the limit is unique, this concludes the proof of (2).

\end{proof}

Since $\nu^M$ (periodically extended on $\R$)
weakly converges to $\nu$ as $M\to\infty$ on $L_r^2(\R), r>1$
(note $E^\nu[e^{2B(x)}] = e^{2|x|}, x\in \R$), this 
proposition shows that $\nu$ is invariant for the tilt variable of the
logarithm of the solution $Z(t,\cdot)$ of the SPDE \eqref{eq:8-b}.

Let $\bar Z(t)$ be the solution of  the SPDE \eqref{eq:8-b} and set
$$
Z(t) := e^{-\tfrac1{24}t} \bar Z(t).
$$
Then,  one can easily see that $Z(t)$ is a solution of \eqref{eq:1.1}
and $Z(t) \sim \bar Z(t)$ by the definition.
Since the tilt variable of $Z^\e(t)$ has $\tilde\mu^\e$ as 
invariant probability measure, we see in the limit the tilt variable
of $\bar Z(t)$ has $\tilde\mu$ as invariant probability measure. 
Since $Z(t) \sim \bar Z(t)$, this concludes that the tilt variable of 
$Z(t)$ also has $\tilde\mu$ as invariant probability measure.
In other words, we obtain the invariance of the
distribution of the geometric Brownian motion for the tilt process
determined by the stochastic heat equation \eqref{eq:1.1}:

\begin{prop} \label{prop:3.7}
For any bounded and continuous function $G=G(Z)$ on 
$\tilde{\mathcal{C}}_+$ and for any $\e>0$, $t\ge 0$, we have that
\begin{equation*}
\int_{\tilde{\mathcal{C}}_+} G(Z(t))d\tilde\mu
= \int_{\tilde{\mathcal{C}}_+} G(Z(0))d\tilde\mu,
\end{equation*}
where $\tilde\mu$ is a probability distribution on $\tilde{\mathcal{C}}_+$ of 
$Z(\cdot)=e^{B(\cdot)}$ with $B(\cdot)\in \tilde{\mathcal{C}}$ distributed under $\nu$.
\end{prop}

In order to deduce Theorem \ref{thm:1.1} from Proposition
\ref{prop:3.7}, one needs to now recover the height at the origin $h(t,0)$, which is defined
by $h(t,x) = \log Z(t,x)$.  However, since this does not really
work as in Section \ref{section:2.7},
we consider its smooth approximation.
Let $\eta^\e$ be the function introduced previously, and define a function
$h^\e(x,Z) := \log (Z*\eta^\e(x))$ of $x\in \R$ and $Z\in \mathcal{C}_+$.
Let $Z(t)$ be the solution of the SPDE \eqref{eq:1.1}.  Then, by It\^o's 
formula, the approximated height $h^\e(t,x) := h^\e(x,Z(t))$ satisfies
\begin{equation}\label{smoothed1}
h^\e(t,x) = h^\e(0,x) + \int_0^t b^\e(x,Z(s))ds + \int_0^t \int_\R
\si^\e(x,y,Z(s)) W(dsdy),
\end{equation}
where
\begin{align*}
& b^\e(x,Z) = \frac12 \left\{ \partial_x^2 h^\e(x,Z)
+ (\partial_x h^\e(x,Z))^2
- \frac{(Z^2*(\eta^\e)^2)(x)}{(Z*\eta^\e(x))^2} \right\}, \\
& \si^\e(x,y,Z) = \frac{\eta^\e(x-y)Z(y)}{Z*\eta^\e(x)}.
\end{align*}
The key point is that, as functions of $Z$, both $b^\e$ and $\si^\e$ are
defined on the quotient space $\tilde{\mathcal{C}}_+$.  Therefore, once 
$Z(t)\in\tilde{\mathcal{C}}_+$ is determined by solving the SPDE 
\eqref{eq:1.1}, we can recover its height as
\begin{equation}\label{smoothed2-b}
h^\e(t,0) = h^\e(0,0) + X_t,
\end{equation}
where  $X_t$ is the sum of the second and third terms in the right hand side
of \eqref{smoothed1} with $x=0$, which depends only on $Z(\cdot)$.

\begin{lem}
For any bounded, integrable and continuous function $G=G(h,Z)$ on 
$\mathcal{C}_+ \equiv \R\times\tilde{\mathcal{C}}_+$ and for any $\e>0$, $t\ge 0$, we have that
\begin{equation}\label{smoothed3-b}
\int_{\R\times\tilde{\mathcal{C}}_+} G(h^\e(t,0),Z(t))dh_0d\tilde\mu
= \int_{\R\times\tilde{\mathcal{C}}_+} G(h^\e(0,0),Z(0))dh_0d\tilde\mu.
\end{equation}
\end{lem}

\begin{proof}
From \eqref{smoothed2-b} and the translation-invariance of the Lebesgue measure,
the left hand side of \eqref{smoothed3-b} is equal to
$$
\int_{\R\times\tilde{\mathcal{C}}_+} G(h^\e(0,0)+X_t,Z(t))dh_0d\tilde\mu
= \int_{\R\times\tilde{\mathcal{C}}_+} G(h^\e(0,0),Z(t))dh_0d\tilde\mu,
$$
by performing the integral in $dh_0$ first.
But, this is equal to the right hand side of \eqref{smoothed3-b}
by the invariance of $\mu$ under $Z(t)\in\tilde{\mathcal{C}}_+$;
see Proposition \ref{prop:3.7}.
\end{proof}

Letting $\e$ tend to zero in \eqref{smoothed3-b}, we obtain
\eqref{smoothed3-b} also for 
$(h(t,0),Z(t))$.  This implies the invariance of the product measure
$dh_0d\tilde\mu$ for this joint process with the state space
$\R\times\tilde{\mathcal{C}}_+ $.
However, since the image measure of $dh_0d\tilde\mu$ under the
map $(h_0,Z) \in \R\times\tilde{\mathcal{C}}_+ \mapsto e^{h_0}Z \in
\mathcal{C}_+$ is nothing but $\mu$, the proof of Theorem \ref{thm:1.1} is completed.

\vskip 5mm
\noindent
{\bf Acknowledgement.} $\,$ 
The authors thank Makiko Sasada for pointing out
a simple proof of \eqref{eq:1/3}.

\end{document}